\documentclass[runningheads,a4paper]{llncs}

\usepackage{lipsum}
\usepackage{amsfonts,amsmath,mathrsfs,dsfont,bbold,latexsym,enumerate,empheq}
\usepackage{graphicx}
\usepackage{epstopdf,pdfpages}
\usepackage{algorithm,dsfont,algcompatible}
\ifpdf
  \DeclareGraphicsExtensions{.eps,.pdf,.png,.jpg}
\else
  \DeclareGraphicsExtensions{.eps}
\fi

\usepackage{enumitem}
\setlist[enumerate]{leftmargin=.5in}
\setlist[itemize]{leftmargin=.5in}

\usepackage[normalem]{ ulem }
\usepackage{soul,cancel}
\usepackage{blkarray,url,xcolor}

\usepackage{amsopn}

\urldef{\mailsc}\path|{nd448,cbs31}@cam.ac.uk| 
\urldef{\mailsb}\path|j.aston@statslab.cam.ac.uk| 
\urldef{\mailsa}\path|fabien.bonardi@univ-evry.fr| 
\urldef{\mailsd}\path|{alistair.forbes,marina.romanchikova}@npl.co.uk|
\urldef{\mailse}\path|carole.le-guyader@insa-rouen.fr| 
\begin{document}

\title{A Variational Model Dedicated to Joint Segmentation, Registration and Atlas Generation for Shape Analysis}
\titlerunning{A Variational Multitasking Framework for Shape Analysis}
\author{No\'emie Debroux, John Aston, Fabien Bonardi, Alistair Forbes, Carole Le Guyader, Marina Romanchikova \and Carola Sch\"onlieb}
\authorrunning{N. Debroux, J. Aston, F. Bonardi, A. Forbes, C. Le Guyader, M. Romanchikova, and C.-B. Sch\"onlieb}
\institute{\scriptsize Department of Applied Mathematics and Theoretical Physics (DAMPT), Centre for Mathematical Sciences, University of Cambridge, Wilberforce Road, Cambridge CB3 OWA, UK\\ \mailsc \\ Statslab, Department of Pure Mathematics and Mathematical Statistics, Centre for Mathematical Sciences, University of Cambridge, Wilberforce Road, Cambridge CB3 0WA, UK \\ \mailsb \\ IBISC, Universit\'e d'\'Evry, 36, Rue du Pelvoux, CE1455 Courcouronnes, 91020 Evry C\'edex, France \\ \mailsa \\ National Physical Laboratory, Hampton Road, Teddington, Middlesex, TW11 0LW UK \\ \mailsd \\ Normandie Univ, Institut National des Sciences Appliqu\'ees de Rouen, Laboratory of Mathematics, 76000 Rouen, France\\ \mailse }

\maketitle

\begin{abstract}
  In medical image analysis, constructing an atlas, i.e. a mean representative of an ensemble of images, is a critical task for practitioners to estimate variability of shapes inside a population, and to characterise and understand how structural shape changes have an impact on health. This involves identifying significant shape constituents of a set of images, a process called segmentation, and mapping this group of images to an unknown mean image, a \textcolor{black}{task} called registration, making a statistical analysis of the image population possible. To achieve this goal, we propose treating these operations jointly to leverage their positive mutual influence, in a hyperelasticity setting, \textcolor{black}{by viewing the shapes to be matched as Ogden materials.} 
  The approach is complemented by novel hard constraints on the $L^\infty$ norm of both the Jacobian and its inverse, ensuring that the deformation is a bi-Lipschitz homeomorphism. Segmentation is based on \textcolor{black}{the} Potts model, which allows for a partition into more than two regions\textcolor{black}{, i.e. more than one shape}. The connection to the registration problem is ensured by the dissimilarity measure that aims to align the segmented \textcolor{black}{shapes}. A representation of the deformation field in a linear space equipped with a scalar product is then computed in order to perform a geometry-driven Principal Component Analysis (PCA) and to extract the main modes of variations inside the image population. Theoretical results emphasizing the mathematical soundness of the model are provided, among which existence of minimisers, analysis of a numerical method of resolution, asymptotic results and a PCA analysis, as well as numerical simulations demonstrating the ability of the modeling to produce an atlas exhibiting sharp edges, high contrast and a consistent shape. 
\end{abstract}

\begin{keywords}
  Segmentation, registration, nonlinear elasticity, Ogden materials, Potts model, atlas generation, asymptotic results, $D^m$-splines, geometric PCA
\end{keywords}


\section{Introduction}\label{sec:introduction}
In recent years, joint image processing models have experienced increasing attention, including combined segmentation/registration models \cite{debroux-bib:droske_rumpf_2007,debroux-bib:han} (joint phase field approximation and registration), \cite{debroux-bib:lord} (model based on metric structure comparison),  \cite{debroux-bib:debroux_ozere,debroux-bib:swierczynski} (level set formulation that merges the piecewise constant Mumford-Shah model with registration principles), \cite{debroux-bib:gooya} (grounded in the expectation maximisation algorithm), \cite{debroux-bib:debroux_le_guyader_SIIMS} (based on a nonlocal characterisation of weighted-total variation and nonlocal shape descriptors), or \cite{debroux-bib:An2005,debroux-bib:leguyader1,debroux-bib:ozere_gout_le_guyader,debroux-bib:rumpf_wirth,debroux-bib:vemuri,debroux-bib:wyatt}\textcolor{black}{;}  joint image reconstruction and motion estimation \cite{debroux-bib:blume,debroux-bib:burger_dirks,debroux-bib:chun,debroux-bib:odille,debroux-bib:schumacher,debroux-bib:tomasi,debroux-bib:burger_2017,debroux-bib:lucka,debroux-bib:aviles}\textcolor{black}{;} joint reconstruction and registration for post-acquisition motion correction \cite{debroux-bib:corona} with the goal to reconstruct a single motion-free corrected image and retrieve the physiological dynamics through the deformation maps, joint optical flow estimation with phase field segmentation of the flow field \cite{debroux-bib:brune}, or joint segmentation/optimal transport models \cite{debroux-bib:essay70998} (to determine the velocity of blood flow in vascular structures).
This can be attributed to several factors:
\begin{itemize}
\item[(i)] the will to limit error propagation. Indeed, addressing the considered tasks in a unified joint framework (or multitasking) and exploiting, thus the strong correlation between them reduces the propagation of uncertainty, contrary to a sequential treatment that may amplify errors from step to step; 
\item[(ii)] Second, \textemdash and this is a corollary of the previous point \textemdash, performing simultaneously these tasks yields positive mutual influence and benefit on the obtained results \textcolor{black}{as shown in Figure \ref{figure:intro}}. To exemplify this observation, we can think first of joint models for image reconstruction and registration: not only does the approach correct the misalignment problem, but it also allows for alleviating ghosting artefacts.\newline
In the case of joint segmentation/registration models \textemdash the case that will be addressed more thoroughly afterwards \textemdash, as salient component pairing, shape/geometrical feature matching and intensity distribution comparison drive registration, processing these tasks simultaneously in a single framework may in particular reduce the influence of noise since the mapping can be done through the pairing of significant structures, e.g., by transferring the edges, and not only \textcolor{black}{through} intensity correlation.\newline
Besides, the registration operation can be viewed as the inclusion of priors to guide the segmentation process, in particular, for the questions of topology-preservation (the unknown deformation is substituted for the classical evolving curve of the segmentation process \textemdash \cite{debroux-bib:kass,le_guyader_vese_snakes_2008}\cite[Chapter 9]{debroux-bib:vese2015variational} for instance \textemdash and the related Jacobian determinant is subject to positivity  constraints) and geometric priors (since the registration allows to overcome the issue raised by weak boundary definition due to noise sources in the acquisition device, to degradation of the image contents during reconstruction, etc., by restoring them). In return, relevant segmented structures contribute to fostering accurate registration, providing then a reliable estimation of the deformation between the encoded structures, not only based on intensity matching (which takes the form of a local criterion), but also on geometrical/shape pairing (which has a nonlocal character).
\begin{figure}[H]
    \centering
    \includegraphics[width=\linewidth]{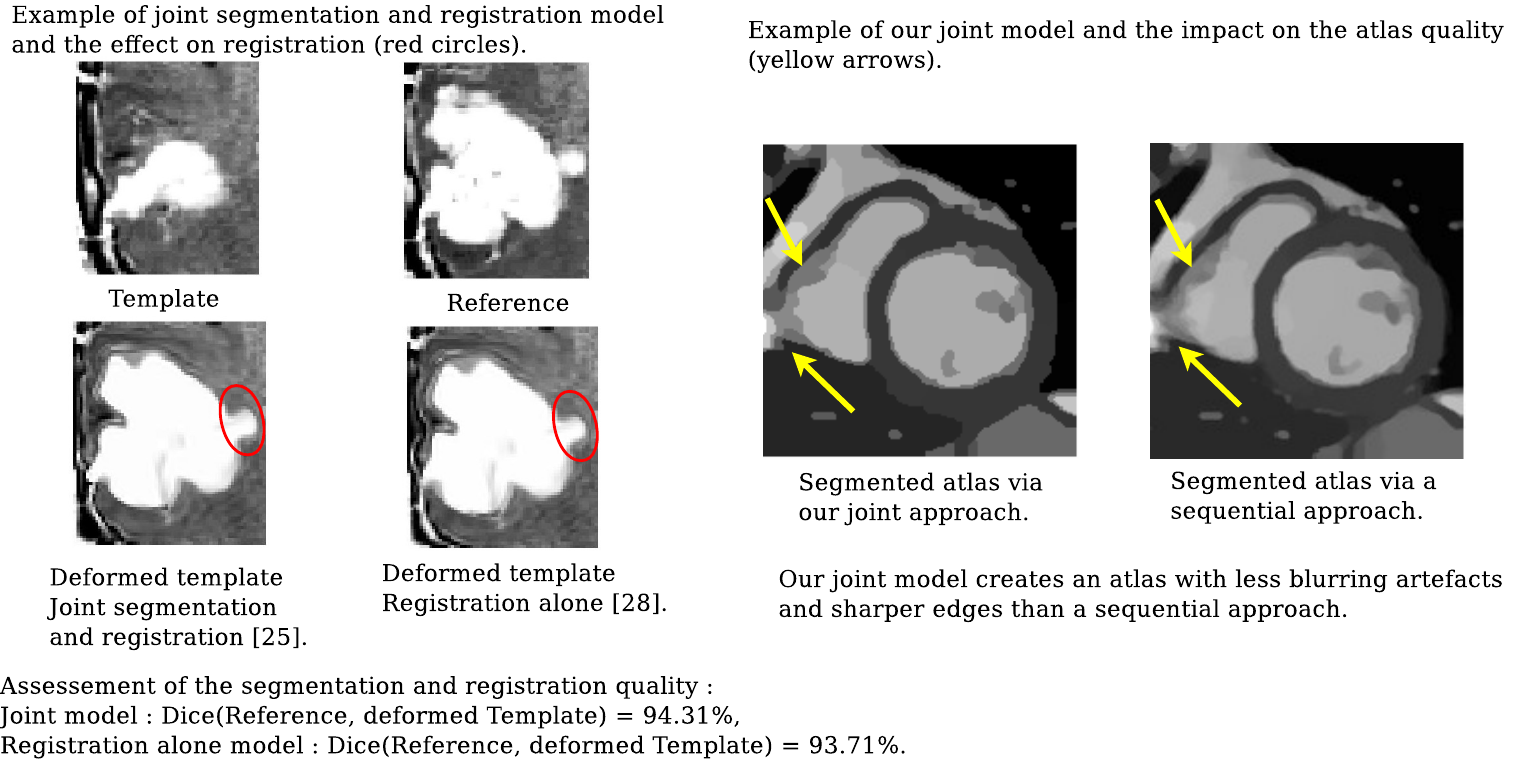}
    \caption{Illustration of the positive impact of joint approaches against sequential ones.}
    \label{figure:intro}
\end{figure}
\item[(iii)] Lastly, the pooling of the various results produced by the joint model allows for accurate post-processing treatments based on mutual analysis: for instance, the representation of the true underlying anatomy of an organ from a set of multiple acquisitions corrupted by motion, when tackling simultaneously reconstruction and registration, or the generation of an atlas in the context of joint segmentation and registration. The term atlas refers to a specific model for an ensemble of images and serves as a benchmark, i.e. a meaningful statistical image, to account for the variability (e.g., different shapes and sizes for organs in medical imaging) that might be observed in a population of images. 
\end{itemize}
The proposed work adopts this joint model philosophy. It aims at addressing the issue of designing a unified variational model for joint segmentation, registration and atlas generation by exploiting the strong correlation between the two former tasks thus reducing error propagation, in the medical imaging setting. The latter one requires the mapping of a group of images to a {\it{mean}} representative, which is an additional unknown of the problem, the subsequent goal being to extract a relevant hidden structure from this ensemble of images. As in medical images the variability between individuals is significant, constructing a meaningful statistical image of the global underlying anatomy of an organ from a set of images to measure this variableness is of great interest. It allows for the derivation of image statistics, the retrieval of the inherent dynamics of a single individual's organ, the estimation of the probability that a particular spatial location takes on a certain label, the detection and quantisation of abnormalities, that is, more generally, it allows to characterise and understand how geometrical and structural changes influence health. A large body of papers feeds the field of atlas generation and shape statistics among which \cite{debroux-bib:joshi} (atlas generation problem phrased in the Large Deformation Diffeomorphic Metric Mapping (LDDMM) framework \cite{debroux-bib:beg}), \cite{debroux-bib:papadakis2} (the shapes to be analysed are modeled as random histograms and in order to learn principal modes of variation from such data, the Wasserstein distance between probability measures is introduced), \cite{debroux-bib:zhang} (dedicated to elastic shape analysis; a unified registration/parameterised object statistical analysis framework is tackled, based on square-root transformations and able to process data as diverse as curves, functions, surfaces and images), \cite{debroux-bib:arsigny} (statistics performed on the space of diffeomorphisms), \cite{debroux-bib:hong} (the use of a kernel descriptor that characterises local shape properties ensures geometrically meaningful correspondence between shapes with statistical studies of the deformations), \cite{debroux-bib:rumpf_wirth,debroux-bib:Rumpf2011} (the shapes are viewed as closed contours approximated by phase fields, and shape averaging and covariance analysis are carried out in a nonlinear elasticity setting),  to name a few.\newline

The difficulty in designing the model arises from the complexity of the formulation that is generally underconstrained, involves nonlinearity and non-convexity, and is dictated by the given application. While segmentation attempts to reproduce the ability of human beings to track down significant patterns and automatically gather them into significant constituents (see \cite[Chapter 4]{debroux-bib:aubert} or \cite[Part II]{debroux-bib:vese2015variational} for a relevant analysis of this problem), it remains a challenging and ill-posed task (as emphasised by Zhu et al. (\cite{debroux-bib:zhu1})) since the definition of an object encompasses various acceptations: it can be something material \textemdash a thing \textemdash or a periodic pattern, this heterogeneity entailing the design of suitable methodologies for each specific application. Similarly, for the registration assignment (see \cite{debroux-bib:modersitzki1,debroux-bib:modersitzki2,debroux-bib:sotiras} for the registration counterpart with Matlab software), the sought deformation is usually viewed as a minimal argument (uniqueness defaults in general) of a specifically tailored cost function that has a polymorphous character in nature. For images acquired on different devices and depicting various physical phenomena, the quality of registration is not measured by intensity distribution alignment, but by the degree of shape/geometrical feature pairing. Also, several stances can be adopted to describe the setting the objects to be matched fall within (physical models \textemdash \cite{debroux-bib:beg}, \cite{debroux-bib:broit}, \cite{debroux-bib:burger}, \cite{debroux-bib:thesis_christensen}, \cite{debroux-bib:davatzikos_elasticity}, \cite{debroux-bib:derfoul}, \cite{debroux-bib:droske}, \cite{debroux-bib:fischer2}, \cite{debroux-bib:leguyader1}, \cite{debroux-bib:ozere_gout_le_guyader}, \cite{debroux-bib:rumpf_wirth} \textemdash, purely geometric ones \textemdash \cite{debroux-bib:ashburner}, \cite{debroux-bib:davis}, \cite{debroux-bib:sederberg}, \cite{debroux-bib:zagorchev}
\textemdash, models including a priori knowledge (\cite{debroux-bib:clatz}), depending on the assumption regarding the
properties of the deformation to be recovered) and to devise the measure of alignment (that is, how the available data are exploited to drive the registration process), increasing thus the complexity of the problem.
To meet these criteria, we devise, in a variational framework, a theoretically well-motivated and physically relevant combined model, capable of handling large deformations, reliable in terms of pairing of the \textcolor{black}{shapes} encoded in the images, and efficient in extracting a relevant underlying structure \textcolor{black}{ decomposed into shapes} from the considered set of images. Statistical shape analysis is then performed by means of a Principal Component Analysis (PCA) on the obtained deformations to retrieve the main modes of variations inside the dataset, after finding a suitable representative of the deformation in a linear space (i.e. in order that the recovered deformation lives in a vector space). \newline

The results are obtained through the use of the hyperelasticity setting and the design of an original geometric dissimilarity measure ensuring alignment of the {\textcolor{black}{(possibly nested)}} shapes for the combined model \textemdash {\textcolor{black}{thus favouring the matching of shapes rather than the coupling of grey levels with the underlying goal to potentially process images with different modalities}}\textemdash, and the introduction of a tensor-based approximation problem for the statistical analysis. Unlike \cite{debroux-bib:rumpf_wirth}, the shapes to be matched are not modeled by their closed contour but through a piecewise-constant partition (Potts model \cite{debroux-bib:MumfShah89,debroux-bib:potts}), which constitutes the main difference with \cite{debroux-bib:rumpf_wirth}. {\textcolor{black}{Not only does the shape pairing rely on the object outer envelope matching, but also on the internal structure alignment}}. This way of looking at shapes entails substantial modifications in the design of the functional to be minimised and in the search for an appropriate representative of the deformation in a vector space.\newline

More precisely, the novelty of the paper rests upon: (i) an original modeling involving the stored energy function of an Ogden material complemented by two new hard constraints on the Jacobian and its inverse (in addition to the theoretical utility of these constraints, it also allows to control changes of length), \textcolor{black}{the} Potts model for segmentation, and an original discrepancy measure ensuring edge mapping; (ii) the derivation of theoretical results encompassing non straightforward mathematical tools; (iii) the analysis and comparison of three different methods to perform statistical analysis on the obtained deformation: the first one, based on linearised elasticity principles largely inspired by \cite{debroux-bib:Rumpf2011}, the second one using the Cauchy-stress tensors motivated by \cite{debroux-bib:Rumpf2011}, and the last one, more novel \textcolor{black}{and on which the paper focuses}, relying on tensor-based smoothing D$^m$ splines, influenced by \cite{debroux-bib:le-guyader-apprato}.
Section \ref{sec:mathematical_modeling} is devoted to the analysis of the mathematical model including a theorem of existence of minimisers,  while Section \ref{sec:numerical_method_resolution} is dedicated to the theoretical analysis of a numerical method of resolution based on a splitting approach and {\textcolor{black}{implying}} \textcolor{black}{the Alternating Direction Method of Multipliers (ADMM) techniques and proximal gradient descent algorithms}. Section \ref{sec:representation_deformation_linear_space_PCA} deals with the resulting geometry-driven statistical analysis, which requires finding a fitting representation of the obtained deformation in a linear space before performing a PCA. As already mentioned, our motivation is to investigate how the linear elasticity based approach and the Cauchy-stress based method compare to the D$^m$ splines approximation based procedure. The first two are non-straightforward adaptations of the techniques envisioned in \cite{debroux-bib:Rumpf2011}, but the view we take to model the shapes \textemdash phase field rather than closed contours \textemdash implies substantial changes in the physics of the problem, while the emphasis is put on the last one for its novelty. 
Section \ref{sec:numerical_simulations} focuses on numerical simulations with a thorough comparison between sequential treatments and the proposed joint model, demonstrating the ability of our model to handle large deformations and to produce in the end, an atlas with sharp edges, high contrast and reflecting a realistic shape.\newline
Let us emphasise that the focus of the paper is on the mathematical presentation of a nonlinear elasticity-based unified segmentation\textcolor{black}{, }registration\textcolor{black}{, and }atlas generation model. Hence, the computational results are currently still restricted to two dimensions {\textcolor{black}{due in practice to the applications that were presented to us by clinicians}}. However, as will be seen next, the proposed algorithm can be easily adapted to the three-dimensional case.
\section{Mathematical Modeling}\label{sec:mathematical_modeling}
\subsection{Depiction of the Model}\label{sub-sec:depiction_model}
\label{debroux-le-guyader-sec:depiction_model}
Let $\Omega$ be a connected bounded open subset of $\mathbb{R}^3$ with boundary sufficiently smooth (convenient way of saying that in a given definition, the smoothness of the boundary is such that all arguments make sense {\textcolor{black}{and enabling us to use compact Sobolev embeddings among others}}). Let us denote by $T_i: \bar{\Omega}\rightarrow \mathbb{R}$ the $i$-th template image with $i=1,\cdots,M$ {\textcolor{black}{\textemdash available data in our problem \textemdash}}, $M$ being the total number of initial images. For theoretical and numerical purposes, we assume that each $T_i$ is compactly supported on $\Omega$ to ensure that $T_i\circ\varphi_i$ is always defined and we assume that $T_i$ is Lipschitz continuous. It can thus be considered as an element of the Sobolev space $W^{1,\infty}(\mathbb{R}^3)$, and the chain rule applies. {\textcolor{black}{The partitioning of each template $T_i$ into regions with homogeneous intensities\textcolor{black}{, defining shapes,}} is encoded in the variable $\theta_i : \bar{\Omega} \rightarrow \mathbb{R}$ {\textcolor{black}{\textemdash the variables $\left\{\theta_i\right\}_{i=1}^{M}$ belonging to the set of unknowns of the problem and being read as simplified versions of the images $T_i$ that encompass the geometrical \textcolor{black}{shapes} \textemdash}}, and $\theta_R : \bar{\Omega}\rightarrow \mathbb{R}$ is the unknown segmented atlas generated by our model. {\textcolor{black}{As will be seen later, these variables allow making the connection between segmentation and registration. Also, using these schematic versions of the images tends to favour shape pairing rather than grey level mapping}}. Let $\varphi_i : \bar{\Omega}\rightarrow \mathbb{R}^3$ be the sought deformation between $\theta_{T_i}$ and the unknown mean segmentation $\theta_R$. Of course, in practice, $\varphi_i$ should be with values in $\bar{\Omega}$ but from a mathematical point of view, if we work with such spaces of functions, we lose the structure of vector space. Nonetheless, we can show that our model retrieves deformations with values  in $\bar{\Omega}$ \textcolor{black}{\textemdash}  based on Ball's results~\cite{debroux-bib:ball}. A deformation is a smooth mapping that is orientation-preserving and injective, except possibly on $\partial \Omega$, if self-contact is allowed. The deformation gradient is $\nabla \varphi_i:\, \bar{\Omega} \rightarrow M_3(\mathbb{R})$, the set $M_3(\mathbb{R})$ being the set of real square matrices of order $3$. The associated displacement field is denoted by $u_i$ such that $\varphi_i=\mathrm{Id}+u_i$, and $\nabla \varphi_i=I_3+\nabla u_i$, with $\mathrm{Id}$, the identity mapping, and $I_3$, the $3\times3$ identity matrix. We also need the following notations: $A:B={\mathrm{tr}}A^TB$, the matrix inner product and $||A||={\sqrt{A:A}}$, the related matrix norm (Frobenius norm).\newline

Following the joint model philosophy in a variational framework, the sought deformations $\varphi_i$, the segmentations $\theta_{T_i}$, for all $i=1,\cdots,M$, as well as the segmented atlas $\theta_R$ are seen as {\textcolor{black}{minimal arguments}} of a specifically designed cost function. It comprises a regularisation on $\varphi_i$, for all $i=1,\cdots,M$, prescribing the nature of the deformations, a {\textcolor{black}{penalisation}} on $\theta_{T_i}$, for all $i=1,\cdots,M$, and $\theta_R$, favouring piecewise constant segmentations, a segmentation fidelity term ensuring the closeness of the $\theta_{T_i}$ to the initial available image $T_i$, and a data-driven term measuring \textcolor{black}{the} alignment between the deformed segmentation{\textcolor{black}{s}} {\textcolor{black}{$\left\{\theta_{T_i}\circ \varphi_i\right\}_{i=1}^{M}$ and $\theta_R$}}, intertwining then segmentation and registration.\newline

In this work, we {\textcolor{black}{view}} all the {\textcolor{black}{template}} images and their respective {\textcolor{black}{partitioning}} as deformed versions of a single image/segmentation. Inspired by the observation in \cite{debroux-bib:rumpf_wirth}: {\textit{"the arithmetic mean $x$ of observations \textcolor{black}{$\{x_i\}_{i=1}^{M}$} can be interpreted as the minimiser of the total elastic deformation energy in a system where the average $x$ is connected to each $x_i$ by an elastic spring, under the Hooke's law"}}, a natural choice for
the definition of the mean segmentation is given by the particular deformed
configuration that minimises the total nonlinear hyperelastic deformation energy
required to align each segmentation to this mean configuration. To allow large deformations, the shapes to be matched are viewed as isotropic (exhibiting the same mechanical properties in every direction), homogeneous (showing the same behaviour everywhere inside the material), and hyperelastic (with a stress-strain relation derived from a strain energy density) materials, and more precisely as Ogden ones (\cite{debroux-bib:ciarlet}). Note that rubber, filled elastomers, and biological tissues are often modeled within the hyperelastic framework, which motivates our modeling. This perspective drives the design of the regularisation on the deformations $\varphi_i$ which is thus based on the stored energy function of an Ogden material, prescribing then a physically-meaningful nature.\newline

We recall that the general expression for the stored energy function of an Ogden material (see \cite{debroux-bib:ciarlet}\cite{debroux-bib:ledret}) is given by 
\begin{align*}
    W_O(F)=\underset{i=1}{\overset{K_1}{\sum}} a_i\|F\|^{\gamma_i} + \underset{j=1}{\overset{K_2}{\sum}}b_i\|\mathrm{Cof}F\|^{\beta_j}+\Gamma(\mathrm{det}F),
\end{align*}
with $a_i>0$, $b_i>0$, $\gamma_j\geq 0$, $\beta_j\geq 0$, for all $i=1,\cdots,K_1$ and all $j=1,\cdots,K_2$, and $\Gamma\, :\, ]0,+\infty[ \rightarrow \mathbb{R} $ being a convex function satisfying $\underset{\delta \rightarrow 0^+}{\lim}\Gamma(\delta)=\underset{\delta \rightarrow +\infty}{\lim} \Gamma(\delta)=+\infty$. The first term penalises changes in length, the second one controls the changes in area {\textcolor{black}{while}} the third one restricts the changes in volume. The latter also ensures preservation of topology by imposing  {\textcolor{black}{positivity}} of the Jacobian determinant {\textcolor{black}{almost everywhere}}. In this work, we focus on the following particular energy: 
\begin{align*}
    &W_{Op}(F) = \\
    &\left\{ \begin{array}{l} a_1\|F\|^4+a_2\|\mathrm{Cof}F\|^4+a_3(\mathrm{det}F-1)^2+\frac{a_4}{{\textcolor{black}{\left(\mathrm{det}F\right)^{10}}}}-3(a_1+a_2)-a_4 \text{ if } \mathrm{det}F>0,\\ +\infty \text{ otherwise,}\end{array} \right.
\end{align*}
fulfilling the previous assumptions. The third and fourth terms govern the distribution of the Jacobian determinant~: the latter prevents singularities and large contractions by penalising small values of the determinant, while the former promotes values of the determinant close to $1$ avoiding {\textcolor{black}{thus}} expansions and contractions that are too large. 
The choice of the remaining terms is motivated by the theoretical results in \cite{debroux-bib:ball} to ensure that the deformations are homeomorphisms. The constants are added to fulfill the energy property $W_{Op}(I_3)=0$. In order to avoid singularity as much as possible, to get deformations that are bi-Lipschitz homeomorphisms, and to obtain Cauchy-stress tensors (whose formal definition will be given in Section \ref{sec:representation_deformation_linear_space_PCA}) in the linear space $L^2(\Omega,{\textcolor{black}{M_{3}}}(\mathbb{R}))$, we complement this stored energy function $W_{Op}$ by the term $\mathds{1}_{\{\|.\|_{L^\infty(\Omega,M_3(\mathbb{R}))}\leq \alpha\}}(F)+\mathds{1}_{\{\|.\|_{L^\infty(\Omega,M_3(\mathbb{R}))}\leq \beta\}}(F^{-1})$, with $\alpha\geq 1$, and $\beta \geq 1$, where $\mathds{1}_A$ denotes the convex characteristic function of a convex set $A$. Therefore, the regularisation can be written as
\begin{align*}
    W(F) = \int_\Omega W_{Op}(F)\,dx + \mathds{1}_{\{\|.\|_{\{L^\infty(\Omega,M_3(\mathbb{R}))}\leq \alpha\}}(F)+\mathds{1}_{\{\|.\|_{L^\infty(\Omega,M_3(\mathbb{R}))}\leq \beta\}}(F^{-1}).
\end{align*}
{\textcolor{black}{\begin{remark}
In terms of functional spaces, if $\varphi \in W^{1,\infty}(\Omega,\mathbb{R}^3)$ (suitable space owing to the $L^{\infty}$ hard constraints), $\mathrm{Cof}\,\nabla \varphi$ and $\mathrm{det}\,\nabla \varphi$ are automatically elements of $L^{\infty}(\Omega,M_{3}(\mathbb{R}))$ and $L^{\infty}(\Omega)$ respectively, since $L^{\infty}(\Omega,\mathbb{R}^3)$ has a structure of commutative Banach algebra. Penalising the $L^{\infty}$ norm of $\nabla \varphi$ thus entails control over the Jacobian determinant. \textcolor{black}{This additional term implicitly gives an upper and lower bound on the Jabobian determinant ensuring thus topology preservation.} 
\end{remark}}}
The aforementioned regulariser is then applied along {\textcolor{black}{with}} a discrepancy measure, which allows {\textcolor{black}{intertwining}} the segmentation and registration tasks, and a segmentation part comprising a fidelity term and a sparsity measure on the {\textcolor{black}{paired}} edges based on the Potts model (\cite{debroux-bib:potts}). The latter, also known as piecewise-constant Mumford-Shah model \cite{debroux-bib:MumfShah89} with $N\in \mathbb{N}$ phases\textcolor{black}{/shapes} {\textcolor{black}{($N$ is thus a prior)}}, is written, for an observed image $f$, as
\begin{align*}
    \underset{u \in \mathcal{U}}{\inf} E_{Potts}(u)=\underset{ l=1}{\overset{N}{\sum}} \frac{\alpha}{2} TV(u_l) + \int_\Omega \underset{l=1}{\overset{N}{\sum}} u_l(c_l-f)^2\,dx,
\end{align*}
with 
$\mathcal{U} =\{{\textcolor{black}{u=(u_l)_{l=1,\cdots,N}\in \left(BV(\Omega,\{0,1\})\right)^N,\, \underset{l=1}{\overset{N}{\sum}} u_l=1,\, a.e.\, on\, \Omega}}\},\,{\mbox{and}}\,\newline c_l = \left\{ \begin{array}{cc} \frac{\int_\Omega fu_l\,dx}{\int_\Omega u_l\,dx} &\text{ if } \, \int_\Omega u_l\,dx \neq 0, \\ 0 &\text{ otherwise,}\end{array}\right.$, $\alpha>0$ {\textcolor{black}{being}} a weighting parameter {\textcolor{black}{balancing}} the fidelity term and the regularisation. The notation $TV$ denotes the classical Total Variation, measuring the perimeter length of the set defined by $\{x\in \Omega \, |\, u_l(x)=1\}$ thanks to the coarea formula (\cite{debroux-bib:demengel,debroux-bib:evans}). The segmentation/{\textcolor{black}{partitioning}} is then retrieved by $\tilde u = \underset{l=1}{\overset{N}{\sum}} u_lc_l$, {\textcolor{black}{$\tilde u$}} being a decomposition of the initial image $f$ into $N$ \textcolor{black}{shapes} defined by the characteristic function{\textcolor{black}{s}} $u_l$ with constant intensity value{\textcolor{black}{s}} $c_l$, each one corresponding to an object of interest under the assumption that it is defined by a homogeneous region with close intensity values. 
\begin{remark}\label{debroux:remark1}Extensions to homogeneous regions in terms of texture with a piecewise-smooth approximation instead of a piecewise-constant approximation (see \cite{debroux-bib:vese2015variational}) or in terms of histograms (see \cite{debroux-bib:papadakis1}) are possible, depending on the nature of the considered images, but this is not the scope of this paper. 
\end{remark}
The characteristic functions $u_l$ give a good representation of the geometric features inside the images, and can be seen as nonlocal shape descriptors that will help the registration process. In that prospect, we introduce this novel geometric dissimilarity measure whose aim is to align the salient structures based on the previous decomposition without taking into account the intensity values {\textcolor{black}{\textemdash thus favouring shape pairing \textemdash }}: 
\begin{align*}
    E_{\text{dist}}((\theta_{T_i},\varphi_i)_{i=1,\cdots,M},\theta_R)=\frac{1}{2M}\underset{i=1}{\overset{M}{\sum}}\underset{l=1}{\overset{N}{\sum}}TV(\theta_{T_i,l}\circ \varphi_i-\theta_{R,l}),
\end{align*}
{\textcolor{black}{with notations consistent with the definition of $\mathcal{U}$, i.e. $\forall i \in \left\{1,\cdots,M\right\}$, $\theta_{T_i}=\left(\theta_{T_i,l}\right)_{l=1,\cdots,N} \in \left(BV(\Omega,\left\{0,1\right\})\right)^N$ and $\theta_{R}=\left(\theta_{R,l}\right)_{l=1,\cdots,N} \in \left(BV(\Omega,\left\{0,1\right\})\right)^N$}}.
{\textcolor{black}{
\begin{remark}\label{debroux:remark2}Consistently with Remark \ref{debroux:remark1}, we could also envision a model \textcolor{black}{including }
both the deformations $\varphi_i$ pairing the structures ({\it{i.e.}} viewed as global deformations) and additional components reflecting better the more local deformations. \textcolor{black}{This results} mathematically in a composition of deformations. Again, this is not the scope of the proposed work.
\end{remark}
}}
It thus allows for the registration of images acquired through different mechanisms and is more robust to small changes of intensities that can happen even for images of the same modality, especially in medical images. It measures the perimeter length of the misaligne\textcolor{black}{d} region for each structure of interest and thus drives the registration process by mapping the \textcolor{black}{shapes}.\newline

In the end, the overall problem denoted by \ref{initial_problem} is stated by
 \begin{align}
 \nonumber \inf\, \mathcal{F}_1(\theta_R,\textcolor{black}{\{\theta_{T_i},\varphi_i\}_{i=1}^{M}}) &= \frac{1}{M}\underset{i=1}{\overset{M}{\sum}} \bigg( \frac{\gamma_T}{2} \underset{l=1}{\overset{N}{\sum}}TV(\theta_{T_i,l}) + \int_\Omega \underset{l=1}{\overset{N}{\sum}}\theta_{T_i,l}(c_{T_i,l}-T_i)^2\,dx \\
\nonumber &+ \frac{\gamma_R}{2}\underset{l=1}{\overset{N}{\sum}}TV(\theta_{R,l}) 
 + \int_\Omega \underset{l=1}{\overset{N}{\sum}}\theta_{R,l}(c_{R,l}-T_i\circ \varphi_i)^2\,dx  \\
\nonumber &+ \frac{\lambda}{2} \underset{l=1}{\overset{N}{\sum}}TV(\theta_{T_i,l}\circ \varphi_i - \theta_{R,l})+ W(\nabla \varphi_i)\bigg),
\tag{P}
\label{initial_problem}
 \end{align}
with $c_{T_i,l}=\left\{ \begin{array}{l} \frac{\int_\Omega \theta_{T_i,l}(x)T_i(x)\,dx}{\int_\Omega \theta_{T_i,l}(x)\,dx}\text{ if } \int_\Omega \theta_{T_i,l}(x)\,dx \neq 0\\ 0 \text{ otherwise }\end{array} \right.$,\\ $c_{R,l}=\left\{ \begin{array}{l} \frac{1}{M}\underset{i=1}{\overset{M}{\sum}}\frac{\int_\Omega \theta_{R,l}(x)T_i\circ \varphi_i(x)\,dx}{\int_\Omega \theta_{R,l}(x)\,dx}\text{ if } \int_\Omega \theta_{R,l}(x)\,dx \neq 0\\ 0 \text{ otherwise }\end{array} \right.$, $\alpha\geq 1$ and $\beta \geq 1$.\newline
\textcolor{black}{An illustration of the overall components of the model as well as the pipeline of the resulting analysis is given in Figure \ref{figure:teaser}.}
\begin{figure}[H]
    \centering
    \includegraphics[width=\linewidth]{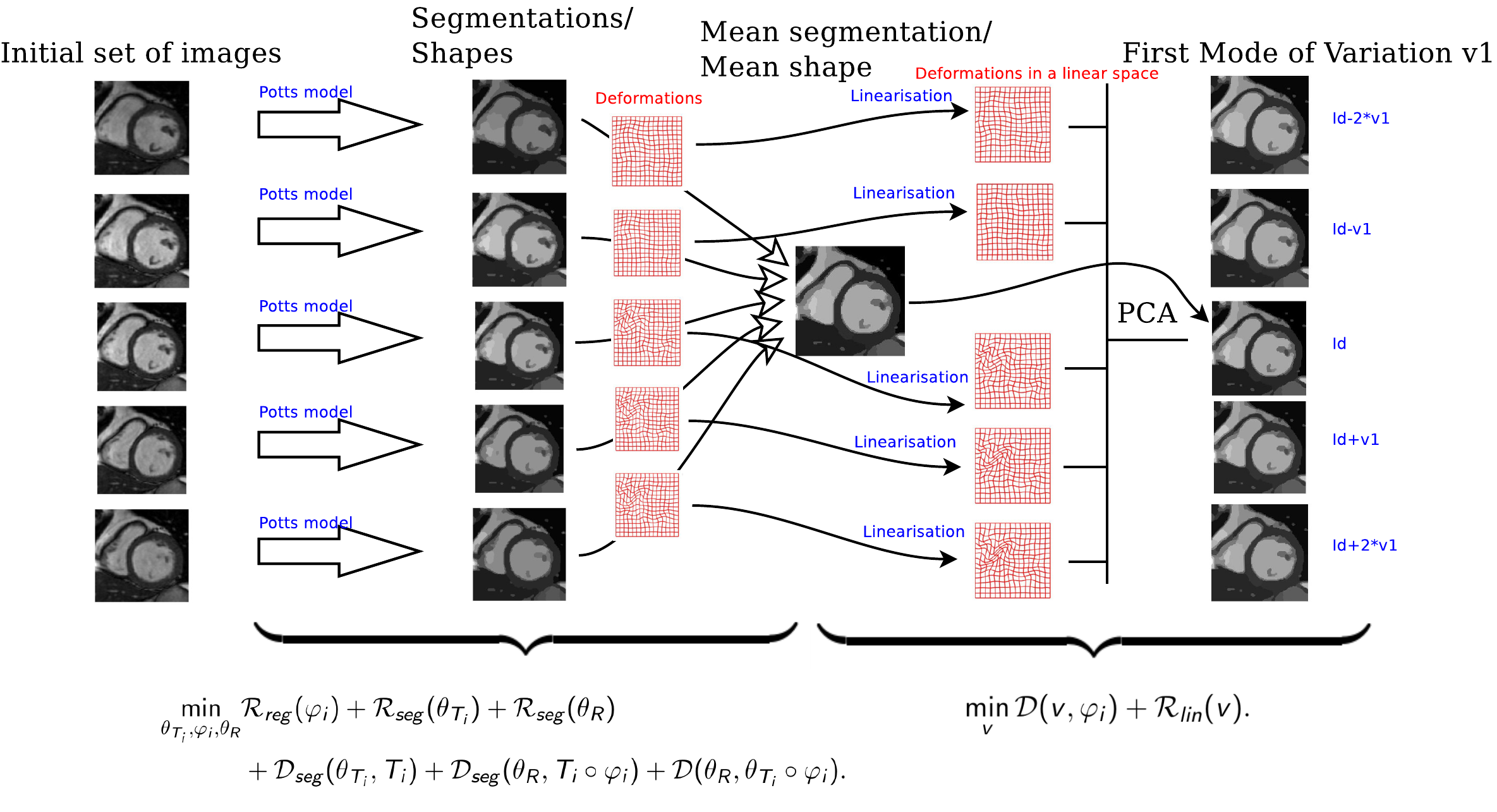}
    \caption{Overview of our framework}
    \label{figure:teaser}
\end{figure}
\subsection{Theoretical Results}\label{sub-sec:theoretical_results}
In this subsection, we theoretically analyse problem~\ref{initial_problem} by showing its well-definedness. In that purpose, we prove the existence of minimisers in the following theorem. 
\begin{theorem}[Existence of minimisers.]\newline 
\label{thm:existence_minimizers}
We introduce the functional space: 
\begin{itemize}
 \item $\hat{\mathcal{W}} = \{\psi \in \mathrm{Id}+W_0^{1,\infty}(\Omega,\mathbb{R}^3),\, \frac{1}{\mathrm{det} \nabla \psi}\in L^{10}(\Omega),\, \det {\textcolor{black}{\nabla}}\psi>0 \text{ a.e. in } \Omega,$\\$ \|\nabla \psi\|_{L^\infty(\Omega,M_3(\mathbb{R}))}$ $\leq \alpha, \|(\nabla \psi)^{-1}\|_{L^\infty(\Omega,M_3(\mathbb{R}))}\leq \beta\}$,
\end{itemize}
The infimum is searched for $\theta_R\in{\textcolor{black}{\mathcal{U}}}$, $\theta_{T_i}\in {\textcolor{black}{\mathcal{U}}}$, and $\varphi_i\in \hat{\mathcal{W}}$ for all $i\in \{1,\cdots,M\}$ such that $\theta_{T_i,l}\circ \varphi_i - \theta_{R,l} \in BV(\Omega)$ for all $l\in \{1,\cdots,N\}$ and for all $i\in \{1,\cdots,M\}$. 
There exists {\textcolor{black}{at least one minimiser}} to this problem.
\end{theorem}
\begin{proof}
The proof is based on the theory of the calculus of variations, and relies on Ball's results \cite{debroux-bib:ball} and arguments inspired by \cite{debroux-bib:benedikt_wirth}. See Section 1 of the supplementary material for the detailed proof.
\end{proof}
We now investigate an original numerical method for the resolution of problem \ref{initial_problem}.
\section{Numerical Method of Resolution}\label{sec:numerical_method_resolution}
\subsection{Description and Analysis of the Numerical Method}\label{sub-sec:description_analysis_numerical_method}
Inspired by a prior work by Negr\'on Marrero \cite{debroux-bib:negron} in which {\textcolor{black}{the author}} describes and analyses a numerical method detecting singular minimisers and avoiding the Lavrentiev phenomenon for 3D problems in nonlinear elasticity, we introduce auxiliary variables {\textcolor{black}{and split the original problem into sub-problems that are \textcolor{black}{computationally more tractable}}}. The idea of Marrero's work is to decouple the deformation $\varphi$ from its gradient $\nabla \varphi$ and to formulate a related decoupled problem under equality constraints, moving thus the nonlinearity in the Jacobian to this new variable. With this in mind, we introduce the following auxiliary variables: $V_i$ simulating the Jacobian of $\varphi_i$ for each $i$, $W_i$ simulating the inverse Jacobian $(\nabla \varphi_i)^{-1}$ for all $i$, {\textcolor{black}{and $\forall i \in \left\{1,\cdots,M\right\}$, $\forall l\in\left\{1,\cdots,N\right\}$,  $\theta_{\tilde T_i,l}=\theta_{T_i,l}\circ \varphi_i-\theta_{R,l}$}}, 
to simplify numerical computations, and derive a functional minimisation problem phrased in terms of {\textcolor{black}{$(\theta_{T_i},\,\varphi_i,\,V_i,\,W_i,\,\theta_{\tilde T_i})_{i=1,\cdots,M}$,\,$\theta_{R}$}}. However, we do not impose {\textcolor{black}{equality constraints}} as in \cite{debroux-bib:negron}, but integrate instead, $L^p$-type penalisations ($p=1$ or $p=2$; the choice for the $L^1$-penalisation will be discussed later) into the functional,  partially relaxing a constrained problem under both equality and inequality constraints by a problem under inequality constraints only. The decoupled problem is thus defined by means of the following functional: 
\begin{align}
  \nonumber \inf \Bigg\{ &\mathcal{F}_{1,\gamma}({\textcolor{black}{\{\varphi_i,\theta_{T_i},V_i,W_i\}_{i=1}^{M},(\theta_{\tilde T_i,l})_{i=1,\cdots,M \atop l=1,\cdots,N},\theta_R}})= \frac{1}{M}\underset{i=1}{\overset{M}{\sum}}\Big(\frac{\gamma_T}{2} \underset{l=1}{\overset{N}{\sum}} TV(\theta_{T_i,l}) \nonumber \\
  &+ \int_\Omega \underset{l=1}{\overset{N}{\sum}} \theta_{T_i,l}({\textcolor{black}{c_{T_i,l}}}-T_i)^2\,dx
  + \frac{\gamma_R}{2} \underset{l=1}{\overset{N}{\sum}} TV(\theta_{R,l}) + \int_\Omega \underset{l=1}{\overset{N}{\sum}} \theta_{R,l} ({\textcolor{black}{c_{R,l}}}-T_i\circ \varphi_i)^2\,dx\nonumber  \\
  \nonumber &+ \frac{\lambda}{2}\underset{l=1}{\overset{N}{\sum}} TV(\theta_{\tilde T_i,l}) +\gamma \int_\Omega \underset{l=1}{\overset{N}{\sum}} |\theta_{\tilde T_i,l} - (\theta_{T_i,l}\circ \varphi_i - \theta_{R,l})|\,dx+ \int_\Omega W_{Op}({\textcolor{black}{V_i}})\,dx\\
 \nonumber & + \frac{\gamma}{4}\|V_i-\nabla \varphi_i\|_{L^4(\Omega,M_3(\mathbb{R}))}^4+\mathds{1}_{\{\|.\|_{L^\infty(\Omega,M_3(\mathbb{R}))}\leq \alpha\}}(V_i)+\mathds{1}_{\{\|.\|_{L^\infty(\Omega,M_3(\mathbb{R}))}\leq \beta\}}(W_i)\\
  &+\frac{\gamma}{2}\|W_i-V_i^{-1}\|_{L^2(\Omega,M_3(\mathbb{R}))}^2\Big)\Bigg\}, \tag{DP} \label{decoupled_problem}
 \end{align}
with $\forall l \in \left\{1,\cdots,N\right\}$, $c_{{\textcolor{black}{T_i,l}}} = \left\{ \begin{array}{cc} \frac{\int_\Omega \theta_{T_i,l}(x)T_i(x)\,dx}{\int_\Omega \theta_{T_i,l}(x)\,dx}&\text{ if } \int_\Omega \theta_{T_i,l}(x)\,dx \neq 0 \\ 0 &\text{ otherwise} \end{array} \right.$,\\ $c_{{\textcolor{black}{R,l}}} = \left\{ \begin{array}{cc}\frac{1}{M}{\underset{i=1}{\overset{M}{\sum}}} \frac{\int_\Omega \theta_{R,l}(x) T_i\circ \varphi_i(x)\,dx }{\int_\Omega \theta_{R,l}(x)\,dx}&\text{ if } \int_\Omega \theta_{R,l}(x)\,dx \neq 0 \\ 0 &\text{ otherwise }\end{array}\right.$, $\alpha\geq1$ and $\beta\geq1$. We address this problem for $\varphi_i \in \mathrm{Id}+W^{1,4}_0(\Omega,\mathbb{R}^3)$, 
{\textcolor{black}{$V_i\in \{\xi \in L^{\infty}(\Omega,M_3(\mathbb{R}))\, |\, \mathrm{det}\xi >0 \text{ a.e. on }\Omega, \frac{1}{\mathrm{det}\xi}\in L^{10}(\Omega), \|\xi\|_{L^\infty(\Omega,M_3(\mathbb{R}))}\leq \alpha \}$}}, $W_i\in \{\xi \in L^2(\Omega,M_3(\mathbb{R}))\, |\,  \|\xi\|_{L^\infty(\Omega,M_3(\mathbb{R}))}\leq \beta\}$, $\theta_{T_i} \in {\textcolor{black}{\mathcal{U}}}$ such that $\theta_{T_i,l}\circ \varphi_i \in L^1(\Omega)$ for all $l\in \{1,\cdots,N\}$ \textcolor{black}{and for all $i\in\{1,\cdots,M\}$}, ${\textcolor{black}{\theta_{\tilde T_i,l}\in  BV(\Omega,\{-1,0,}}$ ${\textcolor{black}{1\})}}$  for all $l\in \{1,\cdots,N\}$ \textcolor{black}{and for all $i\in \{1,\cdots,M\}$}, and $\theta_R \in {\textcolor{black}{\mathcal{U}}}$.
\subsection{Theoretical Results}\label{sub-sec:theoretical_result}
In this subsection, we theoretically analyse problem~\ref{decoupled_problem} and show an asymptotic result relating the decoupled problem~\ref{decoupled_problem} to the initial problem~\ref{initial_problem}.
\begin{theorem}[Asymptotic result]
Let $(\gamma_j)_{j\geq 0}$ be an increasing sequence of positive real numbers such that $\underset{j\rightarrow +\infty}{\lim} \gamma_j = +\infty$. Let {\textcolor{black}{$(\{\varphi_{i,k_j},\theta_{T_i,k_j},V_{i,k_j},W_{i,k_j}\}_{i=1}^{M},$ $\left(\theta_{\tilde T_i,l,k_j}\right)_{i=1,\cdots,M \atop l=1,\cdots,N},\theta_{R,k_j})$}} be a minimising sequence of the problem $\mathcal{F}_{1,\gamma}$ for $\gamma=\gamma_j$. Then there exists a subsequence such that $\varphi_{i,k_j}\underset{j\rightarrow +\infty}{\overset{W^{1,4}(\Omega,\mathbb{R}^3)}{\rightharpoonup}}\bar{\varphi}_i$, $\theta_{T_i,k_j} \underset{j\rightarrow +\infty}{\overset{(L^1(\Omega))^{{\textcolor{black}{{N}}}}}{\longrightarrow}} \bar{\theta}_{T_i}$, $\theta_{R,k_j} \underset{j \rightarrow +\infty}{\overset{(L^1(\Omega))^{\textcolor{black}{{N}}}}{\longrightarrow}} \bar{\theta}_{R}$, $V_{i,k_j} \underset{j\rightarrow +\infty}{\overset{*}{\rightharpoonup}} \nabla \bar{\varphi}_i$ in $L^\infty(\Omega,M_3(\mathbb{R}))$, $W_{i,k_j} \underset{j\rightarrow +\infty}{\overset{*}{\rightharpoonup}} (\nabla \bar{\varphi}_i)^{-1} $ in $L^\infty(\Omega,M_3(\mathbb{R}))$, $\theta_{\tilde T_i,l,k_j} \underset{j \rightarrow +\infty}{\overset{L^1(\Omega)}{\longrightarrow}} \bar{\theta}_{T_i,l}\circ\bar{\varphi}_i - \bar{\theta}_{R,l}$, for all $l=1,\cdots,N$, for all $i=1,\cdots,M$, and $\underset{j\rightarrow +\infty}{\lim} \mathcal{F}_{1,\gamma_j}$(${\textcolor{black}{\{\varphi_{i,k_j},\theta_{T_i,k_j},V_{i,k_j},W_{i,k_j}\}_{i=1}^{M},}}$ ${\textcolor{black}{\left(\theta_{\tilde T_i,l,k_j}\right)_{i=1,\cdots,M \atop l=1,\cdots,N},\theta_{R,k_j})}}={\textcolor{black}{\mathcal{F}_1(\bar{\theta}_R,\{\bar{\theta}_{T_i},}}$ ${\textcolor{black}{\bar{\varphi}_i\}_{i=1}^{M}) = \inf \mathcal{F}_1}}$, so that $(\bar{\theta}_R,\{\bar{\theta}_{T_i},$ $\bar{\varphi}_i\}_{i=1}^{M})\in \mathcal{U}^{M+1} \times \hat{\mathcal{W}}^M$ is a minimiser of the initial problem \ref{initial_problem}.
\label{thm:asymptotic_result}
\end{theorem}
\begin{proof}
This proof is divided into three parts. The first one consists of deriving a coercivity inequality. The second one shows the convergence of a minimising sequence and the last one is dedicated to the lower semi-continuity of the functional. See Section 2 of the supplementary material for a detailed proof.
\end{proof}
Equipped with this material and argument, we propose the following discretised numerical scheme.
\subsection{Numerical Scheme}\label{sub-sec:numerical_scheme}
In this subsection, we restrict ourselves to the two-dimensional case and make some minor changes to the model for the purpose of simplicity but the extension to the three-dimensional case shouldn't {\textcolor{black}{induce additional challenges}}.\newline

We now consider the following discrete two-dimensional decoupled problem {\textcolor{black}{\textemdash note that from now on, $\theta_{T_i}$ denotes the partition of $T_i$ into piecewise constant regions, i.e. $\theta_{T_i}=\displaystyle{\sum_{l=1}^{N}}c_{T_i,l}\,\theta_{T_i,l}$, the number of \textcolor{black}{shapes}, $N$ being an unknown (see Remark \ref{remark_discrete_Potts}) similarly for $\theta_R $\textemdash}}:
\begin{align}
  \nonumber \inf \Bigg\{ &\mathcal{F}_{2,\gamma}(\{\varphi_i,\theta_{T_i},V_i,\theta_{\tilde T_i},W_i\}_{i=1}^{M},\theta_R) = \frac{1}{M}\underset{i=1}{\overset{M}{\sum}}\textcolor{black}{\gamma_T}\|\nabla \theta_{T_i}\|_{L^0(\Omega)} \\
  \nonumber &+ \textcolor{black}{\lambda_{T}}\|\theta_{T_i}-T_i\|_{L^2(\Omega)}^2+ \textcolor{black}{\gamma_R}\|\nabla \theta_{R}\|_{L^0(\Omega)} + \textcolor{black}{\lambda_R} \| \theta_R - T_i\circ \varphi_i\|_{L^2(\Omega)}^2 \\
  \nonumber &+ \textcolor{black}{\gamma_{\tilde T}}\|\nabla \theta_{\tilde T_i}\|_{L^0(\Omega)} +\frac{\gamma_1}{2}\|\theta_{\tilde T_i} - (\theta_{T_i}\circ \varphi_i - \theta_{R})\|_{L^2(\Omega)}^2+ \int_\Omega W_{Op}'(V_i, \mathrm{det} V_i)\,dx\\
 \nonumber & + \frac{\gamma_2}{2}\|V_i-\nabla \varphi_i\|_{L^2(\Omega,M_2(\mathbb{R}))}^2+\mathds{1}_{\{\|.\|_{L^\infty(\Omega,M_3(\mathbb{R}))}\leq \alpha\}}(V_i)+\mathds{1}_{\{\|.\|_{L^\infty(\Omega,M_2(\mathbb{R}))}\leq \beta\}}(W_i)\\
  &+\frac{\gamma_3}{2}\|W_i-V_i^{-1}\|_{L^2(\Omega,M_2(\mathbb{R}))}^2\Bigg\}, \tag{DPb} \label{decoupled_problem_bidimensional}
 \end{align}
with $W_{Op}'(\psi,\delta) = \left\{ \begin{array}{ll} a_1\|\psi\|^4+a_2(\delta-1)^2+\frac{a_3}{\delta^{\textcolor{black}{{10}}}}-2a_1-a_3&\text{ if } \delta >0 \\ +\infty &\text{ otherwise }\end{array}\right.$. 
\begin{remark}In the two-dimensional case, the cofactor matrix vanishes and we only {\textcolor{black}{need}} an $L^2$-penalisation to get the asymptotic result as in \cite{debroux-bib:debroux_le_guyader_SIIMS}. Also $BV(\Omega) \hookrightarrow L^2(\Omega) $ in 2 dimensions, so we can replace the $L^1$-penalisation for the auxiliary variable $\theta_{\tilde T_i}$ by an $L^2$-penalisation term.
\end{remark}
\begin{remark}\label{remark_discrete_Potts}
We have also opted for the discrete Potts model for the segmentation as in \cite{debroux-bib:storath} since it does not require any prior knowledge on the number of \textcolor{black}{shapes} in the image. If the number of \textcolor{black}{shapes} is known a priori, another approach based on convexification as in \cite{debroux-bib:pock} can be applied.
\end{remark}
We address this optimisation problem by an alternating scheme in which we fix all the variables except one and solve the subproblem related to the remaining unknown iteratively.
\begin{itemize}
    \item {\bf{Sub-problem 1. Optimisation over $\theta_{T_i}$}}. For each $i=1,\cdots,M$, the problem in $\theta_{T_i}$ amounts to solve
    \begin{align*}
        &\underset{\theta_{T_i}}{\inf}\,\textcolor{black}{\gamma_T}\|\nabla \theta_{T_i}\|_{L^0(\Omega)}+\textcolor{black}{\lambda_T}\|\theta_{T_i}-T_i\|_{L^2(\Omega)}^2+\frac{\gamma_1}{2}\|\theta_{\tilde T_i}-\theta_{T_i}\circ \varphi_i+\theta_{R}\|_{L^2(\Omega)}^2,\\
        \Leftrightarrow & \underset{\theta_{T_i}}{\inf}\,\textcolor{black}{\gamma_T}\|\nabla \theta_{T_i}\|_{L^0(\Omega)}+\textcolor{black}{\lambda_T}\|\theta_{T_i}-T_i\|_{L^2(\Omega)}^2\\
        &+\frac{\gamma_1}{2}\|(\theta_{\tilde T_i}\circ \varphi_i^{-1}-\theta_{T_i}+\theta_{R}\circ \varphi_i^{-1}){\textcolor{black}{\left(\mathrm{det}\nabla \varphi_i\right)^{-\frac{1}{2}}}}\|_{L^2(\Omega)}^2,\\
        \Leftrightarrow & \underset{\theta_{T_i}}{\inf}\,\textcolor{black}{\gamma_T}\|\nabla \theta_{T_i}\|_{L^0(\Omega)}\\
        &+\|{\textcolor{black}{\sqrt{\textcolor{black}{\lambda_T}+(\mathrm{det}\nabla \varphi_i)^{-1}\frac{\gamma_1}{2}}}}\,\theta_{T_i}-\frac{\textcolor{black}{\lambda_T} T_i + {\textcolor{black}{\frac{\gamma_1}{2}\,(\mathrm{det}\nabla \varphi_i)^{-1}}}(\theta_{\tilde T_i}\circ \varphi_i^{-1}+\theta_R\circ\varphi_i^{-1})}{{\textcolor{black}{\sqrt{\textcolor{black}{\lambda_T}+(\mathrm{det}\nabla \varphi_i)^{-1}\frac{\gamma_1}{2}}}}}\|_{L^2(\Omega)}^2.
    \end{align*}
    This amounts to solve the Potts model with nonnegative weights and we use the algorithm in \cite{debroux-bib:storath} based on the Alternating Direction Method of Multipliers (ADMM) and linear programming.
    \item {\bf{Sub-problem 2. Optimisation over $\theta_{\tilde T_i}$}}. For each $i=1,\cdots,M$, the sub-problem in $\theta_{\tilde T_i}$ is the following {\textcolor{black}{one}}:
    \begin{align*}
        &\underset{\theta_{\tilde T_i}}{\inf}\, \textcolor{black}{\gamma_{\tilde T}} \|\nabla \theta_{\tilde T_i}\|_{L^0(\Omega)}+\frac{\gamma_1}{2}\|\theta_{\tilde T_i}-(\theta_{T_i}\circ \varphi_i-\theta_R)\|_{L^2(\Omega)}^2.
    \end{align*}
    This is again the Potts model and we use the same algorithm \cite{debroux-bib:storath} to solve it in practice.
    \item {\bf{Sub-problem 3. Optimisation over $\theta_R$}}. By fixing all the other variables, the optimisation problem with respect to $\theta_R$ becomes
    \begin{align*}
        &\underset{\theta_R}{\inf}\, \textcolor{black}{\gamma_R}\|\nabla \theta_R\|_{L^0(\Omega)}+\frac{1}{M}\underset{i=1}{\overset{M}{\sum}}\textcolor{black}{\lambda_R}\|\theta_R-T_i\circ \varphi_i\|_{L^2(\Omega)}+\frac{\gamma_1}{2}\|\theta_R-(\theta_{T_i}\circ \varphi_i-\theta_{\tilde T_i})\|_{L^2(\Omega)}^2,\\
        \Leftrightarrow & \underset{\theta_R}{\inf}\, \textcolor{black}{\gamma_R}\|\nabla \theta_R\|_{L^0(\Omega)}+\|(\textcolor{black}{\lambda_R}+\frac{\gamma_1}{2})(\theta_R-\frac{(\frac{1}{M}\underset{i=1}{\overset{M}{\sum}}\textcolor{black}{\lambda_R} T_i\circ \varphi_i+\frac{\gamma_1}{2}(\theta_{T_i}\circ \varphi_i)-\theta_{\tilde T_i})}{\textcolor{black}{\lambda_R}+\frac{\gamma_1}{2}})\|_{L^2(\Omega)}^2.
    \end{align*}
    This is again a Potts model that we solve with the Algorithm \cite{debroux-bib:storath}.
    \item {\bf{Sub-problem 4. Optimisation over $V_i$.}} For each $i=1,\cdots,M$, the sub-problem in $V_i$ reads
    \begin{align*}
        &\underset{V_i}{\inf} F(V_i)+{\textcolor{black}{{Reg}}}(V_i)=\int_\Omega a_1\|V_i\|^4+a_2(\mathrm{det}V_i-1)^2+\frac{a_3}{(\mathrm{det}V_i)^{10}}\,dx\\
        &+\frac{\gamma_2}{2}\|V_i-\nabla \varphi_i\|_{L^2(\Omega,M_2(\mathbb{R}))}^2+\frac{\gamma_3}{2}\|W_i-V_i^{-1}\|_{L^2(\Omega)}^2+\mathds{1}_{\{\|.\|_{L^\infty(\Omega,M_2(\mathbb{R}))}\leq \alpha\}}(V_i).
    \end{align*}
    This can be cast as a structured convex non-smooth optimisation problem of the sum of a proper closed convex function ${\textcolor{black}{{Reg}}}(.)= \mathds{1}_{\{\|.\|_{L^\infty(\Omega,M_2(\mathbb{R}))}\leq \alpha\}}(.)$ and a smooth function $F$ corresponding to the remaining of the functional. This is a classical optimisation problem and several schemes have been developed to solve it. In practice, we use the simple iterative forward-backward splitting algorithm \cite{debroux-bib:lions}:
    \begin{align*}
        V_i^{k+1}=\mathrm{prox}_{\gamma {\textcolor{black}{{Reg}}}}(V_i^k-\gamma \nabla F(V_i^k)),
    \end{align*}
    with $\mathrm{prox}_{\gamma {\textcolor{black}{{Reg}}}}(y) = \underset{x}{\min}\frac{1}{2}\|x-y\|_2^2+\gamma {\textcolor{black}{{Reg}}}({\textcolor{black}{y}}) = \underset{x}{\min}\frac{1}{2}\|x-y\|_2^2+\gamma \mathds{1}_{\{\|.\|_{L^\infty(\Omega,M_2(\mathbb{R}))}\leq \alpha\}}({\textcolor{black}{y}})=P_{\{\|.\|_{L^\infty(\Omega,M_2(\mathbb{R}))}\leq \alpha\}}(y)$, $P_C$ being the projection operator onto the convex set $C$. This could be improved in future work by using for instance the algorithm proposed in \cite{debroux-bib:liang}.
    \item {\bf{Sub-problem 5. Optimisation over $W_i$.}} For each $i=1,\cdots,M$, we solve the following minimisation problem
    \begin{align*}
  &\underset{W_i}{\inf}\, \frac{\gamma_3}{2}\|W_i-V_i^{-1}\|_{L^2(\Omega)}^2 + \mathds{1}_{\{\|.\|_{L^\infty(\Omega,M_2(\mathbb{R}))}\leq \beta\}}(W_i) =P_{\{\|.\|_{L^\infty(\Omega,M_2(\mathbb{R}))}\leq \alpha\}}(V_i^{-1}).
    \end{align*}
    \item {\bf{Sub-problem 6. Optimisation over $\varphi_i$.}} For each $i=1,\cdots,M$, the sub-problem in $\varphi_i$ reads
    \begin{align*}
   \underset{\varphi_i}{\inf}\,     \gamma_R\|\theta_R-{\textcolor{black}{T_i}}\circ \varphi_i\|_{L^2(\Omega)}^2+\frac{\gamma_1}{2}\|\theta_{\tilde T_i}-\theta_{T_i}\circ \varphi_i+\theta_R\|_{L^2(\Omega)}^2+\frac{\gamma_2}{2}\|V_i-\nabla \varphi_i\|_{L^2(\Omega)}^2.
    \end{align*}
    We propose to solve the {\textcolor{black}{associated}} Euler-Lagrange equation using an $L^2$-gradient flow scheme with an implicit Euler time stepping.
\end{itemize}
The overall algorithm is summarised in Algorithm \ref{alg:numerical_resolution}.
\begin{algorithm}[htbp!]
\caption{Alternating scheme of resolution.}
\label{alg:numerical_resolution}
\begin{algorithmic}
\STATE{1. }{Define $k:=1$, $T_i:=$ $i$-th template image,  $\theta_{T_i}:={\textcolor{black}{\mathrm{Potts\,\, segmentation\,\, of\,\, }}} T_i$,\newline$\theta_R:={\textcolor{black}{\mathrm{Potts\,\, segmentation\,\, of\,\, }\frac{1}{M}\,{\displaystyle{\sum_{i=1}^{M}}}\,T_i}}$,\, $\theta_{\tilde T_i}:=\theta_{T_i}-\theta_{R}$,  $V_i=\begin{pmatrix}V_{i,11}&V_{i,12}\\V_{i,21}&V_{i,22}\end{pmatrix}:={\textcolor{black}{I}}$,  $W_i:=\begin{pmatrix}W_{i,11} & W_{i,12} \\ W_{i,21} & W_{i,22} \end{pmatrix}:=I$, $a_1$, $a_2$, $a_3$, \textcolor{black}{$\lambda_T$, $\lambda_R$,} $\gamma_T$, $\gamma_R$, $\gamma_1$, $\gamma_2$, $\gamma_3$, $nbIter$, $\alpha$, $\beta$, $U_i=(U_{i,1},U_{i,2}):=0$, displacements associated to the deformation $\varphi_i$, for $i=1,\cdots,M$.}
\WHILE{$k < nbIter$}
\IF{k\%10==0}
\STATE{2.1. }{Compute for $i=1,\cdots,M$ the inverse deformation $\varphi_i^{-1}$ using a Delaunay triangulation and linear interpolation.}
\STATE{2.2. }{For $i=1,\cdots,M$, update $\theta_{T_i}$ by solving the Potts model with algorithm in \cite{debroux-bib:storath}: $
\underset{u}{\inf}\,\textcolor{black}{\gamma_T}\|\nabla u\|_{L^0(\Omega)}+\|{\textcolor{black}{\sqrt{\textcolor{black}{\lambda_T}+(\mathrm{det}\nabla \varphi_i)^{-1}\frac{\gamma_1}{2}}}}\,u
 -\frac{\textcolor{black}{\lambda_T} T_i + {\textcolor{black}{\frac{\gamma_1}{2}\,(\mathrm{det}\nabla \varphi_i)^{-1}}}(\theta_{\tilde T_i}\circ \varphi_i^{-1}+\theta_R\circ\varphi_i^{-1})}{{\textcolor{black}{\sqrt{\textcolor{black}{\lambda_T}+(\mathrm{det}\nabla \varphi_i)^{-1}\frac{\gamma_1}{2}}}}}\|_{L^2(\Omega)}^2
$ end for 2.2.}
\STATE{2.3. }{For $i=1,\cdots,M$, update $\theta_{\tilde T_i}$ by solving the Potts model with algorithm in \cite{debroux-bib:storath}: $\underset{u}{\inf}\textcolor{black}{\gamma_{\tilde T}}\|\nabla u\|_{L^0(\Omega)} +\frac{\gamma_1}{2}\| (\theta_{T_i}\circ \varphi_i-\theta_{R})-u\|_{L^2(\Omega)}^2$ end for 2.3.}
\STATE{2.4. }{Update $\theta_{R}$ by solving the Potts model with algorithm in \cite{debroux-bib:storath}: $\underset{u}{\inf}\textcolor{black}{\gamma_R}\|\nabla u\|_{L^0(\Omega)} +{\textcolor{black}{(\textcolor{black}{\lambda_R}+\dfrac{\gamma_1}{2})\,\|u-\dfrac{\frac{\textcolor{black}{\lambda_R}}{M}\,{\sum_{i=1}^{M}}\,T_i\circ \varphi_i+\frac{\gamma_1}{2M}\,{\sum_{i=1}^{M}}\,\left(\theta_{T_i}\circ \varphi-\theta_{\tilde{T_i}}\right)}{\textcolor{black}{\lambda_R}+\dfrac{\gamma_1}{2}}\|^2_{L^2(\Omega)}}}$}.
\ENDIF
\STATE{2.5. }{For each $i=1,\cdots,M$, for each pixel $(l,j)$, {\textcolor{black}{\textemdash $c$ playing a role similar to the one of a step size in a gradient method\textemdash}} update $V_i$ using the following equations: $\left\{\begin{array}{lll}
temp_1(l,j)&=&V_{i,11}(l,j)+c\bigg(\frac{10a_3}{(\mathrm{det}V_i(l,j))^{11}}\,V_{i,22}(l,j)-4a_1\,V_{i,11}(l,j)\\
&&\|V_i(l,j)\|^2-2a_2(\mathrm{det}V_i(l,j)-1)\,V_{i,22}(l,j)+\gamma_2(\frac{\partial \varphi_{i,1}}{\partial x}(l,j)\\
&&-V_{i,11}(l,j))-\gamma_3(W_{i,11}(l,j)-\frac{V_{i,22}(l,j)}{\mathrm{det}V_i(l,j)})(\frac{V_{i,22}(l,j)^2}{(\mathrm{det}V_i(l,j))^2})\\
&&+\gamma_3(W_{i,12}(l,j)+\frac{V_{i,12}(l,j)}{\mathrm{det}V_i(l,j)})(\frac{V_{i,22}(l,j)V_{i,12}(l,j)}{(\mathrm{det}V_i(l,j))^2})\\
&&+\gamma_3(W_{i,21}(l,j){\textcolor{black}{+}}\frac{V_{i,21}(l,j)}{\mathrm{det}V_i(l,j)})(\frac{V_{i,22}(l,j)V_{i,21}(l,j)}{(\mathrm{det}V_i(l,j))^2})\\
&&-\gamma_3(W_{i,22}(l,j)-\frac{V_{i,11}(l,j)}{\mathrm{det}V_i(l,j)})(-\frac{1}{\mathrm{det}V_i(l,j)}+\frac{V_{i,11}(l,j)V_{i,22}(l,j)}{(\mathrm{det}V_i(l,j))^2})\bigg),\\
V_{i,11}(l,j)&=&\left\{ \begin{array}{ll} -\alpha &\text{ if } temp_1(l,j)< -\alpha,\\temp_1(l,j)&\text{ if } |temp_1(l,j) |\leq \alpha,\\ \alpha &\text{ if }temp_1(l,j)>\alpha \end{array}\right.,\\
temp_2(l,j)&=&V_{i,12}(l,j)+c\bigg(\frac{10a_3}{(\mathrm{det}V_i(l,j))^{11}}(-V_{i,21}(l,j))-4a_1(V_{i,12}(l,j))\\
&&\|V_i(l,j)\|^2+2a_2(\mathrm{det}V_i(l,j)-1)(V_{i,21}(l,j))+\gamma_2(\frac{\partial \varphi_{i,1}}{\partial y}(l,j)\\
&&-V_{i,12}(l,j))+\gamma_3(W_{i,11}(l,j)-\frac{V_{i,22}(l,j)}{\mathrm{det}V_i(l,j)})(\frac{V_{i,22}(l,j)V_{i,21}(l,j)}{(\mathrm{det}V_i(l,j))^2})\\
&&-\gamma_3(W_{i,12}(l,j)+\frac{V_{i,12}(l,j)}{\mathrm{det}V_i(l,j)})(\frac{1}{\mathrm{det}V_i(l,j)}+\frac{V_{i,21}(l,j)V_{i,12}(l,j)}{(\mathrm{det}V_i(l,j))^2})\\
&&-\gamma_3(W_{i,21}(l,j){\textcolor{black}{+}}\frac{V_{i,21}(l,j)}{\mathrm{det}V_i(l,j)})(\frac{V_{i,21}(l,j)^2}{(\mathrm{det}V_i(l,j))^2})\\
&&+\gamma_3(W_{i,22}(l,j)-\frac{V_{i,11}(l,j)}{\mathrm{det}V_i(l,j)})(\frac{V_{i,11}(l,j)V_{i,21}(l,j)}{(\mathrm{det}V_i(l,j))^2})\bigg),\\
V_{i,12}(l,j)&=&\left\{ \begin{array}{ll} -\alpha &\text{ if } temp_2(l,j)< -\alpha,\\temp_2(l,j)&\text{ if } |temp_2(l,j) |\leq \alpha,\\ \alpha &\text{ if }temp_2(l,j)>\alpha \end{array}\right.,
\end{array}\right.$}
\algstore{myalg1}
\end{algorithmic}
\end{algorithm}
\begin{algorithm}[htp!]
\begin{algorithmic}
\algrestore{myalg1}
\STATE{2.5. }{$\left\{ \begin{array}{lll} 
temp_3(l,j)&=&V_{i,21}(l,j)+c\bigg(\frac{10a_3}{(\mathrm{det}V_i(l,j))^{11}}(-V_{i,12}(l,j))-4a_1(V_{i,21}(l,j))\\
&&\|V_i(l,j)\|^2-2a_2(\mathrm{det}V_i(l,j)-1)(-V_{i,12}(l,j))+\gamma_2(\frac{\partial \varphi_{i,2}}{\partial x}(l,j)\\
&&-V_{i,21}(l,j)){\textcolor{black}{+}}\gamma_3(W_{i,11}(l,j){\textcolor{black}{-}}\frac{V_{i,22}(l,j)}{\mathrm{det}V_i(l,j)})\\
&&(\frac{V_{i,22}(l,j)V_{i,12}(l,j)}{(\mathrm{det}V_i(l,j))^2})-\gamma_3(W_{i,12}(l,j)+\frac{V_{i,12}(l,j)}{\mathrm{det}V_i(l,j)})(\frac{V_{i,12}(l,j)^2}{(\mathrm{det}V_i(l,j))^2})\\
&&-\gamma_3(W_{i,21}(l,j){\textcolor{black}{+}}\frac{V_{i,21}(l,j)}{\mathrm{det}V_i(k,j)})({\textcolor{black}{\frac{1}{\mathrm{det}V_i(l,j)}}}+\frac{V_{i,12}(l,j)V_{i,21}(l,j)}{(\mathrm{det}V_i(l,j))^2})\\
&&+\gamma_3(W_{i,22}(l,j)-\frac{V_{i,11}(l,j)}{\mathrm{det}V_i(l,j)})(\frac{V_{i,11}(l,j)(V_{i,12}(l,j))}{(\mathrm{det}V_i(l,j))^2})\bigg),\\
V_{i,21}(l,j)&=&\left\{ \begin{array}{ll} -\alpha &\text{ if } temp_3(l,j)< -\alpha,\\temp_3(l,j)&\text{ if } |temp_3(l,j) |\leq \alpha,\\ \alpha &\text{ if }temp_3(l,j)>\alpha \end{array}\right.\\
temp_4(l,j)&=&V_{i,22}(l,j)+c\bigg(\frac{10a_3}{(\mathrm{det}V_i(l,j))^{11}}V_{i,11}(l,j)-4a_1(V_{i,22}(l,j))\\
&&\|V_i(l,j)\|^2-2a_2(\mathrm{det}V_i(l,j)-1)V_{i,11}(l,j)\\
&&+\gamma_2(\frac{\partial \varphi_{i,2}}{\partial y}(l,j)-V_{i,22}(l,j))-\gamma_3(W_{i,11}(l,j)-\frac{V_{i,22}(l,j)}{\mathrm{det}V_i(l,j)})\\
&&(-\frac{1}{\mathrm{det}V_i(l,j)}+\frac{V_{i,22}(l,j)V_{i,11}(l,j)}{(\mathrm{det}V_i(l,j))^2})+\gamma_3(W_{i,12}(l,j)\\
&&+\frac{V_{i,12}(l,j)}{\mathrm{det}V_i(l,j)})(\frac{V_{i,11}(l,j)V_{i,12}(l,j)}{(\mathrm{det}V_i(l,j))^2})+\gamma_3(W_{i,21}(l,j){\textcolor{black}{+}}\frac{V_{i,21}(l,j)}{\mathrm{det}V_i(l,j)})\\
&&(\frac{V_{i,11}(l,j))V_{i,21}(l,j)}{(\mathrm{det}V_i(l,j))^2})-\gamma_3(W_{i,22}(l,j)-\frac{V_{i,11}(l,j)}{\mathrm{det}V_i(l,j)})\\
&&(\frac{{\textcolor{black}{(V_{i,11}(l,j))^2}}}{(\mathrm{det}V_i(l,j))^2})\bigg),\\
V_{i,22}(l,j)&=&\left\{ \begin{array}{ll} -\alpha &\text{ if } temp_4(l,j)< -\alpha,\\ temp_4(l,j) &\text{ if } |temp_4(l,j)|\leq \alpha,\\ \alpha &\text{ if }temp_4(l,j)>\alpha 
\end{array}\right. .
\end{array} \right.$ end for 2.5.}
\STATE{2.6. }{For each $i=1,\cdots,M$, for each pixel $(l,j)$, update $W_i$ with this closed form: $\left\{ 
\begin{array}{lll}
W_{i,11}(l,j)&=&\left\{\begin{array}{ll} -\beta &\text{ if } \frac{V_{i,22}(l,j)}{\mathrm{det}V_i(l,j)}< -\beta\\ 
\frac{V_{i,22}(l,j)}{\mathrm{det}V_i(l,j)} & \text{ if } |\frac{V_{i,22}(l,j)}{\mathrm{det}V_i(l,j)}|\leq \beta \\ \beta &\text{ if } \frac{V_{i,22}(l,j)}{\mathrm{det}V_i(l,j)}> \beta \end{array}\right.\\
W_{i,12}(l,j)&=&\left\{\begin{array}{ll} -\beta &\text{ if } \frac{-V_{i,12}(l,j)}{\mathrm{det}V_i(l,j)}< -\beta\\ \frac{-V_{i,12}(l,j)}{\mathrm{det}V_i(l,j)} & \text{ if } |\frac{-V_{i,12}(l,j)}{\mathrm{det}V_i(l,j)}|\leq \beta \\ \beta &\text{ if } \frac{-V_{i,12}(l,j)}{\mathrm{det}V_i(l,j)}> \beta \end{array}\right.\\
W_{i,21}(l,j)&=&\left\{\begin{array}{ll} -\beta &\text{ if } \frac{-V_{i,21}(l,j)}{\mathrm{det}V_i(l,j)}<-\beta\\ 
\frac{-V_{i,21}(l,j)}{\mathrm{det}V_i(l,j)} & \text{ if } |\frac{-V_{i,21}(l,j)}{\mathrm{det}V_i(l,j)}|\leq \beta \\ 
\beta &\text{ if } \frac{-V_{i,21}(l,j)}{\mathrm{det}V_i(l,j)}> \beta \end{array}\right.\\
W_{i,22}(l,j)&=&\left\{\begin{array}{ll} -\beta &\text{ if } \frac{{\textcolor{black}{V_{i,11}(l,j)}}}{\mathrm{det}V_i(l,j)}< -\beta\\ \frac{{\textcolor{black}{V_{i,11}(l,j)}}}{\mathrm{det}V_i(l,j)} & \text{ if } |\frac{{\textcolor{black}{V_{i,11}(l,j)}}}{\mathrm{det}V_i(l,j)}|\leq \beta \\ \beta &\text{ if } \frac{{\textcolor{black}{V_{i,11}(l,j)}}}{\mathrm{det}V_i(l,j)}>\beta \end{array}\right.
\end{array}
\right.
$ end for 2.6.}
\STATE{2.7. }{Solve for all $i=1,\cdots,M$, the Euler-Lagrange equation in $U_i$ using an implicit finite difference scheme: $0= \gamma_1 \nabla \theta_{T_i}\circ\varphi_i(\theta_{T_i}\circ \varphi_i-\theta_R-\theta_{\tilde T_i})+\textcolor{black}{\lambda_R}\,(T_i\circ \varphi_i-\theta_R)\nabla T_i\circ \varphi_i+ \gamma_2 \begin{pmatrix}
      \mathrm{div}V_{i,1}\\                                                                                                                                                                                                                                      
          \mathrm{div}V_{i,2}                                                                                                                                                                                                                                 \end{pmatrix}
 $, where $V_{i,j}$ stands for the $j$\textsuperscript{th} row of $V_i$, and $\varphi_i=Id+U_i$. end for 2.7.}
\STATE{2.8. }{$k:=k+1$.}
\ENDWHILE
\label{myalg}
\end{algorithmic}
\end{algorithm}
{\textcolor{black}{
\begin{remark} Under mild assumptions \textemdash by replacing the $L^0$-penalization by an $L^1$ one \textemdash, we can prove the convergence of the algorithm as well as a $\Gamma$-convergence result. 
\end{remark}}
We now turn to the geometry-driven statistical analysis.
\section{Representation of the Deformations in a Linear Space and Geometry-driven PCA}\label{sec:representation_deformation_linear_space_PCA}
In this section, we focus on the performance of a statistical analysis on the obtained deformations in order to retrieve the main modes of variations in terms of geometric distortions in the initial set of images. The main {\textcolor{black}{hindrance}} is that our deformation maps live in a nonlinear space whereas classical statistical tools require the objects to be in a linear space. Therefore, we first need to find a good representation of our deformations in a linear space equipped with a scalar product ({\textcolor{black}{in order to compute the covariance operator}}), {\textcolor{black}{enabling us to perform a Principal Component Analysis (PCA) on these representatives afterwards}}. The fundamental axiom of elasticity stating that the energy required to deform an object from a state of reference to another equilibrium state is the same  {\textcolor{black}{whatever the chosen path is}}, prevents a straightforward definition of geodesics. Therefore, the use of Riemannian geometry principles as in \cite{debroux-bib:srivastava3} {\textcolor{black}{cannot be envisioned}}. 

In the following, we propose, study and compare three different strategies to get a relevant depiction of our deformations in a linear space. {\textcolor{black}{The first two ones are physically/mechanically-oriented and can be viewed as non-straightforward adaptations of \cite{debroux-bib:Rumpf2011}. While in \cite{debroux-bib:Rumpf2011} the shapes are modelled through their boundaries and subject to boundary stresses, our framework involves the whole image as the object to be deformed\textcolor{black}{. N}ote that with the prescribed boundary conditions $\forall i \in \left\{1,\cdots,M\right\}$, $\varphi_i=Id$ on $\partial \Omega$, no boundary stress is applied\textcolor{black}{,} and subsequently, inner volumetric stresses  \textcolor{black}{are considered}. This constitutes a major difference with the work of \cite{debroux-bib:Rumpf2011} and as demonstrated in Section 3 of the supplementary material, it entails substantial adaptations in the mathematical developments.}}
The first two methods rely on fundamental notions of elastic behaviour and the following observation made in \cite{debroux-bib:Rumpf2011}~:
{\textit{"the classical covariance tensor can be identified with the covariance tensor of the displacements obtained by adding a small fraction of the $i$-th spring force under the Hooke's law"}}. {\textcolor{black}{Whilst the first method is based on the linearisation of the stored energy function around the identity,  which might result in the loss of the initial nonlinear nature of the deformations but has the advantage of being fast, the second approach is more intricate. It retrieves the whole nature of the deformations by performing the PCA on the Cauchy stress tensors, relying on the locally underlying one-to-one relation between this tensor and the deformation, but requires the resolution of a highly nonlinear and non convex problem similar to the one studied previously to come back to the deformation space. \newline 
Our goal was to design an alternative method that would be a good compromise between rendering the nonlinear nature of the deformation and in terms of numerical complexity. The first objective is achieved by handling both the deformation field and the deformation tensors that encode the local deformation state resulting from stresses. This constitutes another novelty of the proposed work. We have moved toward a completely different point of view since the problem is no longer \textcolor{black}{explored} as a physical one but is now identified as an approximation one in the $D^m$-spline setting (\cite{debroux-bib:arcangeli}). 
The first two methods thus serve as benchmark to assess the interest of this new vision.\newline
Due to page number limitation and as the third method proves to be a proper trade-off between ability to reproduce the nonlinear nature of the deformations and intermediate computation time, the mathematical details of the first two methods are postponed in the document supplementary material Section 3 and we only focus on the third one.}}
\subsection{Third Approach: Approximation Modelling}
\label{sub-sec:third_approach_approximation}
This section is devoted to the analysis of a novel method in which the linear representation problem is seen as an approximation one in the $D^m$-spline setting. Since the deformation tensor suitably characterises the local deformation (amplitude, direction, etc.), we aim at finding \textcolor{black}{an appropriate} approximation of our deformations in a linear space $H^3(\Omega,\mathbb{R}^2)$ that also approximates well the deformation tensors. {\textcolor{black}{For the sake of clarity, we omit the indices $i$ in the following. Henceforth, $u$ denotes the displacement field related to $\varphi\textcolor{black}{^{-1}}$, \textcolor{black}{inverse }deformation field obtained at the outcome of the first algorithmic stage.}} We consider the following problem in the two-dimensional case {\textcolor{black}{\textemdash case of interest in the numerical part, but straightforwardly extendable to 3D \textemdash}}:
{\textcolor{black}{\begin{align}
    \nonumber \underset{v\in H^3(\Omega,\mathbb{R}^2)}{\min}& \epsilon |v|_{3,\Omega,\mathbb{R}^2}^2 + \frac{\gamma}{2}\langle \rho(\nabla v+\nabla v^T - \nabla u^T - \nabla u^T - \nabla u^T\nabla u)\rangle_{M_2(\mathbb{R}),N}^2\\
    &+\langle \zeta(v-u)\rangle_{\mathbb{R}^2,N}^2,\label{problem_approximation_pratique}
\end{align}
where $|.|_{3,\Omega,\mathbb{R}^2}$ is the semi-norm on $H^3(\Omega,\mathbb{R}^2)$,
{\small{$\zeta : \left| \begin{array}{ccc} H^3(\Omega,\mathbb{R}^2) &\rightarrow & \left(\mathbb{R}^2\right)^N \\
v &\mapsto &\zeta(v)=\left(v(a_1),\cdots,v(a_N)\right)^T
\end{array}\right.$}}, {\small{$\rho : \left| \begin{array}{ccc} H^2(\Omega,M_2(\mathbb{R})) &\rightarrow & \left(M_2(\mathbb{R})\right)^N \\
v &\mapsto &\rho(v)=\left(v(a_1),\cdots,v(a_N)\right)^T
\end{array}\right.$}}. Also, $a_1,\cdots,a_N$ denote the image pixel coordinates with $N$ the total number of pixels, and $\forall \xi \in (M_2(\mathbb{R}))^N$, $\forall \eta \in  (M_2(\mathbb{R}))^N$, $\langle \xi,\eta\rangle_{M_2(\mathbb{R}),N}=\displaystyle{\sum_{i=1}^{N}}\,\xi_i\,:\,\eta_i$, while $\forall \xi \in (\mathbb{R}^2)^N$, $\forall \eta \in  (\mathbb{R}^2)^N$, $\langle \xi,\eta\rangle_{\mathbb{R}^2,N}=\displaystyle{\sum_{i=1}^{N}}\,\xi_i^T \,\eta_i$.
}}
However, {\textcolor{black}{stated as \textcolor{black}{it} is, the problem}} is not well defined since $u\in W^{1,\infty}_0(\Omega,\mathbb{R}^2)$ and does not belong to $\mathcal{C}^1(\Omega, \mathbb{R}^2)$ preventing us from extracting isolated values of $\nabla u$. Therefore, for the theoretical analysis of the model, we introduce $(f_k)\in C^\infty_0(\Omega,\mathbb{R}^2)\cap W^{1,\infty}(\Omega,\mathbb{R}^2)$, the sequence from the density result such that, 
$$ f_k \underset{k\rightarrow +\infty}{\longrightarrow} u\,\,\,\,{\mbox{{\textcolor{black}{in $W^{1,\infty}\textcolor{black}{(\Omega,\mathbb{R}^2)}$}}}}.$$
{\textcolor{black}{In practice however, we solve problem \ref{problem_approximation_pratique} and we  give details on the implementation in Sub-section \ref{sub-sec:numerical_resolution}.}
\begin{remark} {\textcolor{black}{An alternative approach would consist in using Lebesgue-Besicovitch differentiation theorem that states that for almost every point, the value of an integrable function is the limit of infinitesimal averages taken about the point. }}
\end{remark}
Let $A_0=\{a_i\}_{i=1,\cdots,N_0}$ and $A_1=\{b_i\}_{i=1,\cdots,N_1}$ be two sets of $N_0$ and $N_1$ points of $\bar{\Omega}$ respectively, {\textcolor{black}{containing both}} a $P^1$-unisolvent subset. Let us denote by $\rho_0$ the operator defined by 
\begin{align*}
    \rho_0 : \left| \begin{array}{l} H^3(\Omega,\mathbb{R}^2) \rightarrow (\mathbb{R}^2)^{N_0}\\v\mapsto \rho_0(v)=(v(a_i))_{i=1,\cdots,N_0}^T \end{array} \right. ,
\end{align*}
and by $\rho_1$ the operator defined by 
\begin{align*}
    \rho_1:\left| \begin{array}{l} {\textcolor{black}{H^2(\Omega,M_2(\mathbb{R}))}} \rightarrow  {\textcolor{black}{(M_2(\mathbb{R}))^{N_1}}}\\v\mapsto \rho_1(v)=(v(b_i))_{i=1,\cdots,N_1}^T\end{array}\right. .
\end{align*}
We introduce the functionals
$$\mathcal{F}_{\epsilon,k} : \left\{ \begin{array}{l} H^3(\Omega,\mathbb{R}^2)  \rightarrow  \mathbb{R} \\ v  \mapsto  \langle \rho_0(v-f_k) \rangle_{{\textcolor{black}{\mathbb{R}^2,N_0}}}^2 + \frac{\gamma}{2} \langle \rho_1(\nabla v + \nabla v^T - \nabla f_k - \nabla f_k^T -\nabla f_k^T \nabla f_k \rangle_{M_2(\mathbb{R}),N_1}^2\\
+ \epsilon|v|_{3,\Omega,\mathbb{R}^2}^2 \end{array}\right.,  $$
and consider the problem
\begin{align}
 \left\{ \begin{array}{l} \text{Search for }u_\epsilon \in H^3(\Omega,\mathbb{R}^2)\text{ such that: } \\ \forall v \in H^3(\Omega,\mathbb{R}^2),\, \mathcal{F}_{\epsilon,k}(u_\epsilon) \leq \mathcal{F}_{\epsilon,k}(v)   \end{array} 
 \right. . \label{init_prob_approx}
\end{align}
{\textcolor{black}{We omit the explicit dependency of $u_{\epsilon}$ on $k$.}}
In the sequel, we theoretically study the model \ref{init_prob_approx}, and start by proving its equivalence with a variational formulation.
\begin{theorem}[Equivalence of problems]
 The problem \ref{init_prob_approx} is equivalent to the following variational problem: 
 \begin{align}
  \left\{ \begin{array}{l} \text{Search for } u_\epsilon \in H^3(\Omega,\mathbb{R}^2) \text{ such that }\forall v \in H^3(\Omega,\mathbb{R}^2),\\  \langle \rho_0(u_\epsilon),\rho_0(v)\rangle_{{\textcolor{black}{\mathbb{R}^2,N_0}}} + {\textcolor{black}{\frac{\gamma}{2}}}\,\langle\rho_1(\nabla u_\epsilon + \nabla u_\epsilon^T),\rho_1(\nabla v +\nabla v^T)\rangle_{M_2(\mathbb{R}),N_1} + \epsilon(u_\epsilon,v)_{3,\Omega,\mathbb{R}^2}\\
  =\langle \rho_0(v),\rho_0(f_k)\rangle_{\mathbb{R}^2,N_0} + {\textcolor{black}{\frac{\gamma}{2}}} \langle \rho_1({\textcolor{black}{\nabla v+\nabla v^T}}),\rho_1(\nabla f_k + \nabla f_k^T + \nabla f_k^T \nabla f_k)\rangle_{M_2(\mathbb{R}),N_1}.\end{array} \right. \label{init_prob_approx_variationnel}
 \end{align}
 \label{thm:equivalence_problems}
\end{theorem}
\begin{proof}
The detailed proof is available in Section 5 of the supplementary material.
\end{proof}
We now define a new norm equivalent to the classical norm on $H^3(\Omega,\mathbb{R}^2)$, which will be useful in the following. {\textcolor{black}{We make the dependency on the set $A_0$ explicit, while the set $A_1$ is supposed to be fixed once and for all.}
\begin{lemma}[Equivalence of norms]
 The mapping defined by 
 \begin{align*}
  \|.\|_{{\textcolor{black}{A_0}},3,\Omega,\mathbb{R}^2} : \left\{ \begin{array}{l} H^3(\Omega,\mathbb{R}^2) \rightarrow \mathbb{R}\\f\mapsto \|f\|_{A_0,3,\Omega,\mathbb{R}^2}=(\langle \rho_0(f)\rangle_{\mathbb{R}^2,N_0}^2+\langle \rho_1(\nabla f + \nabla f^T)\rangle_{M_2(\mathbb{R}),N_1}^2 + |f|_{3,\Omega,\mathbb{R}^2}^2)^{\frac{1}{2}} \end{array} \right.,
 \end{align*}
 is a Hilbert norm equivalent to the norm $\|.\|_{3,\Omega,\mathbb{R}^2}$ in $H^3(\Omega,\mathbb{R}^2)$.
\label{lemma:equivalence_norms}
\end{lemma}
\begin{proof}
The detailed proof is given in Section 6 of the supplementary material.
 \end{proof}
We are now able to prove the existence and uniqueness of {\textcolor{black}{the}} minimizer.
\begin{theorem}[Existence and uniqueness of a minimizer]
 The variational problem \ref{init_prob_approx_variationnel} admits a unique solution.
 \label{thm:existence_uniqueness_min_approx}
\end{theorem}
\begin{proof}
 The detailed proof is given in Section 7 of the supplementary material.
\end{proof}
We now focus on a convergence result. Let $D$ be a subset of ${\textcolor{black}{]0,+\infty[}}$ for which $0$ is an accumulation point. For any $d\in D$, let $A^d$ be a set of $N=N(d)$ distinct points from $\bar{\Omega}$ that contains a $P^1$-unisolvent subset. We assume that $\underset{x\in \Omega}{\sup}\, \delta(x,A^d)=d$, where $\delta$ is the Euclidean distance in $\mathbb{R}^2$. Thus $d$ is the radius of the biggest sphere included in $\Omega$ that contains no point from $A^d$. Also $d$ is bounded and $\underset{d\rightarrow 0}{\lim}\,\underset{x \in \Omega}{\sup}\,\delta(x,A^d)=0$. For any $d\in D$, let us denote by $\rho^d$ the mapping defined by 
\begin{align*}
 \rho^d : \left\{ \begin{array}{l} H^3(\Omega,\mathbb{R}^2) \rightarrow (\mathbb{R}^2)^N\\ v \mapsto \rho^d(v) = \left((v(a))_{a \in A^d}\right)^T \end{array}\right. ,
\end{align*}
{\textcolor{black}{and by}} $\|.\|_{A^d,3,\Omega,\mathbb{R}^2}$, the norm  defined by 
\begin{align*}
 \|f\|_{A^d,3,\Omega,\mathbb{R}^2}=[\langle \rho^d(f)\rangle_N^2+\langle \rho_1(\nabla f + \nabla f^T)\rangle_{M_2(\mathbb{R}),N_1}^2+|f|_{3,\Omega,\mathbb{R}^2}]^{\frac{1}{2}}.
\end{align*}
As shown in the previous lemma, $\|.\|_{A^d,3,\Omega,\mathbb{R}^2}$ is equivalent to the norm $\|.\|_{3,\Omega,\mathbb{R}^2}$ in $H^3(\Omega,\mathbb{R}^2)$ {\textcolor{black}{\textemdash but not uniformly in $d$\textemdash}}.
\begin{lemma}
 Let $B_1=\{b_{01},\cdots,b_{0,\mathcal{N}}\}$ be a fixed $P^1$-unisolvent subset of $\bar{\Omega}$. By hypothesis, $0\in \bar{D}$, and $\underset{d\rightarrow 0}{\lim}\,\underset{x\in \Omega}{\sup}\, \delta(x,A^d) = 0$ holds, so 
 \begin{align*}
  \forall j=1,\cdots,\mathcal{N},\, \exists (a_{0,j}^d)_{d\in D},\,(\forall d \in D, a_{0,j}^d \in A^d) \text{ and } b_{0j} = \underset{d\rightarrow 0}{\lim}\,a_{0j}^d.
 \end{align*}
For any $d\in D$, let $A_0^d$ be the set $\{a_{01}^d,\cdots,a_{0,\mathcal{N}}^d\}$ and let $\|.\|_{A_0^d,3,\Omega,\mathbb{R}^2}$ be the norm defined by $\forall f \in H^3(\Omega,\mathbb{R}^2)$,
\begin{align*}
 \|f\|_{A_0^d,3,\Omega,\mathbb{R}^2} = [\underset{j=1}{\overset{\mathcal{N}}{\sum}}\langle f(a_{0j}^d) \rangle_{\mathbb{R}^2}^2+\underset{i=1}{\overset{N_1}{\sum}}{\textcolor{black}{\|\nabla f(b_i) + \nabla f(b_i)^T\|^2}} + |f|_{3,\Omega,\mathbb{R}^2}^2]^{\frac{1}{2}}.
\end{align*}
Then there exists $\eta>0$ such that for any $d\leq \eta$, $\|.\|_{A_0^d,3,\Omega,\mathbb{R}^2}$ is a norm on $H^3(\Omega,\mathbb{R}^2)$ uniformly equivalent on $D\cap ]0,\eta]$ to the norm $\|.\|_{3,\Omega,\mathbb{R}^2}$.
\label{lemma:convergence_approx_1}
\end{lemma}
\begin{proof}
 The detailed proof is available Section 8 of the supplementary material.
\end{proof}
{\textcolor{black}{Equipped with this result, we are able to}} prove a convergence result on the following problem: 
\begin{align}
 \left\{ \begin{array}{l} \text{Search for } u_\epsilon^d \in H^3(\Omega,\mathbb{R}^2) \text{ such that } \forall v \in H^3(\Omega,\mathbb{R}^2),\\ \langle \rho^d(u_\epsilon^d-f_k) \rangle_{\mathbb{R}^2,N}^2 + \frac{\gamma}{2}\langle \rho_1(\nabla u_\epsilon^d +(\nabla u_\epsilon^d)^T-\nabla f_k - \nabla f_k^T - \nabla f_k^T \nabla f_k)\rangle_{M_2(\mathbb{R}),N_1}^2\\
 +\epsilon |u_\epsilon^d|_{3,\Omega,\mathbb{R}^2}^2  \leq \langle \rho^d(v-f_k) \rangle_{\mathbb{R}^2,N}^2 +{\textcolor{black}{\frac{\gamma}{2}}} \langle \rho_1(\nabla v + (\nabla v)^T-\nabla f_k -\nabla f_k^T \\
- \nabla f_k^T \nabla f_k)\rangle_{M_2(\mathbb{R}),N_1}^2
 +\epsilon |v|_{3,\Omega,\mathbb{R}^2}^2, 
 \end{array} \right. . \label{problem_convergence_1}
\end{align}
\begin{theorem}[Convergence]
 For any $d\in D$, we denote by $u_\epsilon^d$ the unique solution to problem \ref{problem_convergence_1} for $\epsilon$ fixed. Then under the above assumptions, there exists a subsequence $(u_\epsilon^{d_l})$ with $\underset{l\rightarrow +\infty}{\lim}\,d_l=0$ such that
 \begin{align*}
  u_\epsilon^{d_l} \underset{l\rightarrow +\infty}{\rightharpoonup} f_k
 \end{align*}
in $H^3(\Omega,\mathbb{R}^2)$, and $\underset{k\rightarrow +\infty}{\lim}\,\underset{l\rightarrow +\infty}{\lim}\, \|u_\epsilon^{d_l} - u\|_{1,\infty}=0$.
\label{thm:convergence_approx_1}
\end{theorem}
\begin{proof}
 The detailed proof is available in Section 9 of the supplementary material.
\end{proof}
\textcolor{black}{An alternative convergence study is given in Section 10 of the supplementary material.}
\subsection{Numerical Resolution of the Third Approach}\label{sub-sec:numerical_resolution}
We now turn to the discretisation of the variational problem associated with \ref{thm:existence_uniqueness_min_approx} {\textcolor{black}{in which $u$ is a substitute for $f_k$}}. To do so, we use standard notations of the finite element theory similar to those in \cite{debroux-bib:arcangeli,debroux-bib:ciarlet}. Let $\mathcal{H}$ be an open bounded subset of $]0,+\infty[$ admitting 0 as accumulation point. Let us recall that the elements of class $\mathcal{C}^{k'}$ can be used for the computation of discrete $D^m$-splines (in our case, $m=3$) with $m \leq k'+1$. As a consequence, $(k',m)=(2,3)$ is a suitable combination. For all $n\in \mathbb{N}$ and for all subsets $E$ of $\mathbb{R}^2$, $Q_l(E)$ denotes the space of the restrictions of $E$ of the polynomial functions over $\mathbb{R}^2$ of degree $\leq l$ with respect to each variable. $\forall h \in \mathcal{H}$, let $(V_h)^2$ be the subspace of $H^3(\Omega,\mathbb{R}^2)$ of finite dimension with $(V_h)^2 {\hookrightarrow}\mathcal{C}^1(\bar{\Omega},\mathbb{R}^2)$. The reference finite element is the Bogner-Fox-Schmit $\mathcal{C}^2$ rectangle denoted by $(K,P_K,\Sigma_K)$.\newline
Let $(v^q)_{q=1,2}$ be the components of $v\in H^3(\Omega,\mathbb{R}^2)$ and $w_i=\begin{pmatrix}w_{i,11}&w_{i,12}\\w_{i,21}=w_{i,12}&w_{i,22} \end{pmatrix}:=\nabla u(a_i)+\nabla u(a_i)^T+\nabla u(a_i)^T\nabla u(a_i)$, $\forall i \in \left\{1,\cdots,N\right\}$. Let also $(x_i^q)_{q=1,2}$ be the components of $u(a_i)$, $\forall i \in \left\{1,\cdots,N\right\}$. Now let $M_h$ be the dimension of $V_h$ and $\left\{P_j^h\right\}_{j=1,\cdots,M_h}$ be basis functions. If we denote by $u_{\epsilon}^h$ the solution of the variational problem associated with \ref{thm:existence_uniqueness_min_approx} and approximated in $\left(V_h\right)^2$, we can thus decompose $u_{\epsilon}^h=\left(u_{\epsilon}^{h,q}\right)_{q=1,2}$ into:
\begin{align*}
\forall q =1,2,\,\,\,\exists \left(\alpha_j^q\right)_{j=1,\cdots,M_h}\in \mathbb{R},\,\,\,u_{\epsilon}^{h,q}=\displaystyle{\sum_{j=1}^{M_h}}\,\alpha_j^q\,P_j^h.
\end{align*}
Denoting by $A^h=\left(\dfrac{\partial P_j^h}{\partial x}(a_i)\right)_{1\leq i \leq N \atop 1\leq j \leq M_h}$, $B^h=\left(\dfrac{\partial P_j^h}{\partial y}(a_i)\right)_{1\leq i \leq N \atop 1\leq j \leq M_h} \in \left(M_{N\times M_h}(\mathbb{R})\right)^2$, $C^h=\left(P_{j}^h(a_i)\right)_{1\leq i \leq N \atop 1\leq j \leq M_h} \in M_{N\times M_h}(\mathbb{R})$ and $R^h=\left(\left(P_j^h,P_i^h\right)_{3,\Omega,\mathbb{R}}\right)_{1\leq i,j \leq M_h}$ and taking successively in the variational problem
$v=\begin{pmatrix}
P_l^h\\
0
\end{pmatrix}$, $l=1,\cdots,M_h$ and then $v=\begin{pmatrix}
0\\
P_l^h
\end{pmatrix}$, $l=1,\cdots,M_h$,
the problem amounts to solving the following linear system 
\begin{align*}
\begin{blockarray}{ccc}
 & \small{M_h} & \small{M_h}  \\
  \begin{block}{c(c|c)}
  &&\\
   \small{M_h} & 2\gamma\,(A^h)^TA^h+\gamma\,(B^h)^TB^h\atop +(C^h)^TC^h+\epsilon\, \textcolor{black}{R^h} & \gamma\,(B^h)^TA_h  \\ 
   &&\\ 
   \cline{2-3}
   &&\\   
  \small{M_h} & \gamma\,(A^h)^TB_h&\gamma\,(A^h)^TA^h+2\gamma\,(B^h)^TB^h\atop +(C^h)^TC^h+\epsilon\, \textcolor{black}{R^h} \\
  \end{block}
\end{blockarray}
\quad
\begin{pmatrix}\\
\alpha^1\\
\\
\hline
\\
\alpha^2
\\
\end{pmatrix}
=
\begin{pmatrix}\\
\gamma\,(A^h)^Tw_{11}+\gamma\,(B^h)^Tw_{12}\atop+(C^h)^Tx^1\\
\\
\hline
\\
\gamma\,(A^h)^Tw_{12}+\gamma\,(B^h)^Tw_{22}\atop+(C_h)^Tx^2
\\
\end{pmatrix}
\end{align*}
\textcolor{black}{\begin{remark} A classical PCA is then performed on the obtained displacement fields using the $L^2$ scalar product for the covariance operator, i.e. $(C_{i,j})_{i=1,\cdots,M\atop j=1,\cdots,M} = \int_\Omega  v_{1,j} v_{1,i} + v_{2,j} v_{2,i}\,dx $, where $\left( v_{1,i} \,\, v_{2,i}\right)^T$ for $i=1,\cdots,M$ are the displacement fields obtained in a linear space for each image, and the resulting displacement fields are denoted by $\left( v_{1,pca,j}\,\,v_{2,pca,j}\right)^T$, where $j$ stands for the mode indexation.\end{remark}}
\section{Numerical Simulations}\label{sec:numerical_simulations}
\subsection{General framework}\label{subsec:general_framework}
{\textcolor{black}{This section is devoted to the analysis of numerical experiments. First, on a set of 19 binary images `{\it{device8-1}}' from the MPEG7 shape database (\url{http://www.dabi.temple.edu/~shape/MPEG7/dataset.html}) and then on medical images: cardiac MRI made of 8 frames per slice of size 150x150, the first image of the sequence reflecting the case where the heart is most dilated (end diastole - ED), while the last one illustrating the case where the heart is most contracted (end systole - ES), and liver dynamic MRI made of 14 frames per slice of size 195x166.\newline
The computations have been made on an \textcolor{black}{Intel Core i7} computer with \textcolor{black}{2.60}\,GHz and \textcolor{black}{8} GB memory, using \textcolor{black}{MUMPS} packages\textcolor{black}{, in a C implementation}.\newline The question of assessing the proposed model encompasses several angles of inquiry:
\begin{itemize}
\item[(i)] the qualitative evaluation of the obtained atlas in comparison to a sequential treatment of the segmentation and registration tasks (note that the model requires the segmentation step be processed first since it involves the penalisation $\|\theta_{\tilde{T}_i}-\theta_{T_i}\circ \varphi_i+\theta_R\|_{L^2(\Omega)}^2$).\newline
We do not question here the relevancy of the regularisation on the $\varphi_i$'s since in practice, the $\varphi_i$'s need to be invertible (as in the expression $\theta_{T_i}\circ \varphi_i$, $\theta_{T_i}$ is also an unknown), which is guaranteed with the proposed regulariser. As the average shape is an unknown of the problem and ground truth is not provided, the evaluation of the method itself primarily relies on visual inspection and empirical arguments {\textcolor{black}{consistent with the biological phenomena involved}};
\item[(ii)] the evaluation of the PCA in capturing strongly nonlinear geometric variations.
\end{itemize}
These two main levels of discussion dictate the structure of the section. Each subsection focuses on a specific dataset and provides both the atlas generated by our joint model, together with the first principal modes of variation. The gain of the combined approach in terms of sharp edges and in removing ghosting artefacts {\textcolor{black}{(blurring / splitting into two effects)}} is emphasised compared to a sequential treatment, as well as the accuracy of our proposed model in reflecting the high nonlinear geometric variations in comparison to the two more physically/mechanically oriented methods. {\textcolor{black}{This stage is delicate since as stated below, it requires comparing the three methodologies in the most efficient and impartial manner, and subsequently, setting the involved tuning parameters (e.g., the weight balancing the loading forces for the Cauchy stress tensor based PCA) adequately. That said, results demonstrate nevertheless and as expected, that linearised elasticity produces very small displacements, which tends to favor comparisons with the Cauchy stress tensor based PCA.}} \newline
For the sake of reproducibility, we provide in Table~\ref{parameters} the values of the tuning parameters. The coefficients $a_1$, $a_2$ and $a_3$ involved in the Ogden stored energy function affect respectively the averaged local change of length, \textcolor{black}{and the averaged local change of area}, impacting subsequently on the rigidity of the deformation. The higher the $a_i$'s are, the more rigid the deformation is. The ranges of these parameters are rather stable for the medical experiments. Parameters $\gamma_T$ and $\gamma_R$ \textcolor{black}{weighting the $L^0$ norm} control thus the balance between the $L^0$-component and the $L^2$-penalisation \textcolor{black}{together with $\lambda_R$ and $\lambda_T$}. \textcolor{black}{High} values of $\gamma_T/\gamma_R$ favor few large partitions, while \textcolor{black}{small} values yield an approximation exhibiting more jumps. {\textcolor{black}{$\lambda_R$, $\gamma_{\tilde T}$}}, and {\textcolor{black}{$\gamma_1$}} weight the fidelity term in the registration task, and thus the higher they are, the closer the deformed templates and deformed segmentations are to the mean segmentation.}}}
{\textcolor{black}{The visualisation of the main modes of variation is done as follows.}} {\textcolor{black}{We denote by $\left(v_{1,pca_1,i}\,\, v_{2,pca_1,i}\right)^T$ the $i$\textsuperscript{th} resulting displacement field (related to the $i$\textsuperscript{th} mode of variation) from the first method based on the linearisation of the stored energy function around the identity, providing in practice and as expected extremely small displacements. $\sigma_{pca_2,i}$ represents the $i$\textsuperscript{th} resulting Cauchy tensor from the second method and $\left(v_{1,pca_2,i,\delta} \,\, v_{2,pca_2,i,\delta}\right)^T$ the associated displacement field obtained with weighting parameter $\delta$ balancing the inner forces chosen equal to $1, 2, 0.3$ for the T-shape, the liver and the heart respectively, that prove to be suitable parameters for an unbiased analysis. At last,  $\left( v_{1,pca_3,i} \,\, v_{2,pca_3,i}\right)^T$ stands for the $i$\textsuperscript{th} resulting displacement field from the third method based on approximation theory. We propose visualising the $i$\textsuperscript{th} mode of variation by showing $\theta_R\circ (\mathrm{Id}+5c.10^6\begin{pmatrix}v_{1,pca_1,i}\\v_{2,pca_1,i}\end{pmatrix})$, $\theta_R\circ (\mathrm{Id}+5c\begin{pmatrix}v_{1,pca_2,i,\delta}\\v_{2,pca_2,i,\delta}\end{pmatrix})$, and $\theta_R\circ (\mathrm{Id}+50c\begin{pmatrix}v_{1,pca_3,i}\\v_{2,pca_3,i}\end{pmatrix})$, for each method respectively, with $c$ varying from $-5$ to $5$. The parameters are chosen in order to make the comparison as fair as possible.}} 
The computation times for each method and each example are provided in Table \ref{tab:execution_time}}.
\begin{table}[htp]
\centering
\begin{tabular}{|c||c|c|c|c|c|c|c|c|c|c|c|c|c|c|}
\hline 
&\scriptsize{$\scriptscriptstyle a_1$}&\scriptsize{$\scriptscriptstyle a_2$}&\scriptsize{$\scriptscriptstyle a_3$}&\scriptsize{$\scriptscriptstyle \gamma_1$}&\scriptsize{$\scriptscriptstyle \gamma_2$}&\scriptsize{$\scriptscriptstyle \gamma_3$}&\scriptsize{$\scriptscriptstyle \alpha$}&\scriptsize{$\scriptscriptstyle \beta$}&\scriptsize{$\scriptscriptstyle \gamma_R$}&\scriptsize{$\scriptscriptstyle \gamma_T$}&\scriptsize{$\scriptscriptstyle \lambda_T$}&\scriptsize{$\scriptscriptstyle \lambda_R$}&\scriptsize{$\scriptscriptstyle dt$}&\tiny{$\scriptscriptstyle {\mbox{nbIter}}$}\\ \hline 
\hline
\scriptsize{T-shape}&\scriptsize{$\scriptscriptstyle 1$}&\scriptsize{$\scriptscriptstyle 5.10^3$} &\scriptsize{$\scriptscriptstyle 0.01$}&\scriptsize{$\scriptscriptstyle 1$}&\scriptsize{$\scriptscriptstyle 8.10^{4}$}&\scriptsize{$\scriptscriptstyle 1$}&\scriptsize{$\scriptscriptstyle 10$}&\scriptsize{$\scriptscriptstyle 100$}&\scriptsize{$\scriptscriptstyle 3$}&\scriptsize{$\scriptscriptstyle 0.5$}&\scriptsize{$\scriptscriptstyle 1$}&\scriptsize{1}&\scriptsize{$\scriptscriptstyle 0.001$}&\scriptsize{$\scriptscriptstyle 100$}\\ \hline
\hline 
\tiny{Heart ED(108)-ES(101)}&\scriptsize{$\scriptscriptstyle 5$}&\scriptsize{$\scriptscriptstyle 1.10^{3}$} &\scriptsize{$\scriptscriptstyle 4$}&\scriptsize{$\scriptscriptstyle 1$}&\scriptsize{$\scriptscriptstyle 8.10^4$}&\scriptsize{$\scriptscriptstyle 1$}&\scriptsize{$\scriptscriptstyle 100$}&\scriptsize{$\scriptscriptstyle 100$}&\scriptsize{$\scriptscriptstyle 0.02$}&\scriptsize{$\scriptscriptstyle 0.03$}&\scriptsize{$\scriptscriptstyle 1.5$}&\scriptsize{$\scriptscriptstyle 1.5$}&\scriptsize{$\scriptscriptstyle 0.01$}&\scriptsize{$\scriptscriptstyle 500$}\\ \hline
\hline 
\tiny{Liver - slice 12 }&\scriptsize{$\scriptscriptstyle 5$}&\scriptsize{$\scriptscriptstyle 1.10^3$} &\scriptsize{$\scriptscriptstyle 4$}&\scriptsize{$\scriptscriptstyle 1.5$}&\scriptsize{$\scriptscriptstyle 8.10^4$}&\scriptsize{$\scriptscriptstyle 1$}&\scriptsize{$\scriptscriptstyle 10$}&\scriptsize{$\scriptscriptstyle 100$}&\scriptsize{$\scriptscriptstyle 0.05$}&\scriptsize{$\scriptscriptstyle 0.05$}&\scriptsize{$\scriptscriptstyle 1.5$}&\scriptsize{$\scriptscriptstyle 1.5$}&\scriptsize{$\scriptscriptstyle 0.01$}&\scriptsize{$\scriptscriptstyle 100$}\\ \hline
\hline 
\end{tabular}
\ \\[0.3cm]
\caption{Parameters.}
\label{parameters}
\end{table}
\begin{table}[htp]
    \centering
    \begin{tabular}{|c||c|c|c|c|}
    \hline
         \scriptsize{Execution time}&\scriptsize{Atlas generation}&
         \scriptsize{PCA $1^{\text{st}}$ method}&\scriptsize{PCA $2^{\text{nd}}$ method}&\scriptsize{PCA $3^{\text{rd}}$ method}\\ \hline  \hline
         \scriptsize{T-shape}&7 min&
         2 sec&4 min&44 sec \\ \hline
         \scriptsize{Heart ED(108)-ES(101)}&49 min&
         3 sec&10 min&3 min 27 sec \\ \hline
         \scriptsize{Liver slice 12}&26 min&
         6 sec&7 min 28 sec&13 min \\ \hline
 \hline
    \end{tabular}
    \ \\[0.3cm]
    \caption{Execution times}
    \label{tab:execution_time}
\end{table}
\subsection{T-shape example: 19 images}\label{sub-sec:T-shape}
{\textcolor{black}{The proposed method is first evaluated on a synthetic example (Figure \ref{fig:input_images}) 
to emphasise the ability of the model to generate large deformations and to produce \textcolor{black}{a physically sound average shape}. As depicted in Figure~\ref{fig:results_T-shape_comparaison_atlas}, contrary to a sequential treatment of the tasks, the joint model creates a mean object with sharp edges perfectly matched by the deformed templates, \textcolor{black}{see particularly \textcolor{black}{the bottom of the T shape}}. \textcolor{black}{The joint approach also tends to better preserve the original contrast of the images than the sequential approach.} The principal component analysis (Figure \ref{fig:results_T-shape_comparaison_pca}) shows that the first two modes of variation have an effect on the undulation of both the vertical and horizontal bars \textcolor{black}{(the first mode tends to represent the undulation in the direction \textit{bottom-left corner-top-right corner} whereas the second one tends to capture the undulation in the direction \textit{top-left corner-bottom-right corner})}, while the third mode affects the thickness of the vertical bar particularly at the junction. The fourth mode, for its part, acts more locally on the curvature of the envelope of the shape, especially on the lower part of the horizontal bar that exhibits cavities. A first observation is that the proposed model allows uncorrelating the main tendencies, which is what is expected from such an analysis. It seems (Figure \ref{fig:results_T-shape_comparaison_pca_3_methods}) that this decoupling property is not as well exemplified when applying the linearisation around identity \textcolor{black}{which is not able to recover the nonlinear variations coming from the undulation} or the Cauchy stress tensor based PCA \textcolor{black}{mixing the first, the third and the fourth modes of variation from our third approach based on approximation modelling}.
  }}
  \begin{figure}
      \centering
      \includegraphics[width=\linewidth]{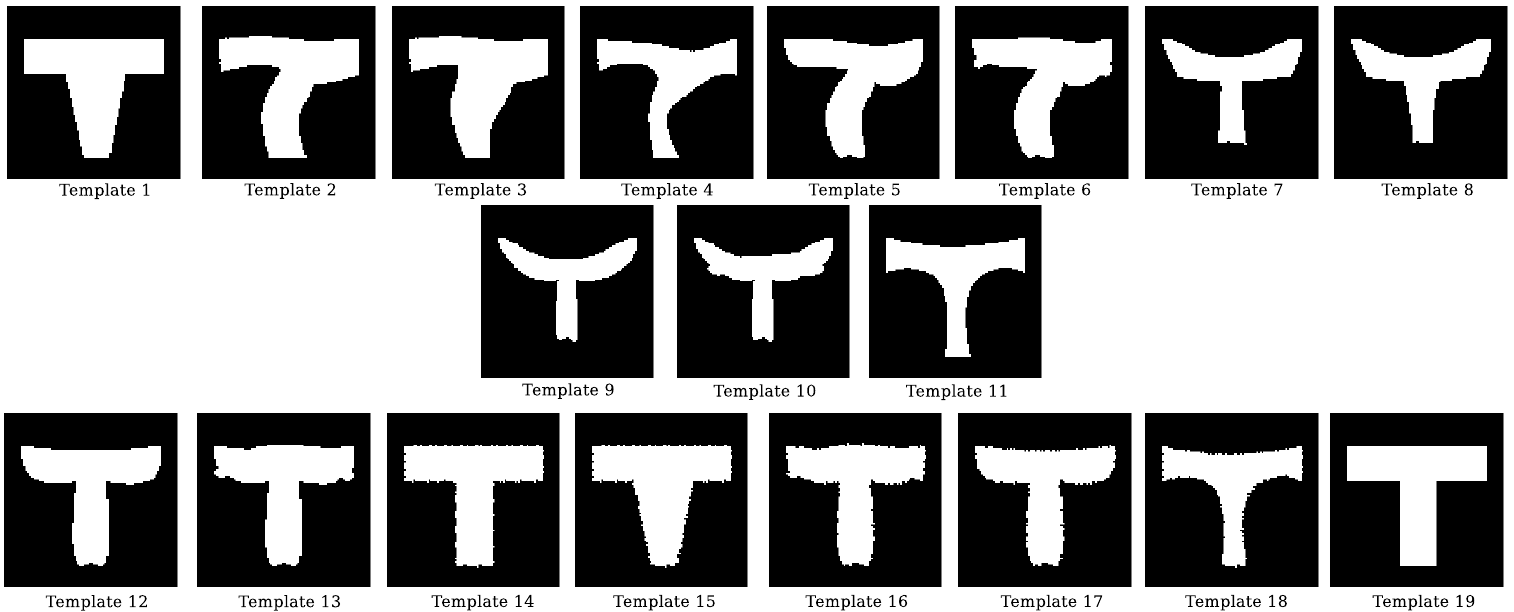}\vspace{-0.5cm}
      \caption{Input images.}
      \label{fig:input_images}
  \end{figure}
\begin{figure}[htp]
    \centering
    \includegraphics[width=\linewidth]{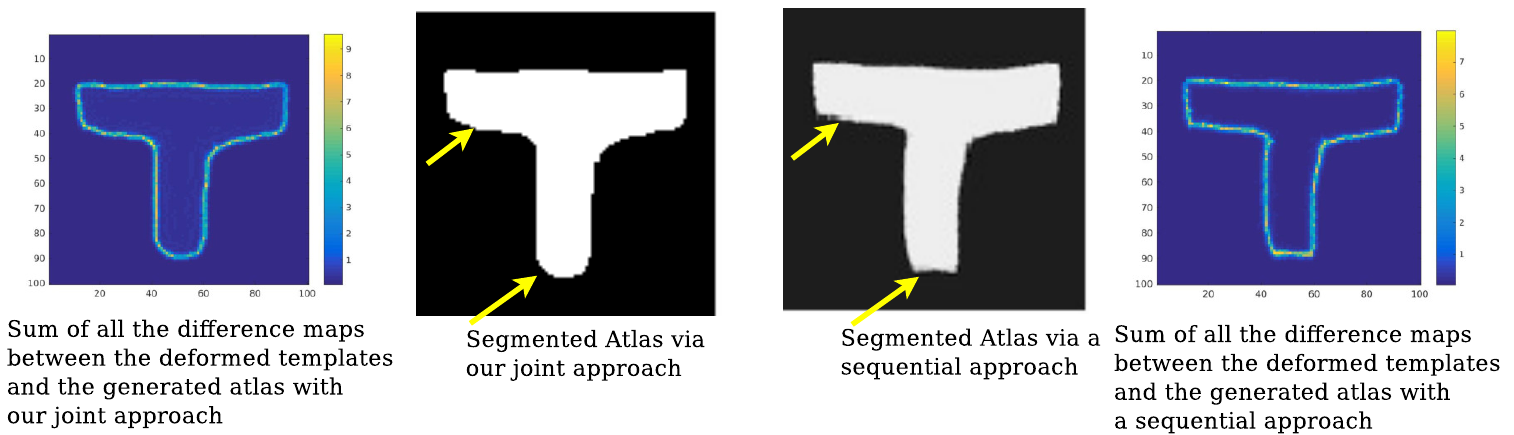}\vspace{-0.5cm}
    \caption{Comparison of the atlases generated by our joint model, and by a sequential approach with our discrepancy measure. \textcolor{black}{Visual assessment in terms of blurring artefacts, and contrast are pointed out with yellow arrows}.}
    \label{fig:results_T-shape_comparaison_atlas}
\end{figure}
\begin{figure}[htp]
    \centering
    \includegraphics[width=\linewidth]{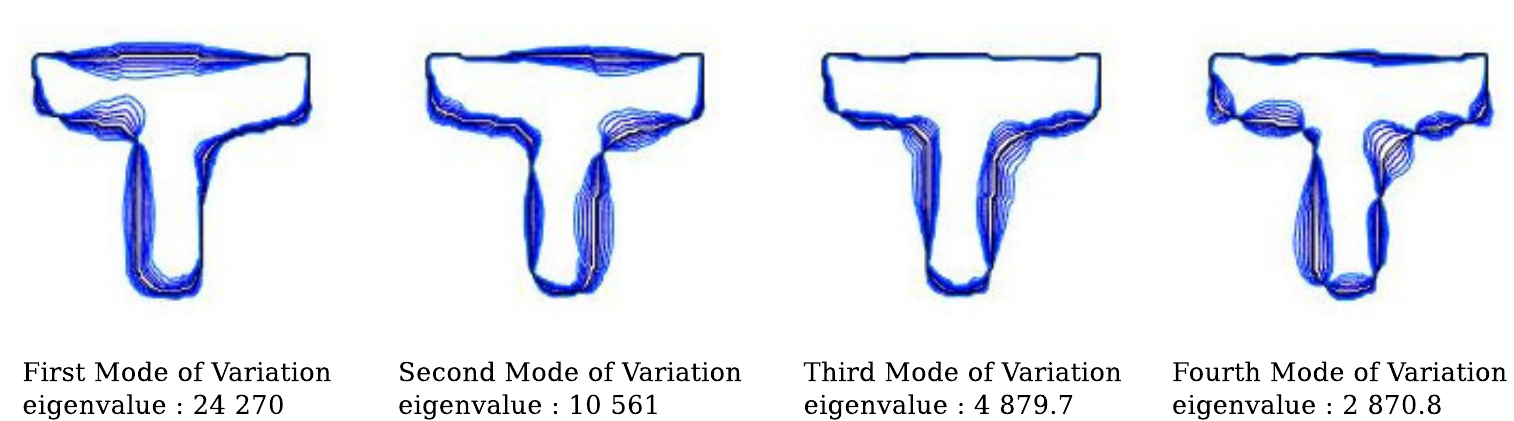}\vspace{-0.5cm}
    \caption{First four modes of variation obtained with our method based on approximation modelling via contour representations.}
    \label{fig:results_T-shape_comparaison_pca}
\end{figure}
\begin{figure}[htp]
    \centering
    \includegraphics[width=\linewidth]{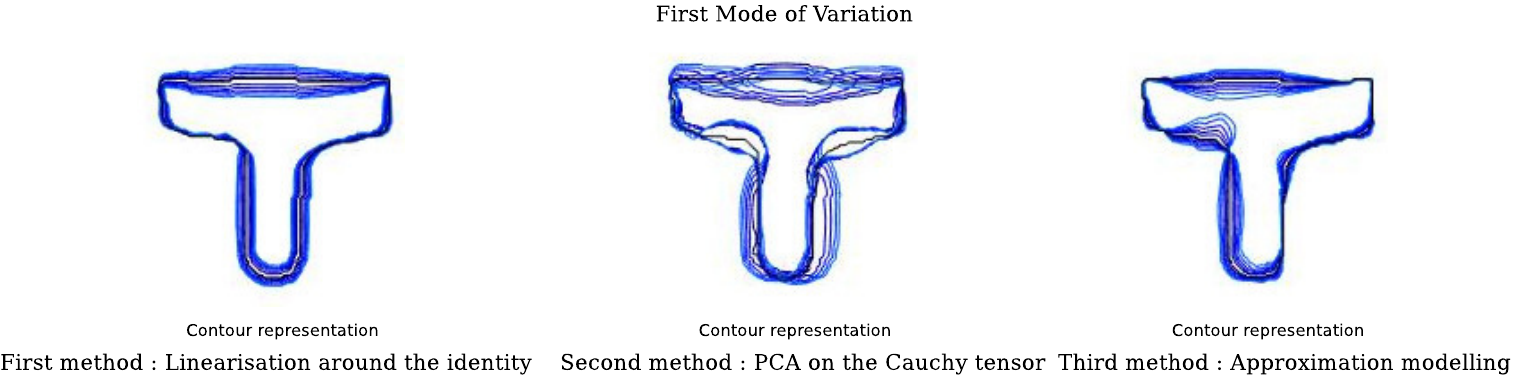}\vspace{-0.5cm}
    \caption{Comparison of the first mode of variation from three different methods using contour representations.}
    \label{fig:results_T-shape_comparaison_pca_3_methods}
\end{figure}
\subsection{Liver example: Slice 12 \cite{debroux-bib:siebenthal}}\label{sub-sec:liver_slice12}
{\textcolor{black}{The second example is dedicated to \textcolor{black}{right lobe} liver dynamic MRI \url{http://www.vision.ee.ethz.ch/~organmot/chapter_download.shtml}, the shape of this organ being influenced by the surrounding structures such as the diaphragm (\cite{debroux-bib:von_siebenthal}). It exemplifies the ability of the method to alleviate ghosting artefacts, in particular splitting into two effects, through parameter $\gamma_R$ that influences the number of phases (Figure \ref{fig:results_liver_12_joint}). The obtained atlas exhibits sharp edges with fewer artefacts (Figure \ref{fig:results_liver_12_comparaison_atlas}) than with a sequential approach. The statistical analysis is then performed (Figure \ref{fig:results_liver_12_comparaison_pca}). The first mode of variation encodes the motion in superior/inferior direction of the liver, which is consistent with the physics (\cite{debroux-bib:von_siebenthal}): respiration is largely governed by the diaphragm, and the liver, located beneath the diaphragm, is thus strongly influenced by \textcolor{black}{breathing} (pushed downwards when the diaphragm is contracted and upwards when expanded). The same applies to the kidney located in the right lower part. As for the second mode of variation, it reflects the transversal motion, particularly visible on the right excrescences exhibited by the liver \textcolor{black}{and the left wall}. A comparison of the first mode of variation obtained with both the linearisation around the identity and the Cauchy stress tensor based PCA shows that our method reflects the geometric variability slightly better in the superior/inferior direction (Figure \ref{fig:results_liver_12_comparaison_pca_3_methods}) particularly in the right lower part of the kidney.
}}
\begin{figure}[htp]
    \centering
    \includegraphics[width=\linewidth]{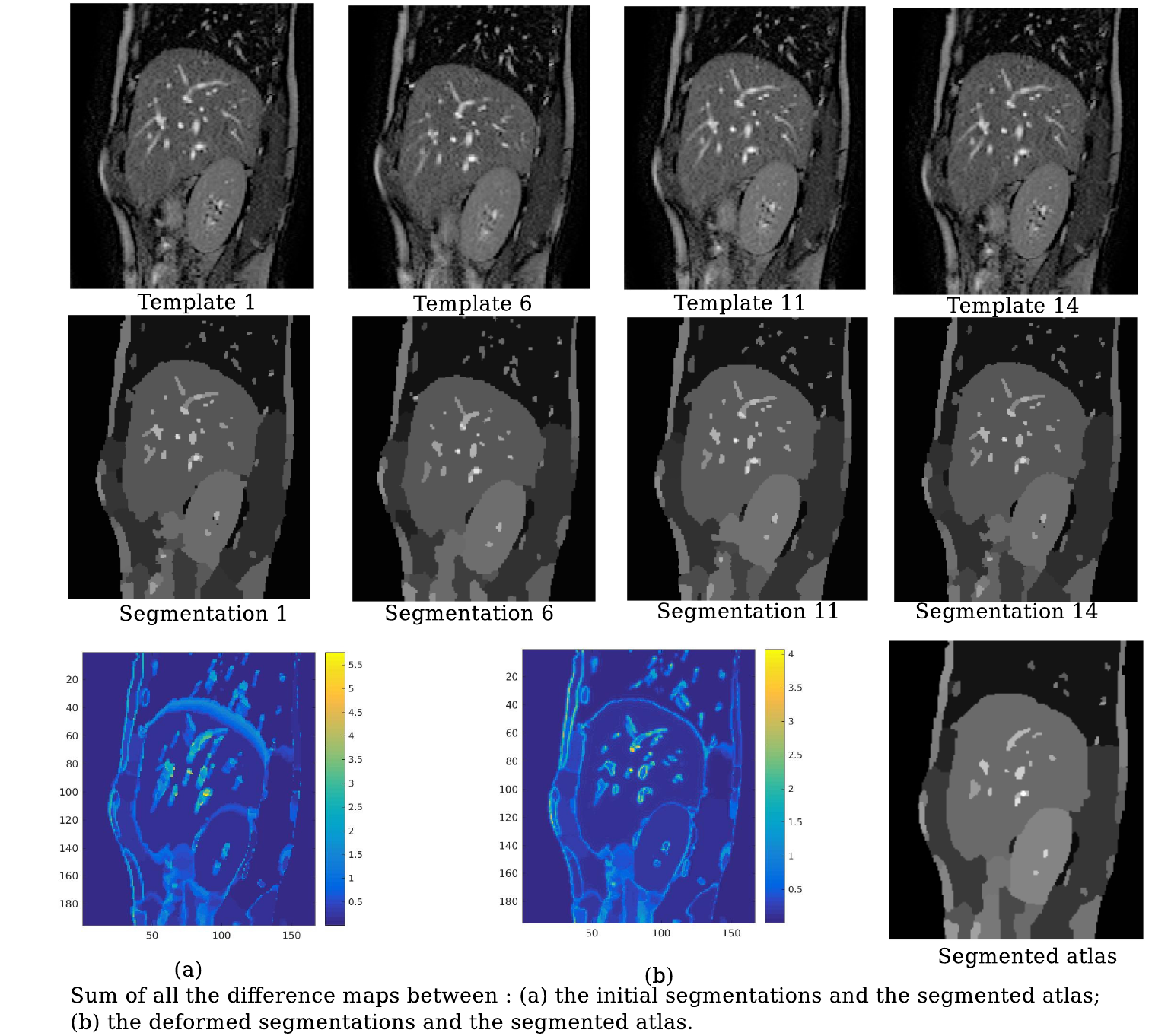}\vspace{-0.5cm}
    \caption{Some of the initial slices (1,6,11,14) with the associated segmentations given by our joint algorithm underneath. The last row displays the sum of difference maps between the initial segmentation and the segmented atlas, the deformed/registered segmentations and the segmented atlas, and the segmented atlas generated by our joint model.}
    \label{fig:results_liver_12_joint}
\end{figure}
\begin{figure}[htp]
    \centering
    \includegraphics[width=0.5\linewidth]{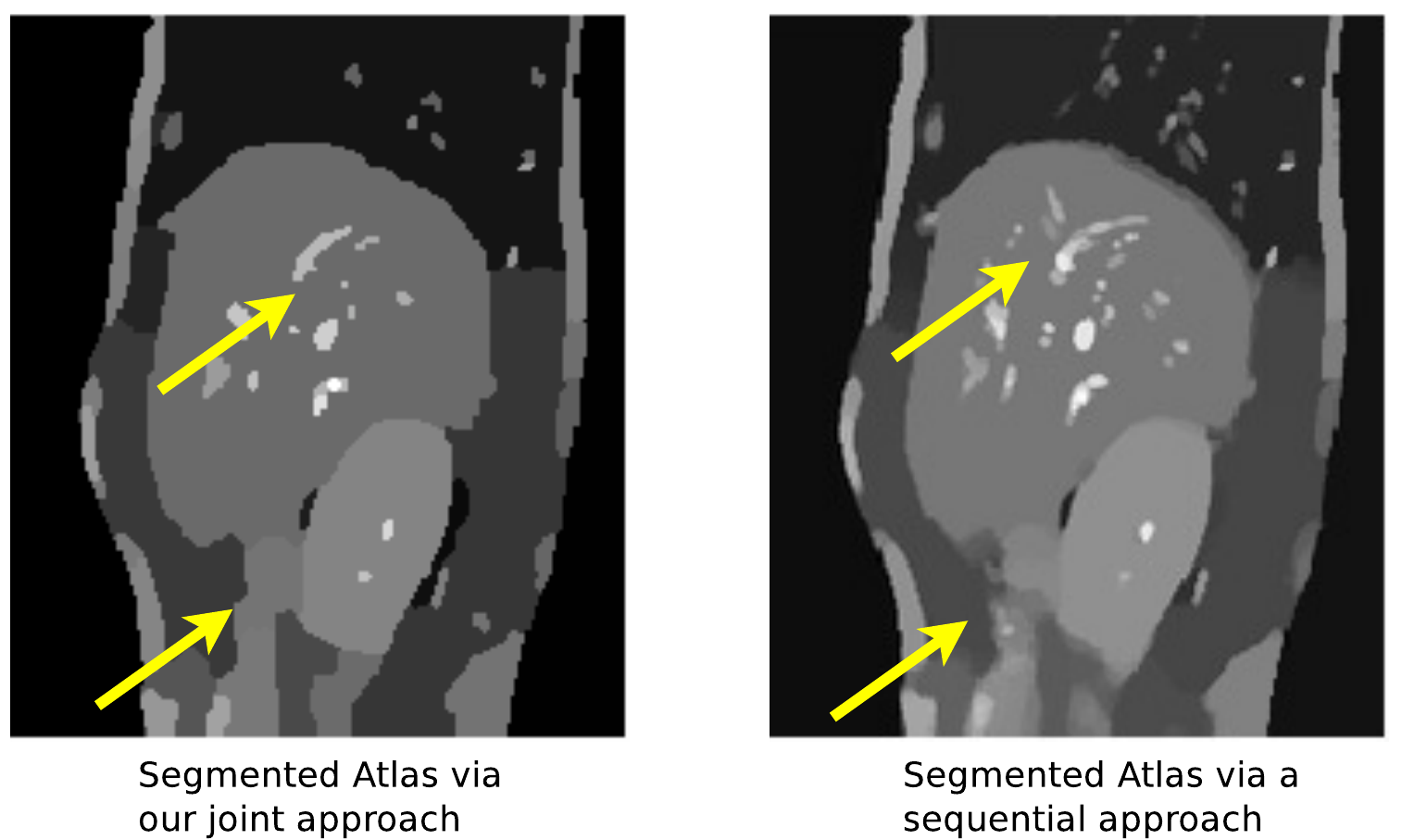}\vspace{-0.5cm}
    \caption{Comparison of the atlases generated by our joint model, and by a sequential approach with our discrepancy measure. \textcolor{black}{Visual assessment in terms of blurring and ghosting artefacts are indicated by yellow arrows}.}
    \label{fig:results_liver_12_comparaison_atlas}
\end{figure}
\begin{figure}[htp]
    \centering
    \includegraphics[width=\linewidth]{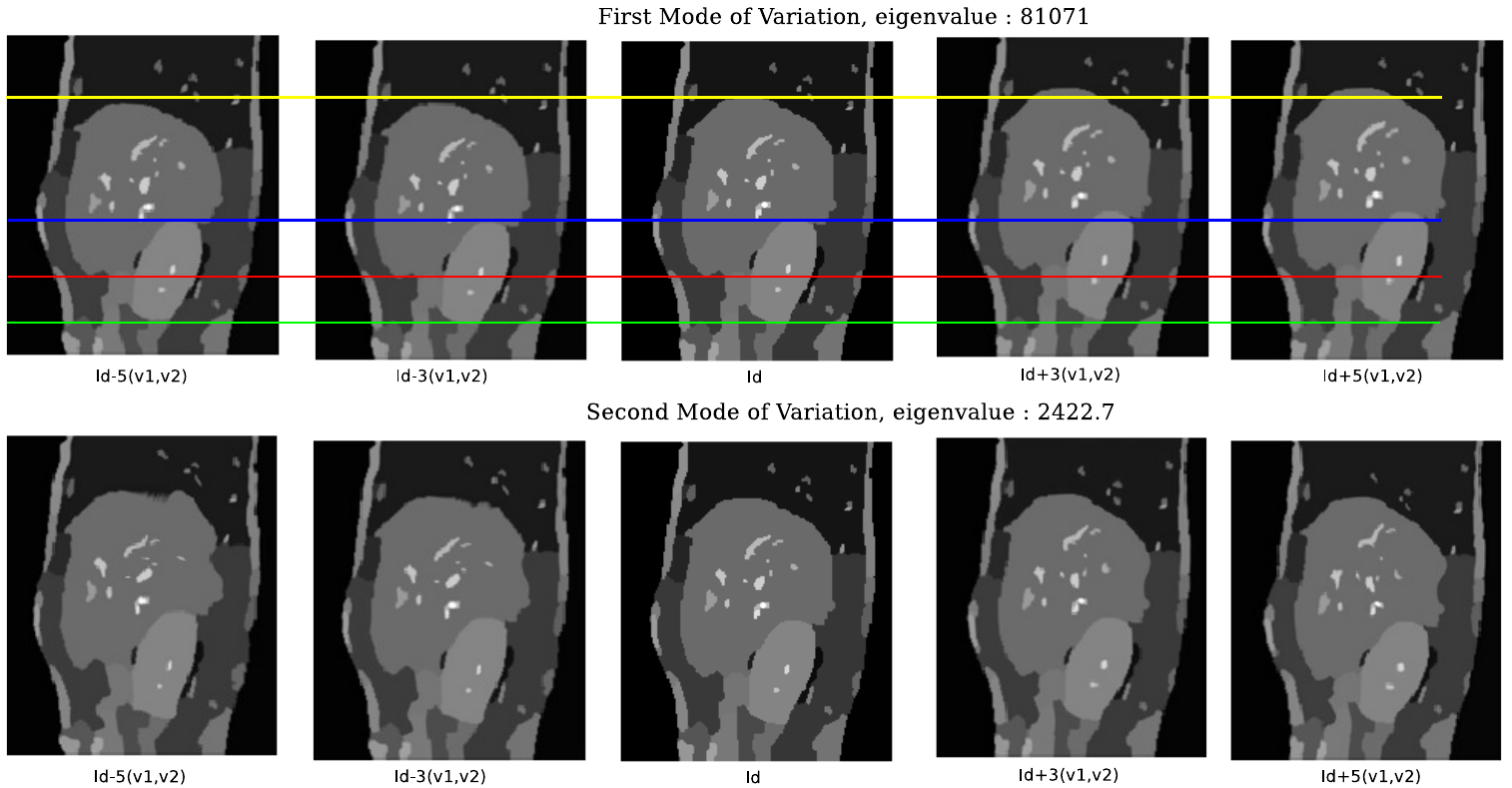}\vspace{-0.5cm}
    \caption{First two modes of variation obtained with our method based on approximation modelling. \textcolor{black}{The horizontal lines help visualise the vertical movement of the liver and the kidney.}}
    \label{fig:results_liver_12_comparaison_pca}
\end{figure}
\begin{figure}[htp]
    \centering
    \includegraphics[width=\linewidth]{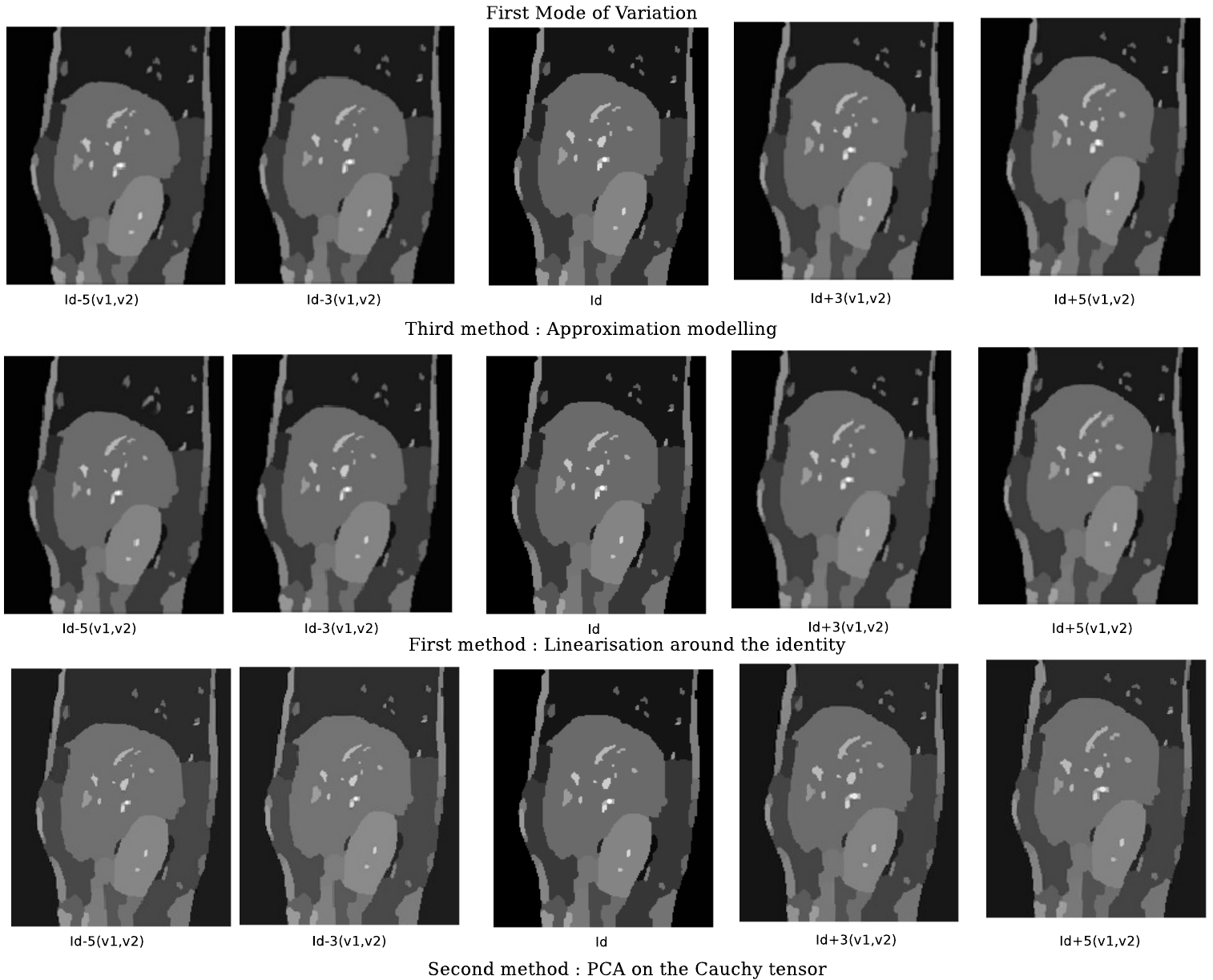}\vspace{-0.5cm}
    \caption{Comparison of the first mode of variation for 3 different methods.}
    \label{fig:results_liver_12_comparaison_pca_3_methods}
\end{figure}
\subsection{Heart example: frames ED(108)-ES(101)}\label{sub-sec:heart_101_108}
{\textcolor{black}{Lastly, the joint algorithm is applied to the set of MRI representing a cardiac cycle. The obtained segmentations (Figure \ref{fig:results_heart_100_108_joint}) are piecewise constant approximations of the initial images and reflect well the geometrical \textcolor{black}{shapes} of the template images. The deformed segmentations (involving large deformations) are close to each other and well aligned to the mean reference. The obtained atlas exhibits sharp edges, contrary to the result produced by a sequential treatment (Figure Figure \ref{fig:results_heart_100_108_comparaison_atlas}) that shows blurry artefacts. It is also anatomically consistent: the mean reference corresponds to a compromise between the full expansion and the full contraction, which is reasonable from a biological standpoint. The second step of the algorithm is then applied. The obtained first modes of variation are again consistent with the anatomical dynamic of the heart (Figure \ref{fig:results_heart_100_108_comparaison_pca}): while the first mode encodes the dilation/contraction of the right ventricular chamber in the transverse direction, the second mode of variation conveys the vertical stretching of the left ventricular chamber. Our approximation-based PCA thus allows to disconnect these two movements, this property being less visible in the case of the Cauchy stress tensor based PCA (Figure \ref{fig:results_heart_100_108_comparaison_pca_3_methods}) for which the first mode of variation encodes both tendencies.  
}}
\begin{figure}[htp]
    \centering
    \includegraphics[width=\linewidth]{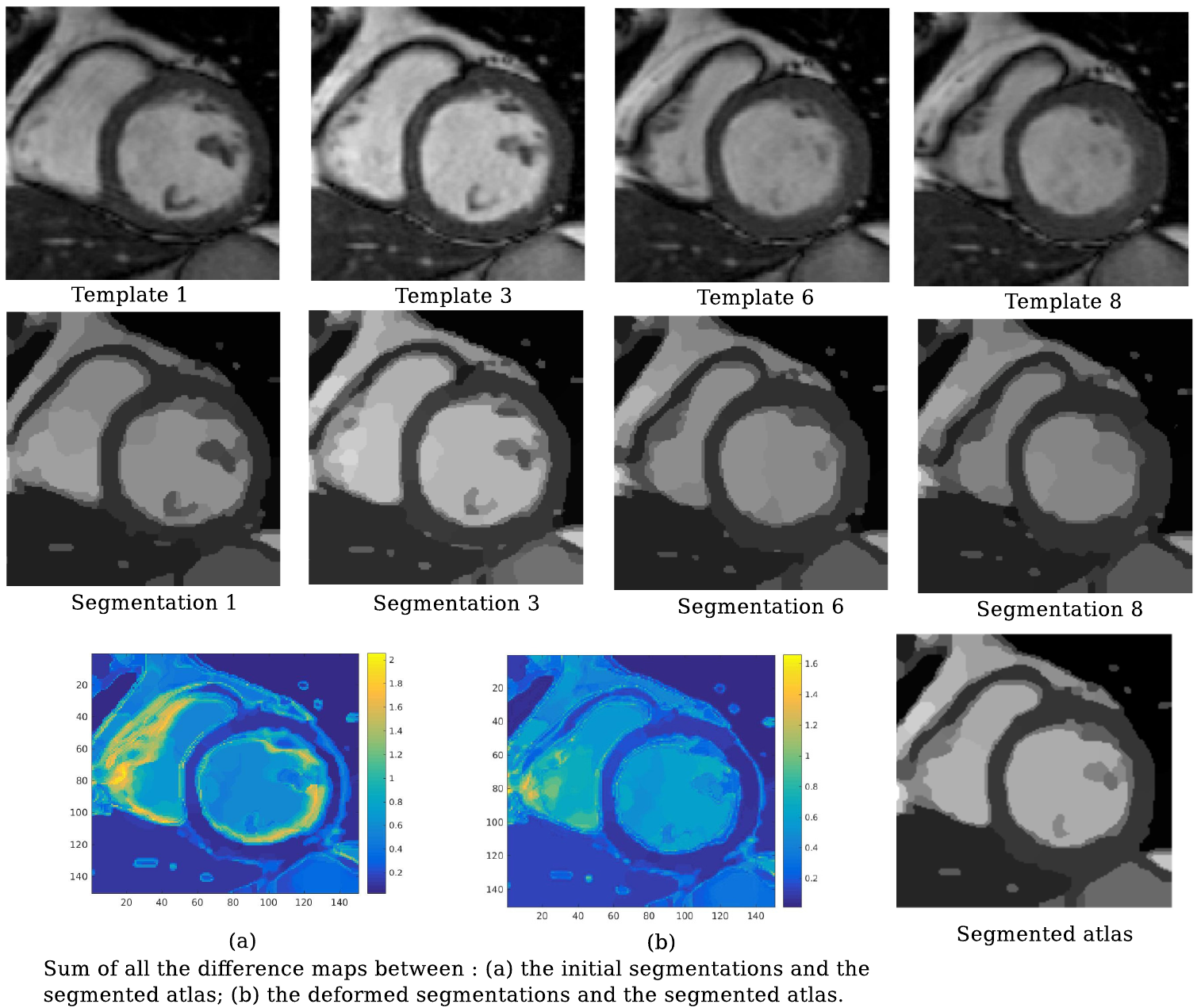}\vspace{-0.5cm}
    \caption{Some of the initial images (1,3,6,8) with the associated segmentations given by our joint algorithm underneath. The last line displays the sum of difference maps between the initial segmentation and the segmented atlas, the deformed/registered segmentations and the segmented atlas, and the segmented atlas generated by our joint model.}
    \label{fig:results_heart_100_108_joint}
\end{figure}
\begin{figure}[htp]
    \centering
    \includegraphics[width=0.5\linewidth]{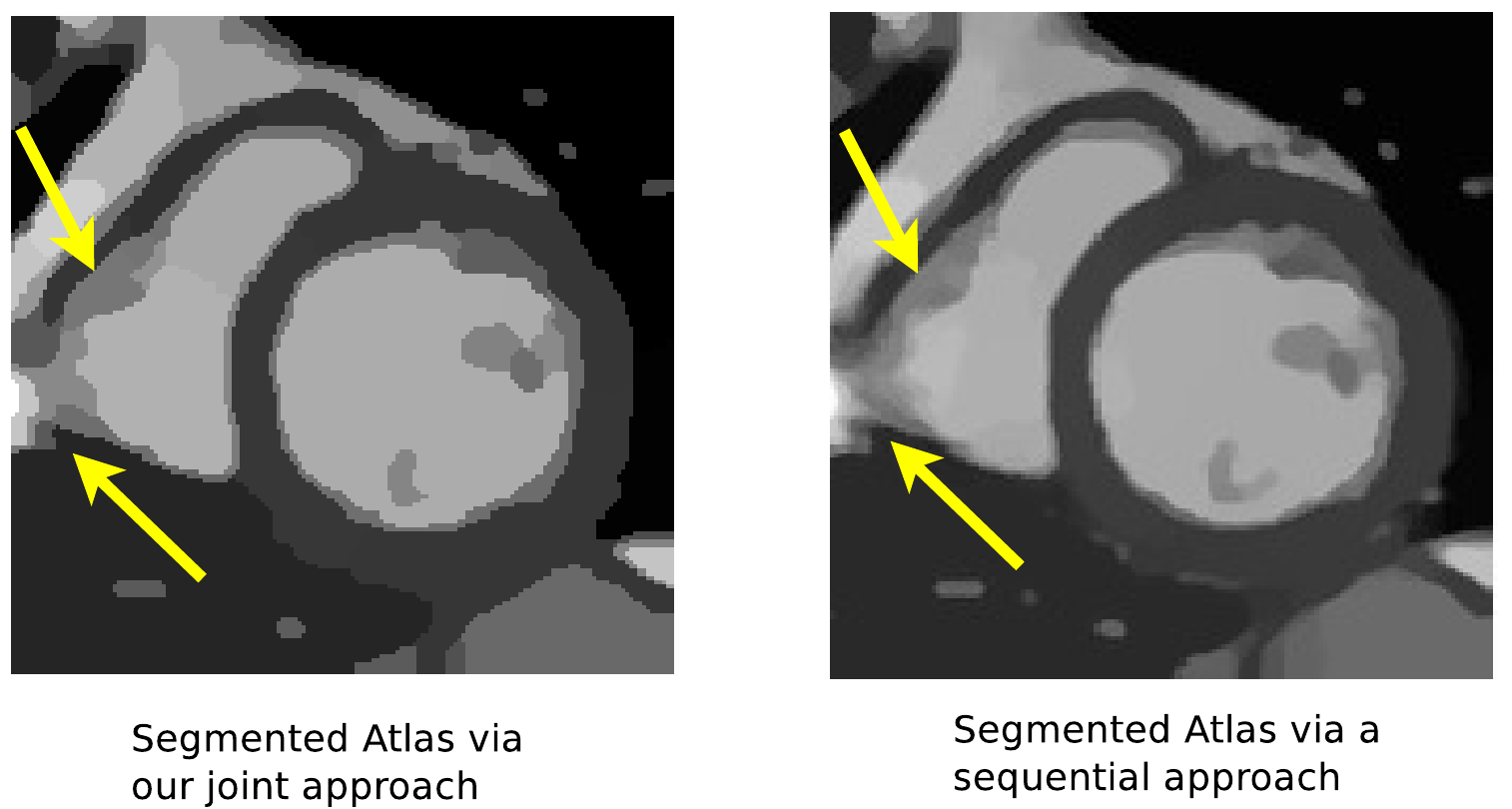}\vspace{-0.5cm}
    \caption{Comparison of the atlases generated by our joint model, and by a sequential approach with our discrepancy measure. \textcolor{black}{Visual assessment in terms of blurring artefacts are pointed out with yellow arrows}.}
    \label{fig:results_heart_100_108_comparaison_atlas}
\end{figure}
\begin{figure}[htp]
    \centering
    \includegraphics[width=\linewidth]{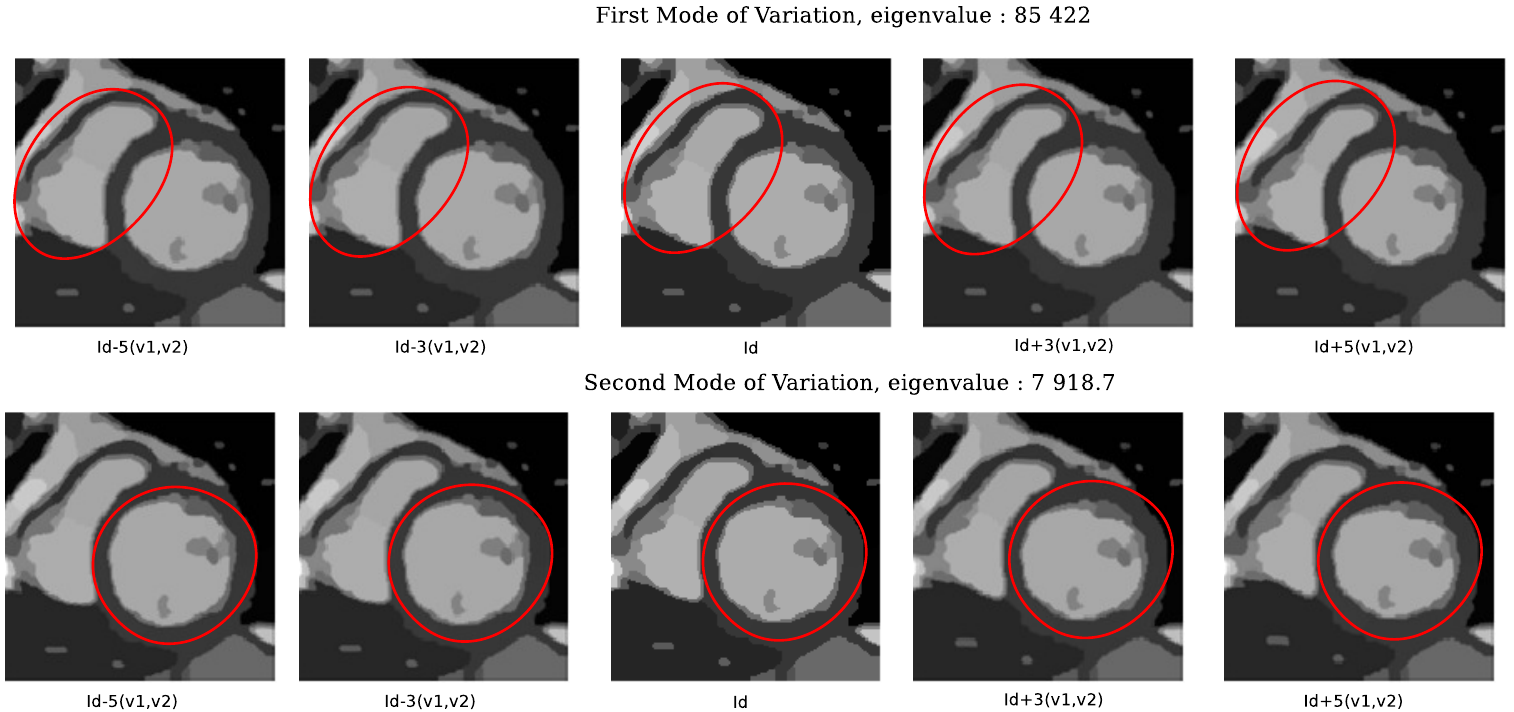}\vspace{-0.5cm}
    \caption{First three modes of variation obtained with our method based on approximation modelling. \textcolor{black}{The red circles shows the moving part: dilation/contraction of the right ventricular chamber for the first mode, and vertical stretching of the left ventricular chamber for the second mode.}}
    \label{fig:results_heart_100_108_comparaison_pca}
\end{figure}
\begin{figure}[htp]
    \centering
    \includegraphics[width=\linewidth]{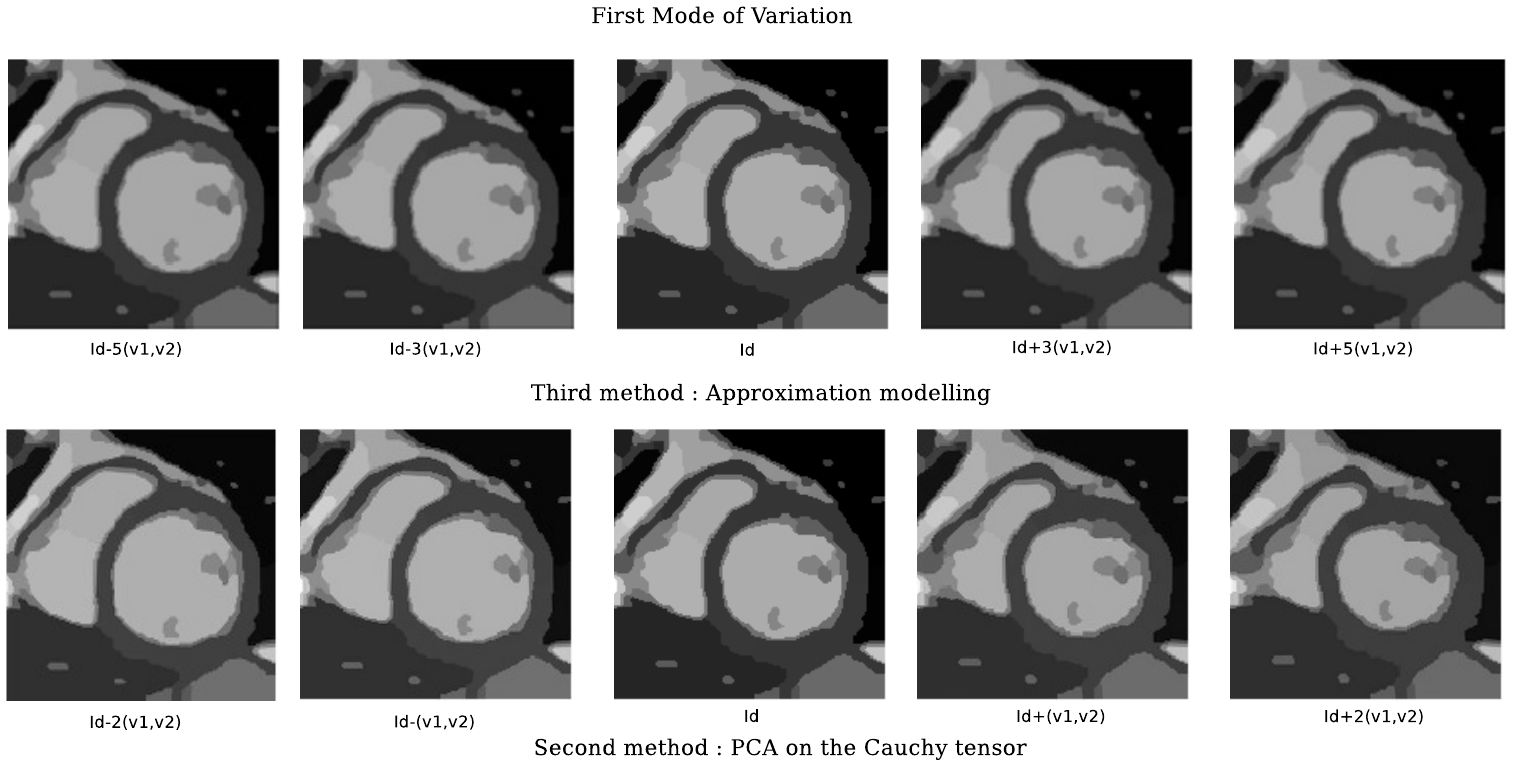}\vspace{-0.5cm}
    \caption{Comparison of the first mode of variation for 2 different methods.}
    \label{fig:results_heart_100_108_comparaison_pca_3_methods}
\end{figure}
\section{Conclusion}\label{sec:conclusion}
{\textcolor{black}{This paper addressed the twofold question of finding an average representative of a dataset of different subjects and deriving then some statistics by identifying the main modes of variation. To achieve this goal, the problem is envisioned as a joint registration/segmentation one, based on nonlinear elasticity concepts. Once a linearisation is performed in order for the displacement fields to live in a vector space, a PCA is investigated to capture the meaningful geometric variations. The computational feasibility is exemplified through several applications, demonstrating that the generated atlas encodes the fine geometrical structures and exhibits sharp edges with fewer ghosting artefacts, while the proposed approximation based PCA uncorrelates properly the more significant tendencies. This work paves the way to several applications among which: (i) motion-correction problem from a set of multiple acquisitions corrupted by motion in this joint reconstruction and registration framework or (ii) multiscale segmentation/registration problem to extract the deformation pairing the structures (i.e. viewed as global deformations) and an additional deformation reflecting the more local variability.}}

\section*{Acknowledgments}
This project is co-financed by the European Union with the European regional development fund (ERDF, HN0002137) and by the Normandie Regional Council via the M2NUM project. We also would like to acknowledge the support from the National Physics Laboratory, the Leverhulme Trust project “Breaking the non-convexity barrier”, the EPSRC grant EP/M00483X/1, the EPSRC Centre Nr. EP/N014588/1, the RISE projects ChiPS and NoMADS, the Cantab Capital Institute for the Mathematics of Information, and the Alan Turing Institute.

\bibliographystyle{siamplain}
\bibliography{references}


\section*{Supplementary material}

# Supplementary material : A Variational Model Dedicated to Joint Segmentation, Registration and Atlas Generation for Shape Analysis

Noémie Debroux, John Aston, Fabien Bonardi, Alistair Forbes, Carole Le Guyader, Marina Romanchikova and Carola Schönlieb

Department of Applied Mathematics and Theoretical Physics (DAMPT), Centre for Mathematical Sciences, University of Cambridge, Wilberforce Road, Cambridge CB3 0WA, UK
{nd448,cbs31}@cam.ac.uk
Statslab, Department of Pure Mathematics and Mathematical Statistics, Centre for Mathematical Sciences, University of Cambridge, Wilberforce Road, Cambridge CB3 0WA, UK
j.aston@statslab.cam.ac.uk
IBISC, Université d'Évry, 36, Rue du Pelvoux, CE1455 Courcouronnes, 91020 Evry Cédex, France
fabien.bonardi@univ-evry.fr
National Physical Laboratory, Hampton Road, Teddington, Middlesex, TW11 0LW UK
{alistair.forbes,marina.romanchikova}@npl.co.uk
Normandie Univ, Institut National des Sciences Appliquées de Rouen, Laboratory of Mathematics, 76000 Rouen, France
carole.le-guyader@insa-rouen.fr

## 1   Detailed proof of Theorem 1

*Proof.* The proof follows the calculus of variations arguments. We first have that $\mathcal{F}_1$ is bounded below by 0 and that $\mathcal{F}_1(\theta_R, \{\theta_{T_i}, \varphi_i\}_{i=1}^M) = \frac{1}{M}\sum_{i=1}^M \left(\|T_i - \bar{T}_i\|_{L^2(\Omega)}^2 + \|T_i - \bar{T}\|_{L^2(\Omega)}^2\right) < +\infty$ for $\theta_R = (1, 0, \cdots, 0)$, $\theta_{T_i} = (1, 0, \cdots, 0)$, $\forall i = 1, \cdots, M$, $\varphi_i = \text{Id}$, $\forall i = 1, \cdots, M$ and with $\bar{T}_i = \frac{\int_\Omega T_i \, dx}{\text{meas}(\Omega)}$ and $\bar{T} = \frac{1}{M}\sum_{i=1}^M \frac{\int_\Omega T_i \, dx}{\text{meas}(\Omega)}$. Therefore, the infimum exists and is finite.

Let $(\theta_R^k, \{\theta_{T_i}^k, \varphi_i^k\}_{i=1}^M)_k \in \mathcal{U} \times \mathcal{U}^M \times \hat{\mathcal{W}}^M$ such that $\theta_{T_i,l}^k \circ \varphi_i^k - \theta_{R,l}^k \in BV(\Omega)$ for all $l \in \{1, \cdots, N\}$, and for all $i \in \{1, \cdots, M\}$ be a minimizing sequence such that $\lim_{k \to +\infty} \mathcal{F}_1(\theta_R^k, \{\theta_{T_i}^k, \varphi_i^k\}_{i=1}^M) = \inf \mathcal{F}_1(\theta_R, \{(\theta_{T_i}, \varphi_i)\}_{i=1}^M) < +\infty$. Hence there exists $K \in \mathbb{N}$ such that $\forall k \in \mathbb{N}$, $k \geq K \Rightarrow \mathcal{F}_1(\theta_R^k, \{\theta_{T_i}^k, \varphi_i^k\}_{i=1}^M) \leq \inf \mathcal{F}_1(\theta_R, \{\theta_{T_i}, \varphi_i\}_{i=1}^M) + 1 < +\infty$. From now on, we consider $k \geq K$.

One thus has

$$\frac{1}{M}\sum_{i=1}^M \Bigg( \frac{\gamma_T}{2}\sum_{l=1}^N TV(\theta_{T_i,l}^k) + \frac{\gamma_R}{2}\sum_{l=1}^N TV(\theta_{R,l}^k) + a_1 \|\nabla \varphi_i\|_{L^4(\Omega, M_3(\mathbb{R}))}^4 - 9a_1 \text{meas}(\Omega)$$

$$+ a_2 \|\text{Cof}\nabla \varphi_i\|_{L^4(\Omega, M_3(\mathbb{R}))}^4 - 9a_2 \text{meas}(\Omega) + \frac{a_3}{2}\|\det \nabla \varphi_i\|_{L^2(\Omega)}^2 - (a_3 + a_4)\text{meas}(\Omega)$$

$$+ \mathbb{1}_{\{\|\cdot\|_{L^\infty(\Omega, M_3(\mathbb{R}))} \leq \alpha\}}(\nabla \varphi_i) + \mathbb{1}_{\{\|\cdot\|_{L^\infty(\Omega, M_3(\mathbb{R}))} \leq \beta\}}((\nabla \varphi_i)^{-1})\Bigg)$$

$$\leq \mathcal{F}_1(\theta_R^k, \{\theta_{T_i}^k, \varphi_i^k\}_{i=1}^M) \leq \inf \mathcal{F}_1(\theta_R, \{\theta_{T_i}, \varphi_i\}_{i=1}^M) + 1 < +\infty.$$

We deduce from this inequality that $\forall i = 1, \cdots, M$,

- $(\varphi_i^k)$ is uniformly bounded according to $k$ in $W^{1,4}(\Omega)$ by using the generalized Poincaré's inequality and the fact that $\varphi_i^k = \mathrm{Id}$ on $\partial\Omega$.
- $(\nabla\varphi_i^k)$ is uniformly bounded according to $k$ in $L^\infty(\Omega, M_3(\mathbb{R}))$.
- $((\nabla\varphi_i^k)^{-1})$ is uniformly bounded according to $k$ in $L^\infty(\Omega, M_3(\mathbb{R}))$.
- $(\mathrm{Cof}\nabla\varphi_i^k)$ is uniformly bounded according to $k$ in $L^4(\Omega, M_3(\mathbb{R}))$.
- $(\det\nabla\varphi_i^k)$ is uniformly bounded according to $k$ in $L^2(\Omega)$.
- $\forall l = 1, \cdots, N$, $(\theta_{T_i,l}^k)$ is uniformly bounded according to $k$ in $BV(\Omega)$ since $(TV(\theta_{T_i,l}^k))$ is uniformly bounded according to $k$ and $\|\theta_{T_i,l}^k\|_{L^1(\Omega)} \leq \mathrm{meas}(\Omega)$ for all $k \in \mathbb{N}$.
- $\forall l = 1, \cdots, N$, $(\theta_{R,l}^k)$ is uniformly bounded according to $k$ in $BV(\Omega)$ since $(TV(\theta_{R,l}^k))$ is uniformly bounded according to $k$ and $\|\theta_{R,l}^k\|_{L^1(\Omega)} \leq \mathrm{meas}(\Omega)$ for all $k \in \mathbb{N}$.

Since we have assumed enough regularity on the boundary of $\Omega$, we can apply [7, Theorem 136] and extract subsequences still denoted $(\theta_{T_i,l}^k)$ and $(\theta_{R,l}^k)$ such that

$$\theta_{T_i,l}^k \underset{k\to+\infty}{\longrightarrow} \bar{\theta}_{T_i,l} \in BV(\Omega) \text{ in } L^1(\Omega),$$
$$\theta_{R,l}^k \underset{k\to+\infty}{\longrightarrow} \bar{\theta}_{R,l} \in BV(\Omega) \text{ in } L^1(\Omega).$$

As the convergence is strong in $L^1(\Omega)$, then we can extract subsequences of $(\theta_{T_i,l}^k)_k$ and $(\theta_{R,l}^k)_k$ still denoted $(\theta_{T_i,l}^k)_k$ and $(\theta_{R,l}^k)_k$ such that $\theta_{T_i,l}^k \underset{k\to+\infty}{\longrightarrow} \bar{\theta}_{T_i,l}$ almost everywhere in $\Omega$ and $\theta_{R,l}^k \underset{k\to+\infty}{\longrightarrow} \bar{\theta}_{R,l}$ almost everywhere in $\Omega$ and thus $\bar{\theta}_{T_i,l} \in BV(\Omega, \{0,1\})$ and $\bar{\theta}_{R,l} \in BV(\Omega, \{0,1\})$. Besides, $\sum_{l=1}^N \theta_{T_i,l}^k(x) = 1 \underset{k\to+\infty}{\longrightarrow} \sum_{l=1}^N \bar{\theta}_{T_i,l}(x) = 1$ almost everywhere in $\Omega$ and $\sum_{l=1}^N \theta_{R,l}^k(x) = 1 \underset{k\to+\infty}{\longrightarrow} \sum_{l=1}^N \bar{\theta}_{R,l}(x) = 1$ so that $\bar{\theta}_{T_i} \in \mathcal{U}$ and $\bar{\theta}_R \in \mathcal{U}$.

There also exist subsequences still denoted $(\varphi_i^k)_k$ such that

$$\varphi_i^k \underset{k\to+\infty}{\rightharpoonup} \bar{\varphi}_i \text{ in } W^{1,4}(\Omega, \mathbb{R}^3).$$

By continuity of the trace operator, we deduce that $\bar{\varphi}_i \in \mathrm{Id} + W_0^{1,4}(\Omega, \mathbb{R}^3)$. We can also extract subsequences still denoted $(\mathrm{Cof}\nabla\varphi_i^k)_k$ and $(\det\nabla\varphi_i^k)_k$ such that

$$\mathrm{Cof}\nabla\varphi_i^k \underset{k\to+\infty}{\rightharpoonup} \bar{X}_i \text{ in } L^4(\Omega, M_3(\mathbb{R})),$$
$$\det\nabla\varphi_i^k \underset{k\to+\infty}{\rightharpoonup} \bar{\delta}_i \text{ in } L^2(\Omega),$$

and by [4, Theorem VI.3.3], we deduce that $\bar{X}_i = \mathrm{Cof}\nabla\bar{\varphi}_i$ and $\bar{\delta}_i = \det\nabla\bar{\varphi}_i$.
There exist subsequences still denoted $(\nabla\varphi_i^k)$ and $((\nabla\varphi_i^k)^{-1})$ such that

$$(\nabla\varphi_i^k)^{-1} \underset{k\to+\infty}{\overset{*}{\rightharpoonup}} u_i \text{ in } L^\infty(\Omega, M_3(\mathbb{R})),$$
$$\nabla\varphi_i^k \underset{k\to+\infty}{\overset{*}{\rightharpoonup}} \nabla\bar{\varphi}_i \text{ in } L^\infty(\Omega, M_3(\mathbb{R})),$$

by uniqueness of the weak limit in $L^4(\Omega, M_3(\mathbb{R}))$ and the continuous embedding of $L^\infty(\Omega, M_3(\mathbb{R}))$ into $L^4(\Omega, M_3(\mathbb{R}))$, and thus $\bar{\varphi}_i \in \mathrm{Id} + W_0^{1,\infty}(\Omega, \mathbb{R}^3)$.

Let us now prove that $u_i = (\nabla \bar{\varphi}_i)^{-1}$ for each $i$. For all $p \in L^1(\Omega, M_3(\mathbb{R}))$, we have that

$$\begin{aligned}
\int_\Omega ((\nabla \varphi_i^k)^{-1} \nabla \bar{\varphi}_i - \mathrm{I}_3) : p \, dx &= \int_\Omega ((\nabla \varphi_i^k)^{-1} \nabla \bar{\varphi}_i - u_i \nabla \bar{\varphi}_i + u_i \nabla \bar{\varphi}_i - u_i \nabla \varphi_i^k \\
&\quad + u_i \nabla \varphi_i^k - (\nabla \varphi_i^k)^{-1} \nabla \varphi_i^k) : p \, dx, \\
&= \int_\Omega (((\nabla \varphi_i^k)^{-1} - u_i) \nabla \bar{\varphi}_i) : p \, dx + \int_\Omega (u_i(\nabla \bar{\varphi}_i - \nabla \varphi_i^k)) : p \, dx \\
&\quad + \int_\Omega ((u_i - (\nabla \varphi_i^k)^{-1}) \nabla \varphi_i^k) : p \, dx, \\
&= \int_\Omega ((\nabla \varphi_i^k)^{-1} - u_i) : p(\nabla \bar{\varphi}_i)^T \, dx + \int_\Omega (\nabla \bar{\varphi}_i - \nabla \varphi_i^k) : u_i^T p \, dx \\
&\quad + \int_\Omega (u_i - (\nabla \varphi_i^k)^{-1}) : p(\nabla \varphi_i^k)^T \, dx.
\end{aligned}$$

Since $\nabla \bar{\varphi}_i \in L^\infty(\Omega, M_3(\mathbb{R}))$, $u_i \in L^\infty(\Omega, M_3(\mathbb{R}))$ and for all $k \in \mathbb{N}$, $\|\nabla \varphi_i^k\|_{L^\infty(\Omega, M_3(\mathbb{R}))} \leq \alpha$, then we have that $p(\nabla \bar{\varphi}_i)^T \in L^1(\Omega, M_3(\mathbb{R}))$ so $\int_\Omega ((\nabla \varphi_i^k)^{-1} - u_i) : p(\nabla \bar{\varphi}_i)^T \, dx \xrightarrow[k \to +\infty]{} 0$, $u_i^T p \in L^1(\Omega, M_3(\mathbb{R}))$ so $\int_\Omega (\nabla \bar{\varphi}_i - \nabla \varphi_i^k) : u_i^T p \, dx \xrightarrow[k \to +\infty]{} 0$, and $|\int_\Omega (u_i - (\nabla \varphi_i^k)^{-1}) : p(\nabla \varphi_i^k)^T \, dx| \leq \alpha |\int_\Omega (u_i - (\nabla \varphi_i^k)^{-1}) : p' \, dx| \xrightarrow[k \to +\infty]{} 0$ with $p' \in L^1(\Omega, M_3(\mathbb{R}))$. Therefore, we get $(\nabla \varphi_i^k)^{-1} \nabla \bar{\varphi}_i \xrightharpoonup[k \to +\infty]{*} \mathrm{I}_3$ in $L^\infty(\Omega, M_3(\mathbb{R}))$. But we know that $(\nabla \varphi_i^k)^{-1} \xrightharpoonup[k \to +\infty]{*} u_i$ in $L^\infty(\Omega, M_3(\mathbb{R}))$, so that $(\nabla \varphi_i^k)^{-1} \nabla \bar{\varphi}_i \xrightharpoonup[k \to +\infty]{*} u_i \nabla \bar{\varphi}_i = \mathrm{I}_3$ in $L^\infty(\Omega, M_3(\mathbb{R}))$ and $u_i = (\nabla \bar{\varphi}_i)^{-1}$.

Since $\theta_{T_i,l}^k \xrightarrow[k \to +\infty]{} \bar{\theta}_{T_i,l}$ almost everywhere in $\Omega$, then $T_i \theta_{T_i,l}^k \xrightarrow[k \to +\infty]{} T_i \bar{\theta}_{T_i,l}$ almost everywhere in $\Omega$. Furthermore, $|\theta_{T_i,l}^k| \leq 1 \in L^1(\Omega)$ almost everywhere in $\Omega$ as $\Omega$ is bounded, and $|T_i \theta_{T_i,l}^k| \leq \|T_i\|_{L^\infty(\Omega)} \in L^1(\Omega)$ almost everywhere in $\Omega$. We can thus apply the dominated convergence theorem and get

$$\begin{aligned}
\int_\Omega \theta_{T_i,l}^k \, dx &\xrightarrow[k \to +\infty]{} \int_\Omega \bar{\theta}_{T_i,l} \, dx, \\
\int_\Omega \theta_{T_i,l}^k T_i \, dx &\xrightarrow[k \to +\infty]{} \int_\Omega \bar{\theta}_{T_i,l} T_i \, dx, \\
c_{T_i,l}^k &\xrightarrow[k \to +\infty]{} \bar{c}_{T_i,l}.
\end{aligned}$$

As $T_i \in W^{1,\infty}(\Omega)$, we denote $\kappa_i$ the Lipschitz constant related to $T_i$. Then the chain rule applies and we have that $T_i \circ \varphi_i^k \in W^{1,\infty}(\Omega)$, $T_i \circ \bar{\varphi}_i \in W^{1,\infty}(\Omega)$ and $\|T_i \circ \varphi_i^k - T_i \circ \bar{\varphi}_i\|_{\mathcal{C}^0(\Omega)} \leq \kappa_i \|\varphi_i^k - \bar{\varphi}_i\|_{\mathcal{C}^0(\Omega)} \xrightarrow[k \to +\infty]{} 0$ using the Sobolev compact embedding $W^{1,\infty}(\Omega) \hookrightarrow_c \mathcal{C}^0(\Omega)$. Therefore, up to a subsequence, $\theta_{R,l}^k T_i \circ \varphi_i^k \xrightarrow[k \to +\infty]{} \bar{\theta}_{R,l} T_i \circ \bar{\varphi}_i$ almost everywhere in $\Omega$, $\theta_{R,l}^k \xrightarrow[k \to +\infty]{} \bar{\theta}_{R,l}$ almost everywhere in $\Omega$ with

$|\theta_{R,l}^k T_i \circ \varphi_i^k| \leq \|T_i\|_{L^\infty(\Omega)} \in L^1(\Omega)$ and $|\theta_{R,l}^k| \leq 1 \in L^1(\Omega)$ almost everywhere in $\Omega$. Thus the dominated convergence theorem applies and we get

$$\frac{1}{M}\sum_{i=1}^M \int_\Omega \theta_{R,l}^k T_i \circ \varphi_i^k \, dx \xrightarrow[k\to+\infty]{} \frac{1}{M}\sum_{i=1}^M \int_\Omega \bar\theta_{R,l} T_i \circ \bar\varphi_i \, dx,$$

$$\int_\Omega \theta_{R,l}^k \, dx \xrightarrow[k\to+\infty]{} \int_\Omega \bar\theta_{R,l} \, dx,$$

$$c_{R,l}^k \xrightarrow[k\to+\infty]{} \bar c_{R,l}.$$

We also know that $W_{Op}$ is continuous and convex. If $\psi_n \xrightarrow[n\to+\infty]{} \psi$ in $W^{1,4}(\Omega, \mathbb{R}^3)$, then $\nabla\psi_n \xrightarrow[n\to+\infty]{} \nabla\psi$ in $L^4(\Omega, M_3(\mathbb{R}))$ and one can extract a subsequence still denoted $(\nabla\psi_n)$ such that $\nabla\psi_n \xrightarrow[n\to+\infty]{} \nabla\psi$ almost everywhere in $\Omega$. If $\alpha_n \xrightarrow[n\to+\infty]{} \alpha$ in $L^4(\Omega, M_3(\mathbb{R}))$, then one can extract a subsequence still denoted $(\alpha_n)$ such that $\alpha_n \xrightarrow[n\to+\infty]{} \alpha$ almost everywhere in $\Omega$. If $\delta_n \xrightarrow[n\to+\infty]{} \delta$ in $L^2(\Omega)$ then one can extract a subsequence still denoted $(\delta_n)$ such that $\delta_n \xrightarrow[n\to+\infty]{} \delta$ almost everywhere in $\Omega$. Then by continuity of $W_{Op}$, we get that $W_{Op}(\nabla\psi_n(x), \alpha_n(x), \delta_n(x)) \xrightarrow[n\to+\infty]{} W_{Op}(\nabla\psi(x), \alpha(x), \delta(x))$ almost everywhere in $\Omega$. We then apply Fatou's lemma and get $\liminf_{n\to+\infty} \int_\Omega W_{Op}(\nabla\psi_n(x), \alpha_n(x), \delta_n(x)) \, dx \geq \int_\Omega W_{Op}(\nabla\psi(x), \alpha(x), \delta(x)) \, dx$. As $W_{Op}$ is convex, so is $\int_\Omega W_{Op}(\xi(x), \alpha(x), \delta(x)) \, dx$ and we can apply [2, Corollaire III.8]. Therefore $\int_\Omega W_{Op}(\xi(x), \alpha(x), \delta(x)) \, dx$ is also weakly lower semicontinuous in $L^4(\Omega, M_3(\mathbb{R})) \times L^4(\Omega, M_3(\mathbb{R})) \times L^2(\Omega)$. We then deduce that

$$+\infty > \liminf_{k\to+\infty} \int_\Omega W_{Op}(\nabla\varphi_i^k, \mathrm{Cof}\nabla\varphi_i^k, \det\nabla\varphi_i^k) \, dx \geq \int_\Omega W_{Op}(\nabla\bar\varphi_i, \mathrm{Cof}\nabla\bar\varphi_i, \det\nabla\bar\varphi_i) \, dx.$$

Let $\bar q = q\frac{1+s}{q+s} = 4\frac{1+10}{4+10} = \frac{44}{14} > 3$. We then compute

$$\int_\Omega \|(\nabla\bar\varphi_i)^{-1}\|_F^{\bar q} \det\nabla\bar\varphi_i \, dx = \int_\Omega \left\|\frac{\mathrm{Cof}\nabla\bar\varphi_i^T}{\det\nabla\bar\varphi_i}\right\|_F^{\bar q} \det\nabla\bar\varphi_i \, dx,$$

$$= \int_\Omega \|\mathrm{Cof}\nabla\bar\varphi_i\|_F^{\bar q} (\det\nabla\bar\varphi_i)^{1-\bar q} \, dx,$$

$$\leq \|\mathrm{Cof}\nabla\bar\varphi_i\|_{L^4(\Omega)}^{\bar q} \|(\det\nabla\bar\varphi_i)^{-1}\|_{L^{10}(\Omega)}^{1-\bar q} < +\infty,$$

from Hölder's inequality and the finiteness of $\int_\Omega W_{Op}(\nabla\bar\varphi_i, \mathrm{Cof}\nabla\bar\varphi_i, \det\nabla\bar\varphi_i) \, dx$. Thus Ball's conditions are satisfied [1, Theorems 1 and 2] and $\bar\varphi_i$ are bi-Hölder homeomorphisms.

By the weak-$*$ lower semicontinuity of $\|.\|_{L^\infty(\Omega, M_3(\mathbb{R}))}$, we have that $\|\nabla\bar\varphi_i\|_{L^\infty(\Omega, M_3(\mathbb{R}))} \leq \liminf_{k\to+\infty} \|\nabla\varphi_i^k\|_{L^\infty(\Omega, M_3(\mathbb{R}))} \leq \alpha$ and $\|(\nabla\bar\varphi_i)^{-1}\|_{L^\infty(\Omega, M_3(\mathbb{R}))} \leq \liminf_{k\to+\infty} \|(\nabla\varphi_i^k)^{-1}\|_{L^\infty(\Omega, M_3(\mathbb{R}))} \leq \beta$ so that $\mathbb{1}_{\{\|.\|_{L^\infty(\Omega, M_3(\mathbb{R}))}\leq\alpha\}}(\nabla\bar\varphi_i) = 0 \leq \liminf_{k\to+\infty} \mathbb{1}_{\{\|.\|_{L^\infty(\Omega, M_3(\mathbb{R}))}\leq\alpha\}}(\nabla\varphi_i^k)$ and $\mathbb{1}_{\{\|.\|_{L^\infty(\Omega, M_3(\mathbb{R}))}\leq\beta\}}((\nabla\bar\varphi_i)^{-1}) = 0 \leq \liminf_{k\to+\infty} \mathbb{1}_{\{\|.\|_{L^\infty(\Omega, M_3(\mathbb{R}))}\leq\beta\}}((\nabla\varphi_i^k)^{-1})$. The

same reasoning applies for each $i = 1, \cdots, M$.
We also have

$$\|\nabla(\bar{\varphi}_i^{-1})\|_{L^\infty(\Omega, M_3(\mathbb{R}))} = \|(\nabla\bar{\varphi}_i)^{-1}(\bar{\varphi}_i^{-1})\|_{L^\infty(\Omega, M_3(\mathbb{R}))} = \|(\nabla\bar{\varphi}_i)^{-1}\|_{L^\infty(\Omega, M_3(\mathbb{R}))} \leq \beta < +\infty,$$

since $\bar{\varphi}_i$ is an homeomorphism from $\Omega$ to $\mathrm{Id}(\Omega) = \Omega$, according to Ball's results, meaning that $\bar{\varphi}_i^{-1}$ is a one to one mapping from $\Omega$ to $\Omega$. So, using Poincaré-Wirtinger inequality, we deduce that $\bar{\varphi}_i$ is a bi-Lipschitz homeomorphism. The same reasoning applies for the sequence $(\varphi_i^k)_k$.

We now have that $\forall i \in \{1, \cdots, M\}, \forall l \in \{1, \cdots, N\}$,

- $\theta_{T_i,l}^k \underset{k \to +\infty}{\longrightarrow} \bar{\theta}_{T_i,l}$ in $L^1(\Omega)$ with $\bar{\theta}_{T,l} \in BV(\Omega, \{0, 1\})$.
- $\varphi_i^k \underset{k \to +\infty}{\rightharpoonup} \bar{\varphi}_i$ in $W^{1,4}(\Omega, \mathbb{R}^3)$ with $(\varphi_i^k)$ bi-Lipschitz homeomorphisms $(\varphi_i^k \in \mathcal{C}^{0,\alpha}(\Omega, \mathbb{R}^3), (\varphi_i^k)^{-1} \in \mathcal{C}^{0,\alpha'}(\Omega, \mathbb{R}^3))$, $\bar{\varphi}_i$ bi-Lipschitz homeomorphisms $(\bar{\varphi}_i \in \mathcal{C}^{0,\alpha}(\Omega, \mathbb{R}^3), (\bar{\varphi}_i)^{-1} \in \mathcal{C}^{0,\alpha'}(\Omega, \mathbb{R}^3))$ and $\varphi_i^k \underset{k \to +\infty}{\longrightarrow} \bar{\varphi}_i$ in $\mathcal{C}^{0,\alpha}(\Omega, \mathbb{R}^3)$, $\alpha < \frac{1}{4}$ and $\alpha' < 1 - \frac{3}{q}$ using the Sobolev embeddings properties.

We then can prove that $\varphi_i^k \circ \bar{\varphi}_i^{-1} \underset{k \to +\infty}{\longrightarrow} \mathrm{Id}$ in $\mathcal{C}^{0,\alpha''}(\Omega, \mathbb{R}^3)$ with $\alpha'' = \alpha\alpha'$ for all $i \in \{1, \cdots, M\}$ and all $l \in \{1, \cdots, N\}$. Indeed one has

$$\sup_{x \in \Omega} |\varphi_i^k \circ \bar{\varphi}_i^{-1}(x) - \bar{\varphi}_i \circ \bar{\varphi}_i^{-1}(x)| + \sup_{x \neq y} \frac{|\varphi_i^k \circ \bar{\varphi}_i^{-1}(x) - \varphi_i^k \circ \bar{\varphi}_i^{-1}(y) - \bar{\varphi}_i \circ \bar{\varphi}_i^{-1}(x) + \bar{\varphi}_i \circ \bar{\varphi}_i^{-1}(y)|}{|x-y|^{\alpha''}}$$

$$\leq \sup_{x \in \Omega} |\varphi_i^k(x) - \bar{\varphi}_i(x)| + \sup_{x \neq y} \frac{|\varphi_i^k \circ \bar{\varphi}_i^{-1}(x) - \varphi_i^k \circ \bar{\varphi}_i^{-1}(y) - \bar{\varphi}_i \circ \bar{\varphi}_i^{-1}(x) + \bar{\varphi}_i \circ \bar{\varphi}_i^{-1}(y)|}{|\bar{\varphi}_i^{-1}(x) - \bar{\varphi}_i^{-1}(y)|^\alpha}$$

$$\cdot \frac{|\bar{\varphi}_i^{-1}(y) - \bar{\varphi}_i^{-1}(x)|^\alpha}{|x-y|^{\alpha''}} \text{ as } \bar{\varphi}_i^{-1} \text{ is a one to one mapping from } \Omega \text{ to } \Omega,$$

$$\leq \sup_{x \in \Omega} |\varphi_i^k(x) - \bar{\varphi}_i(x)| + \sup_{x \neq y} \frac{|\varphi_i^k(x) - \varphi_i^k(y) - \bar{\varphi}_i(x) + \bar{\varphi}_i(y)|}{|x-y|^\alpha} \cdot \sup_{x \neq y} \frac{|\bar{\varphi}_i^{-1}(y) - \bar{\varphi}_i^{-1}(x)|^\alpha}{|x-y|^{\alpha''}},$$

$$\leq (1 + \|\bar{\varphi}_i^{-1}\|_{\mathcal{C}^{0,\alpha'}(\Omega, \mathbb{R}^3)}^\alpha)\|\varphi_i^k - \bar{\varphi}_i\|_{\mathcal{C}^{0,\alpha}(\Omega, \mathbb{R}^3)} \underset{k \to +\infty}{\longrightarrow} 0.$$

Let us now prove that $\theta_{T_i,l}^k \circ \varphi_i^k \underset{k \to +\infty}{\longrightarrow} \bar{\theta}_{T_i,l} \circ \bar{\varphi}_i$ in $L^1(\Omega)$ based on the proof of [8, Lemma 5.1].

Since $\varphi_i^k \circ \bar{\varphi}_i^{-1} \underset{k \to +\infty}{\longrightarrow} \mathrm{Id}$ in $\mathcal{C}^{0,\alpha''}(\Omega)$ and so uniformly, we can write

$$\forall \varepsilon > 0, \exists K \in \mathbb{N}, \forall k \in \mathbb{N}, k \geq K \Rightarrow \sup_{x \in \Omega}|\varphi_i^k \circ \bar{\varphi}_i^{-1}(x) - x| \leq \varepsilon.$$

Also, given $\eta > 0$, by Egorov's theorem, there exists a set $S_\eta$ with Lebesgue measure $\mathrm{meas}(S_\eta) \leq \eta$ such that $\theta_{T_i,l}^k \underset{k \to +\infty}{\longrightarrow} \bar{\theta}_{T_i,l}$ uniformly in $\Omega \setminus S_\eta$.

Now, we choose $K$ large enough such that $\forall k \in \mathbb{N}, k \geq K$ then

$$|\varphi_i^k \circ \bar{\varphi}_i(x) - x| \leq \eta \text{ in } \Omega,$$
$$|\theta_{T_i,l}^k(x) - \bar{\theta}_{T_i,l}(x)| \leq \eta \text{ in } \Omega \setminus S_\eta.$$

We also denote $\mathcal{O} = \{x \in \Omega \,|\, \bar{\theta}_{T_i,l}(x) = 1\}$, $(\mathcal{O} \setminus S_\eta)_\eta = \{x \in \mathcal{O} \setminus S_\eta \,|\, \text{dist}(x, \partial(\mathcal{O} \setminus S_\eta)) > \eta\}$, $\mathcal{O}^\eta = \{x \in \Omega \,|\, \text{dist}(x, \partial \mathcal{O}) \leq \eta\}$. Then $\forall x \in (\mathcal{O} \setminus S_\eta)_\eta$, $\varphi_i^k \circ \bar{\varphi}_i^{-1}(x) \in \mathcal{O} \setminus S_\eta$ and $\forall x \in (\Omega \setminus (S_\eta \cup \mathcal{O}))_\eta$, $\varphi_i^k \circ \bar{\varphi}_i^{-1}(x) \in \Omega \setminus (S_\eta \cup \mathcal{O})$. We then have

$$\int_\Omega |\theta_{T_i,l}^k \circ \varphi_i^k - \bar{\theta}_{T_i,l} \circ \bar{\varphi}_i| \, dx = \int_\Omega \frac{|\theta_{T_i,l}^k \circ \varphi_i^k \circ \bar{\varphi}_i^{-1}(x) - \bar{\theta}_{T_i,l}(x)|}{\det \nabla \bar{\varphi}_i(\bar{\varphi}_i^{-1}(x))} \, dx,$$

$$\leq \int_{(\mathcal{O} \setminus S_\eta)_\eta} \frac{|\theta_{T_i,l}^k \circ \varphi_i^k \circ \bar{\varphi}_i^{-1}(x) - 1|}{\det \nabla \bar{\varphi}_i(\bar{\varphi}_i^{-1}(x))} \, dx + \int_{(\Omega \setminus (\mathcal{O} \cup S_\eta))_\eta} \frac{|\theta_{T_i,l}^k \circ \varphi_i^k \circ \bar{\varphi}_i^{-1}(x)|}{\det \nabla \bar{\varphi}_i(\bar{\varphi}_i^{-1}(x))} \, dx$$
$$+ \int_{S_\eta \cup S_\eta^\eta \cup \mathcal{O}^\eta \cup \Omega^\eta} \frac{|\theta_{T_i,l}^k \circ \varphi_i^k \circ \bar{\varphi}_i^{-1}(x) - \bar{\theta}_{T_i,l}(x)|}{\det \nabla \bar{\varphi}_i(\bar{\varphi}_i^{-1}(x))} \, dx,$$

$$\leq \eta \int_{(\mathcal{O} \setminus S_\eta)_\eta} \frac{1}{\det \nabla \bar{\varphi}_i(\bar{\varphi}_i^{-1}(x))} \, dx + \eta \int_{(\Omega \setminus (\mathcal{O} \cup S_\eta))_\eta} \frac{1}{\det \nabla \bar{\varphi}_i(\bar{\varphi}_i^{-1}(x))} \, dx$$
$$+ 2 \int_{S_\eta \cup S_\eta^\eta \cup \mathcal{O}^\eta \cup \Omega^\eta} \frac{1}{\det \nabla \bar{\varphi}_i(\bar{\varphi}_i^{-1}(x))} \, dx,$$

$$\leq (\eta \, \text{meas}(\Omega)^{\frac{10}{11}} + \eta \, \text{meas}(\Omega)^{\frac{10}{11}} + 2(\text{meas}(S_\eta) + \text{meas}(S_\eta^\eta) + \text{meas}(\mathcal{O}^\eta) + \text{meas}(\Omega^\eta))^{\frac{10}{11}})$$
$$(\int_\Omega (\det \nabla \bar{\varphi}_i(x))^{-10} \, dx)^{\frac{1}{11}},$$

$$\leq (2\eta \, \text{meas}(\Omega)^{\frac{10}{11}} + 2(\frac{4\pi}{3}\eta^3(\text{meas}(\partial S_\eta) + \text{meas}(\partial \mathcal{O}) + \text{meas}(\partial \Omega)) + \eta)^{\frac{10}{11}}) \|\det \nabla \bar{\varphi}_i(x)^{-1}\|_{L^{10}(\Omega)}^{\frac{10}{11}},$$

with $\|\det \nabla \bar{\varphi}_i(x)^{-1}\|_{L^{10}(\Omega)}^{\frac{10}{11}} < +\infty$. By letting $\eta$ tend to 0, we obtain that $\theta_{T_i,l}^k \circ \varphi_i^k \xrightarrow[k \to +\infty]{} \bar{\theta}_{T_i,l} \circ \bar{\varphi}_i$ in $L^1(\Omega)$, for all $i \in \{1, \cdots, M\}$ and for all $l \in \{1, \cdots, N\}$. Since $\theta_{T_i,l}^k \xrightarrow[k \to +\infty]{} \bar{\theta}_{T_i,l}$ in $L^1(\Omega)$ and $\theta_{R,l}^k \xrightarrow[k \to +\infty]{} \bar{\theta}_{R,l}$ in $L^1(\Omega)$ and $\theta_{T_i,l}^k \circ \varphi_i^k - \theta_{R,l}^k \xrightarrow[k \to +\infty]{} \bar{\theta}_{T_i,l} \circ \bar{\varphi}_i - \bar{\theta}_{R,l}$ for all $l \in \{1, \cdots, N\}$ and for all $i \in \{1, \cdots, M\}$, we deduce by the lower semicontinuity of the total variation that $\frac{1}{M} \sum_{i=1}^M \sum_{l=1}^N \frac{\gamma_T}{2} TV(\bar{\theta}_{T_i,l}) + \frac{\gamma_R}{2} TV(\bar{\theta}_{R,l}) + \frac{\lambda}{2} TV(\bar{\theta}_{T_i,l} \circ \bar{\varphi}_i - \bar{\theta}_{R,l}) \leq \liminf_{k \to +\infty} \frac{1}{M} \sum_{i=1}^M \sum_{l=1}^N \frac{\gamma_T}{2} TV(\theta_{T_i,l}^k) + \frac{\gamma_R}{2} TV(\theta_{R,l}^k) + \frac{\lambda}{2} TV(\theta_{T_i,l}^k \circ \varphi_i^k - \theta_{R,l}^k)$.

Since from what precedes $\theta_{T_i,l}^k (c_{T_i,l}^k - T_i)^2 \xrightarrow[k \to +\infty]{} \bar{\theta}_{T_i,l}(\bar{c}_{T_i,l} - T_i)^2$ almost everywhere in $\Omega$ and $\theta_{R,l}^k (c_{R,l}^k - T_i \circ \varphi_i^k)^2 \xrightarrow[k \to +\infty]{} \bar{\theta}_{R,l}(\bar{c}_{R,l} - T_i \circ \bar{\varphi}_i)^2$ almost everywhere in $\Omega$ with $|\theta_{T_i,l}^k (c_{T_i,l}^k - T_i)^2| \leq 4 \|T_i\|_{L^\infty(\Omega)}^2$ and $|\theta_{R,l}^k (c_{R,l}^k - T_i \circ \varphi_i^k)^2| \leq 4 \sum_{i=1}^M \|T_i\|_{L^\infty(\Omega)}^2$, we can apply the dominated convergence theorem and get $\int_\Omega \theta_{T_i,l}^k (c_{T_i,l}^k - T_i)^2 \, dx \xrightarrow[k \to +\infty]{} \int_\Omega \bar{\theta}_{T_i,l}(\bar{c}_{T_i,l} - T_i)^2 \, dx$ and $\int_\Omega \theta_{R,l}^k (c_{R,l}^k - T_i \circ \varphi_i^k) \, dx \xrightarrow[k \to +\infty]{} \int_\Omega \bar{\theta}_{R,l}(\bar{c}_{R,l} - T_i \circ \bar{\varphi}_i)^2 \, dx$.

By combining all the results, we get that

$$\mathcal{F}_1(\bar{\theta}_R, \{\bar{\theta}_{T_i}, \bar{\varphi}_i\}_{i=1}^M) \leq \liminf_{k \to +\infty} \mathcal{F}_1(\theta_R^k, \{\theta_{T,i}^k, \varphi_i^k\}_{i=1}^M) < +\infty.$$

By finiteness of the functional, we deduce that $\bar{\theta}_R \in \mathcal{U}$, $\bar{\theta}_{T_i} \in \mathcal{U}$ for all $i \in \{1, \cdots, M\}$ and $\bar{\varphi}_i \in \hat{\mathcal{W}}$ for all $i \in \{1, \cdots, M\}$ with $\bar{\theta}_{T_i,l} \circ \bar{\varphi}_i - \bar{\theta}_{R,l} \in BV(\Omega)$ for all $l \in \{1, \cdots, N\}$ and for all $i \in \{1, \cdots, M\}$ and there exists a minimizer to our problem.

## 2 Detailed proof of Theorem 2

*Proof.* Let $(\gamma_j)_{j \geq 0}$ be an increasing sequence of positive real numbers such that $\lim_{j \to +\infty} \gamma_j = +\infty$. Let $(\{\varphi_{i,k_j}, \theta_{T_i,k_j}, V_{i,k_j}, W_{i,k_j}\}_{i=1}^M, (\theta_{\tilde{T}_i,l,k_j})_{\substack{i=1,\cdots,M \\ l=1,\cdots,N}}, \theta_{R,k_j})$ be a minimizing sequence of this problem $\mathcal{F}_{1,\gamma}$ for $\gamma = \gamma_j$. By definition, there exists $((\theta_{R,n}, \{\theta_{T_i,n}, \varphi_{i,n}\}_{i=1}^M) \in \mathcal{U} \times \mathcal{U}^M \times \hat{\mathcal{W}}^M$ such that $\theta_{T_i,l,n} \circ \varphi_{i,n} - \theta_{R,l,n} \in BV(\Omega, \{-1, 0, 1\})$ for any $l \in \{1, \cdots, N\}$ and all $i = 1, \cdots, M$, and

$$\mathcal{F}_1(\theta_{R,n}, \{\varphi_{i,n}, \theta_{T_i,n}\}_{i=1}^M)$$
$$= \mathcal{F}_{1,\gamma_j}(\{\varphi_{i,n}, \theta_{T_i,n}, \nabla \varphi_{i,n}, (\nabla \varphi_{i,n})^{-1}\}_{i=1}^M, (\theta_{T_i,l,n} \circ \varphi_{i,n} - \theta_{R,l,n})_{\substack{i=1,\cdots,M \\ l=1,\cdots,N}}, \theta_{R,n})$$
$$\leq \inf \mathcal{F}_1(\theta_R, \{\theta_{T_i}, \varphi_i\}_{i=1}^M) + \frac{1}{n} < +\infty.$$

Then by definition of a minimizing sequence we get that

$$\forall \varepsilon > 0, \exists N(\varepsilon, \gamma_j) \in \mathbb{N}, \forall k \in \mathbb{N}, k \geq N(\varepsilon, \gamma_j) \Rightarrow \mathcal{F}_{1,\gamma_j}(\{\varphi_{i,k_j}, \theta_{T_i,k_j}, V_{i,k_j}, W_{i,k_j}\}_{i=1}^M,$$
$$(\theta_{\tilde{T}_i,l,k_j})_{\substack{i=1,\cdots,M \\ l=1,\cdots,N}}, \theta_{R,k_j}) \leq \inf \mathcal{F}_{1,\gamma_j}(\{\varphi_i, \theta_{T_i}, V_i, W_i\}_{i=1}^M, (\theta_{\tilde{T}_i,l})_{\substack{i=1,\cdots,M \\ l=1,\cdots,N}}, \theta_R) + \varepsilon,$$
$$\leq \mathcal{F}_1(\theta_{R,n}, \{\theta_{T_i,n}, \varphi_{i,n}\}_{i=1}^M) + \varepsilon,$$
$$= \mathcal{F}_{1,\gamma_j}(\{\varphi_{i,n}, \theta_{T_i,n}, \nabla \varphi_{i,n}, (\nabla \varphi_{i,n})^{-1}\}_{i=1}^M, (\theta_{T_i,l,n} \circ \varphi_{i,n} - \theta_{R,l,n})_{\substack{i=1,\cdots,M \\ l=1,\cdots,N}}, \theta_{R,n}) + \varepsilon,$$
$$\leq \inf \mathcal{F}_1(\theta_R, \{\theta_{T_i}, \varphi_i\}_{i=1}^M) + \frac{1}{n} + \varepsilon < +\infty.$$

Let us set in particular $\varepsilon = \frac{1}{\gamma_j}$, then there exists $N_j \in \mathbb{N}$ such that for all $k \in \mathbb{N}$, $k \geq N_j \Rightarrow$

$$\mathcal{F}_{1,\gamma_j}(\{\varphi_{i,k_j}, \theta_{T_i,k_j}, V_{i,k_j}, W_{i,k_j}\}_{i=1}^M, (\theta_{\tilde{T}_i,l,k_j})_{\substack{j=1,\cdots,M \\ j=1,\cdots,N}}, \theta_{R,k_j}) \leq \inf \mathcal{F}_1(\theta_R, \{\theta_{T_i}, \varphi_i\}_{i=1}^M) + \frac{1}{n} + \frac{1}{\gamma_j}$$
$$\leq \inf \mathcal{F}_1(\theta_R, \{\theta_{T_i}, \varphi_i\}_{i=1}^M) + \frac{1}{n} + \frac{1}{\gamma_0}$$
$$< +\infty.$$

So according to the previous inequality and by setting, for the sake of simplicity, for all $i = 1, \cdots, M$,

$$\begin{aligned}
\varphi_{i,j} &= \varphi_{i,N_j,j}, \\
V_{i,j} &= V_{i,N_j,j}, \\
\theta_{T_i,j} &= \theta_{T_i,N_j,j}, \\
\theta_{R,j} &= \theta_{R,N_j,j}, \\
\theta_{\tilde{T}_i,l,j} &= \theta_{\tilde{T}_i,l,N_j,j}, \\
W_{i,j} &= W_{i,N_j,j},
\end{aligned}$$

we get that

- $(\varphi_{i,j})$ is uniformly bounded according to $j$ in $W^{1,4}(\Omega)$ by using the generalized Poincaré's inequality since $(V_{i,j})$ is uniformly bounded according to $j$ in $L^4(\Omega, M_3(\mathbb{R}))$, for all $i = 1, \cdots, M$.

- $(V_{i,j})$ is uniformly bounded according to $j$ in $L^4(\Omega, M_3(\mathbb{R}))$ and in $L^\infty(\Omega, M_3(\mathbb{R}))$, for all $i = 1, \cdots, M$.

- $(\theta_{T_i,l,j})$ is uniformly bounded according to $j$ in $BV(\Omega)$ since $\|\theta_{T_i,l,j}\|_{L^1(\Omega)} \leq \text{meas}(\Omega) < +\infty$ and $(TV(\theta_{T_i,l,j}))$ is uniformly bounded according to $j$ for all $l \in \{1, \cdots, N\}$, and for all $i = 1, \cdots, M$.

- $(\theta_{R,l,j})$ is uniformly bounded according to $j$ in $BV(\Omega)$ since $\|\theta_{R,l,j}\|_{L^1(\Omega)} \leq \text{meas}(\Omega) < +\infty$ and $(TV(\theta_{R,l,j}))$ is uniformly bounded according to $j$ for all $l \in \{1, \cdots, N\}$.

- $(\theta_{\tilde{T}_i,l,j})$ is uniformly bounded according to $j$ in $BV(\Omega)$ since $\forall l \in \{1, \cdots, N\}$ and all $i = 1, \cdots, M$, $\|\theta_{\tilde{T}_i,l,j}\|_{L^1(\Omega)} \leq \text{meas}(\Omega)$ and $(TV(\theta_{\tilde{T}_i,l,j}))$ is uniformly bounded according to $j$.

- $(\text{Cof} V_{i,j})$ is uniformly bounded according to $j$ in $L^4(\Omega, M_3(\mathbb{R}))$ for all $i = 1, \cdots, M$.

- $(\det V_{i,j})$ is uniformly bounded according to $j$ in $L^2(\Omega)$ for all $i = 1, \cdots, M$.

- $(W_{i,j})$ is uniformly bounded according to $j$ in $L^\infty(\Omega, M_3(\mathbb{R}))$ and in $L^2(\Omega, M_3(\mathbb{R}))$ since $(V_{i,j})$ is uniformly bounded according to $j$ in $L^4(\Omega, M_3(\mathbb{R}))$ and thus in $L^2(\Omega, M_3(\mathbb{R}))$, for all $i = 1, \cdots, M$.

So, up to subsequences, we get that

$\varphi_{i,j} \underset{j \to +\infty}{\rightharpoonup} \bar{\varphi}_i$ in $W^{1,4}(\Omega, \mathbb{R}^3)$, for all $i = 1, \cdots, M$,

$V_{i,j} \underset{j \to +\infty}{\rightharpoonup} \bar{V}_i$ in $L^4(\Omega, M_3(\mathbb{R}))$, and $V_{i,j} \underset{j \to +\infty}{\overset{*}{\rightharpoonup}} \bar{V}_i$ in $L^\infty(\Omega, M_3(\mathbb{R}))$, for all $i = 1, \cdots, M$,

by uniqueness of the weak limit in $L^4(\Omega, M_3(\mathbb{R}))$ and by the continuous embedding of $L^\infty(\Omega, M_3(\mathbb{R})) \subset L^4(\Omega, M_3(\mathbb{R}))$,

$\mathrm{Cof}V_{i,j} \underset{j \to +\infty}{\rightharpoonup} \bar{X}_i$ in $L^4(\Omega, M_3(\mathbb{R}))$, for all $i = 1, \cdots, M$,

$\det V_{i,j} \underset{j \to +\infty}{\rightharpoonup} \bar{\delta}_i$ in $L^2(\Omega)$, for all $i = 1, \cdots, M$,

$W_{i,j} \underset{j \to +\infty}{\rightharpoonup} \bar{W}_i$ in $L^2(\Omega, M_3(\mathbb{R}))$, and $W_{i,j} \underset{j \to +\infty}{\overset{*}{\rightharpoonup}} \bar{W}_i$ in $L^\infty(\Omega, M_3(\mathbb{R}))$, for all $i = 1, \cdots, M$,

by uniqueness of the weak limit in $L^2(\Omega, M_3(\mathbb{R}))$ and by the continuous embedding of $L^\infty(\Omega, M_3(\mathbb{R})) \subset L^2(\Omega, M_3(\mathbb{R}))$,

$\theta_{T_i,l,j} \underset{j \to +\infty}{\longrightarrow} \bar{\theta}_{T_i,l}$ in $L^1(\Omega)$, with $\bar{\theta}_{T_i,l} \in BV(\Omega)$, $\forall l \in \{1, \cdots, N\}$, $\forall i \in \{1, \cdots, M\}$,

$\theta_{R,l,j} \underset{j \to +\infty}{\longrightarrow} \bar{\theta}_{R,l}$ in $L^1(\Omega)$, with $\bar{\theta}_{R,l} \in BV(\Omega)$, $\forall l \in \{1, \cdots, N\}$,

$\theta_{\tilde{T}_i,l,j} \underset{j \to +\infty}{\longrightarrow} \bar{\theta}_{\tilde{T}_i,l}$ in $L^1(\Omega)$, with $\bar{\theta}_{\tilde{T}_i,l} \in BV(\Omega)$, $\forall l \in \{1, \cdots, N\}$, $\forall i \in \{1, \cdots, M\}$.

Since the convergence is strong in $L^1(\Omega)$ then it is also almost everywhere up to subsequences and we deduce that $\bar{\theta}_{T_i,l} \in BV(\Omega, \{0,1\})$, $\bar{\theta}_{R,l} \in BV(\Omega, \{0,1\})$ and $\bar{\theta}_{\tilde{T}_i,l} \in BV(\Omega, \{-1, 0, 1\})$ for all $l \in \{1, \cdots, N\}$ and all $i \in \{1, \cdots, M\}$ and $\sum_{l=1}^{M} \bar{\theta}_{R,l}(x) = 1$ almost everywhere in $\Omega$, $\sum_{l=1}^{M} \bar{\theta}_{T_i,l}(x) = 1$ almost everywhere in $\Omega$ for all $i = 1, \cdots, M$, so that $\bar{\theta}_R \in \mathcal{U}$ and $\bar{\theta}_{T_i} \in \mathcal{U}$ for each $i = 1, \cdots, M$.

Let us set $z_{i,j,l} = \theta_{\tilde{T}_i,l,j} - (\theta_{T_i,l,j} \circ \varphi_{i,j} - \theta_{R,l,j}) \underset{j \to +\infty}{\longrightarrow} 0$ in $L^1(\Omega)$ for all $l \in \{1, \cdots, N\}$ and all $i \in \{1, \cdots, M\}$, $w_{i,j} = V_{i,j} - \nabla \varphi_{i,j} \underset{j \to +\infty}{\longrightarrow} 0$ in $L^4(\Omega, M_3(\mathbb{R}))$ for all $i = 1, \cdots, M$, and $y_{i,j} = W_{i,j} - V_{i,j}^{-1} \underset{j \to +\infty}{\longrightarrow} 0$ in $L^2(\Omega, M_3(\mathbb{R}))$, for all $i = 1, \cdots, M$.

Since strong convergence implies weak convergence, we have that

$$\forall i = 1, \cdots, M, \forall \Phi \in L^{\frac{4}{3}}(\Omega, M_3(\mathbb{R})), \int_\Omega (V_{i,j} - \nabla \varphi_{i,j}) : \Phi \, dx \underset{j \to +\infty}{\longrightarrow} 0.$$

But we also know that $\int_\Omega V_{i,j} : \Phi \, dx \underset{j \to +\infty}{\longrightarrow} \int_\Omega \bar{V}_i : \Phi \, dx$, and we get $\nabla \varphi_{i,j} \underset{j \to +\infty}{\rightharpoonup} \bar{V}_i$ in $L^4(\Omega, M_3(\mathbb{R}))$, for all $i = 1, \cdots, M$. By uniqueness of the weak limit in $L^4(\Omega, M_3(\mathbb{R}))$, we have that $\nabla \bar{\varphi}_i = \bar{V}_i \in L^\infty(\Omega, M_3(\mathbb{R}))$ for each $i = 1, \cdots, M$. We also have that

$$\forall i \in \{1, \cdots, M\}, \forall \Phi \in L^2(\Omega, M_3(\mathbb{R})), \int_\Omega W_{i,j} - V_{i,j}^{-1} : \Phi \, dx \underset{j \to +\infty}{\rightharpoonup} 0.$$

But we know that $\int_\Omega W_{i,j} : \Phi\,dx \xrightarrow[j\to+\infty]{} \int_\Omega \bar{W}_i : \Phi\,dx$, and we get $V_{i,j}^{-1} \xrightharpoonup[j\to+\infty]{} \bar{W}_i$ in $L^2(\Omega, M_3(\mathbb{R}))$. Let us now show that $\bar{W}_i = \bar{V}_i^{-1} = \nabla\bar{\varphi}_i^{-1}$, for all $i = 1, \cdots, M$.
We have that $\forall \Phi \in L^2(\Omega, M_3(\mathbb{R}))$:

$$\int_\Omega V_{i,j}^{-1}\bar{V}_i - I : \Phi\,dx = \int_\Omega V_{i,j}^{-1}\bar{V}_i - \bar{W}_i\bar{V}_i + \bar{W}_i\bar{V}_i - \bar{W}_i V_{i,j} + \bar{W}_i V_{i,j} - V_{i,j}^{-1}V_{i,j} : \Phi\,dx,$$
$$= \int_\Omega (V_{i,j}^{-1} - \bar{W}_i)\bar{V}_i : \Phi\,dx + \int_\Omega \bar{W}_i(\bar{V}_i - V_{i,j}) : \Phi\,dx + \int_\Omega (\bar{W}_i - V_{i,j}^{-1})V_{i,j} : \Phi\,dx,$$
$$= \int_\Omega (V_{i,j}^{-1} - \bar{W}_i) : \Phi\bar{V}_i^T\,dx + \int_\Omega (\bar{V}_i - V_{i,j}) : \bar{W}_i^T\Phi\,dx + \int_\Omega (\bar{W}_i - V_{i,j}^{-1}) : \Phi V_{i,j}^T\,dx,$$

with $\Phi\bar{V}_i^T \in L^2(\Omega, M_3(\mathbb{R}))$ since $\bar{V}_i \in L^\infty(\Omega, M_3(\mathbb{R}))$ so that $\int_\Omega (V_{i,j}^{-1} - \bar{W}) : \Phi\bar{V}_i^T\,dx \xrightarrow[j\to+\infty]{} 0$, $\bar{W}_i^T\Phi \in L^2(\Omega, M_3(\mathbb{R})) \subset L^{\frac{4}{3}}(\Omega, M_3(\mathbb{R}))$ since $\bar{W}_i \in L^\infty(\Omega, M_3(\mathbb{R}))$ so that $\int_\Omega (V_{i,j} - \bar{V}_i) : \bar{W}_i^T\Phi\,dx \xrightarrow[j\to+\infty]{} 0$, and $\int_\Omega (\bar{W}_i - V_{i,j}^{-1}) : \Phi V_{i,j}^T\,dx \leq \alpha\int_\Omega (\bar{W}_i - V_{i,j}^{-1}) : \Phi'\,dx$ with $\Phi' \in L^2(\Omega)$ since $\forall j \in \mathbb{N}, \forall i \in \{1, \cdots, M\}$, $\|V_{i,j}\|_{L^\infty(\Omega, M_3(\mathbb{R}))} \leq \alpha$ so that $\int_\Omega (\bar{W}_i - V_{i,j}^{-1}) : \Phi V_{i,j}^T\,dx \xrightarrow[j\to+\infty]{} 0$. We therefore obtain that $V_{i,j}^{-1}\bar{V}_i \xrightharpoonup[j\to+\infty]{} I_3$ in $L^2(\Omega, M_3(\mathbb{R}))$ and since $V_{i,j}^{-1} \xrightharpoonup[j\to+\infty]{} \bar{W}_i$ in $L^2(\Omega, M_3(\mathbb{R}))$, we deduce that $\bar{W}_i\bar{V}_i = I_3$ and consequently $\bar{W}_i = \bar{V}_i^{-1} = (\nabla\bar{\varphi}_i)^{-1} \in L^\infty(\Omega, M_3(\mathbb{R}))$, for all $i = 1, \cdots, M$.
Now, we consider $V_{i,j} = w_{i,j} + \nabla\varphi_{i,j}$ and $\mathrm{Cof}\,V_{i,j} = \mathrm{Cof}(w_{i,j} + \nabla\varphi_{i,j})$.
For the purpose of illustration since the same reasoning applies to other indices, we focus on the component of the first row, first column :

$$(\mathrm{Cof}\,V_{i,j})_{11} = \left(w_{i,22} + \frac{\partial\varphi_{i,2}}{\partial x_2}\right)\left(w_{i,33} + \frac{\partial\varphi_{i,3}}{\partial x_3}\right) - \left(w_{i,23} + \frac{\partial\varphi_{i,2}}{\partial x_3}\right)\left(w_{i,32} + \frac{\partial\varphi_{i,3}}{\partial x_2}\right),$$
$$= w_{i,22}w_{i,33} - w_{i,23}w_{i,32} + \frac{\partial\varphi_{i,3}}{\partial x_3}\frac{\partial\varphi_{i,2}}{\partial x_2} - \frac{\partial\varphi_{i,3}}{\partial x_2}\frac{\partial\varphi_{i,2}}{\partial x_3} + w_{i,22}\frac{\partial\varphi_{i,3}}{\partial x_3} + w_{i,33}\frac{\partial\varphi_{i,2}}{\partial x_2}$$
$$- w_{i,23}\frac{\partial\varphi_{i,3}}{\partial x_2} - w_{i,32}\frac{\partial\varphi_{i,2}}{\partial x_3},$$
$$= (\mathrm{Cof}\,w_{i,j})_{11} + (\mathrm{Cof}\,\nabla\varphi_{i,j})_{11} + w_{i,22}\frac{\partial\varphi_{i,3}}{\partial x_3} + w_{i,33}\frac{\partial\varphi_{i,2}}{\partial x_2} - w_{i,23}\frac{\partial\varphi_{i,3}}{\partial x_2} - w_{i,32}\frac{\partial\varphi_{i,2}}{\partial x_3}.$$

Let us set $d_{i,j} = (\mathrm{Cof}\,w_{i,j})_{11} + w_{i,22}\frac{\partial\varphi_{i,3}}{\partial x_3} + w_{i,33}\frac{\partial\varphi_{i,2}}{\partial x_2} - w_{i,23}\frac{\partial\varphi_{i,3}}{\partial x_2} - w_{i,32}\frac{\partial\varphi_{i,2}}{\partial x_3} \in L^2(\Omega)$ by generalized Hölder's inequality. Then using Hölder's inequality, one

can prove that

$$\|d_{i,j}\|_{L^2(\Omega)} \leq \|w_{i,22}\|_{L^4(\Omega)}\|\frac{\partial \varphi_{i,3}}{\partial x_3}\|_{L^4(\Omega)} + \|w_{i,33}\|_{L^4(\Omega)}\|\frac{\partial \varphi_{i,2}}{\partial x_2}\|_{L^4(\Omega)}$$
$$+ \|w_{i,23}\|_{L^4(\Omega)}\|\frac{\partial \varphi_{i,3}}{\partial x_2}\|_{L^4(\Omega)} + \|w_{i,32}\|_{L^4(\Omega)}\|\frac{\partial \varphi_{i,2}}{\partial x_3}\|_{L^4(\Omega)}$$
$$+ \|w_{i,22}\|_{L^4(\Omega)}\|w_{i,33}\|_{L^4(\Omega)} + \|w_{i,23}\|_{L^4(\Omega)}\|w_{i,32}\|_{L^4(\Omega)},$$
$$\leq (\|w_{i,22}\|^4_{L^4(\Omega)} + \|w_{i,33}\|^4_{L^4(\Omega)} + \|w_{i,23}\|^4_{L^4(\Omega)} + \|w_{i,32}\|^4_{L^4(\Omega)})^{\frac{1}{4}}$$
$$(\|\frac{\partial \varphi_{i,2}}{\partial x_3}\|^{\frac{4}{3}}_{L^4(\Omega)} + \|\frac{\partial \varphi_{i,3}}{\partial x_3}\|^{\frac{4}{3}}_{L^4(\Omega)} + \|\frac{\partial \varphi_{i,2}}{\partial x_2}\|^{\frac{4}{3}}_{L^4(\Omega)} + \|\frac{\partial \varphi_{i,3}}{\partial x_2}\|^{\frac{4}{3}}_{L^4(\Omega)})^{\frac{3}{4}}$$
$$+ (\|w_{i,22}\|^4_{L^4(\Omega)} + \|w_{i,33}\|^4_{L^4(\Omega)} + \|w_{i,23}\|^4_{L^4(\Omega)} + \|w_{i,32}\|^4_{L^4(\Omega)})^{\frac{1}{4}}$$
$$(\|w_{i,22}\|^{\frac{4}{3}}_{L^4(\Omega)} + \|w_{i,33}\|^{\frac{4}{3}}_{L^4(\Omega)} + \|w_{i,23}\|^{\frac{4}{3}}_{L^4(\Omega)} + \|w_{i,32}\|^{\frac{4}{3}}_{L^4(\Omega)})^{\frac{3}{4}},$$
$$\leq c_4 c_2^2 \|w_{i,j}\|_{L^4(\Omega, M_3(\mathbb{R}))} \|\nabla \varphi_{i,j}\|_{L^4(\Omega, M_3(\mathbb{R}))} + c_4 c_2^2 \|w_{i,j}\|^2_{L^4(\Omega, M_3(\mathbb{R}))}.$$

Indeed, thanks to the property of equivalence of norms in finite dimension, there exist $(c_1, c_2, c_3, c_4) \in \mathbb{R}^4$ such that $\forall x \in \mathbb{R}^9$, $c_1 \|x\|_2 \leq \|x\|_4 \leq c_2 \|x\|_2$ and $c_3 \|x\|_4 \leq \|x\|_{\frac{4}{3}} \leq c_4 \|x\|_4 \leq c_4 c_2 \|x\|_2$.

As $(\varphi_{i,j})$ is uniformly bounded according to $j$ in $W^{1,4}(\Omega, \mathbb{R}^3)$ and as $w_{i,j} \underset{j \to +\infty}{\longrightarrow} 0$ strongly in $L^4(\Omega)$, then

$$\|d_{i,j}\|_{L^2(\Omega)} \underset{j \to +\infty}{\longrightarrow} 0,$$

and

$$\int_\Omega (\text{Cof} V_{i,j})_{11} \Phi \, dx = \int_\Omega d_{i,j} \Phi \, dx + \int_\Omega (\text{Cof} \nabla \varphi_{i,j})_{11} \Phi \, dx \underset{j \to +\infty}{\longrightarrow} \int_\Omega (\text{Cof} \nabla \bar{\varphi}_i)_{11} \Phi \, dx, \forall \Phi \in L^2(\Omega),$$

$\forall i \in \{1, \cdots, M\}$. Indeed, from [3, Theorem 8.20], as $\varphi_{i,j} \underset{j \to +\infty}{\rightharpoonup} \bar{\varphi}_i$ in $W^{1,4}(\Omega, \mathbb{R}^3)$, then $\text{Cof} \nabla \varphi_{i,j} \underset{j \to +\infty}{\rightharpoonup} \text{Cof} \nabla \bar{\varphi}_i$ in $L^2(\Omega, M_3(\mathbb{R}))$. Thus $\text{Cof} V_{i,j} \underset{j \to +\infty}{\rightharpoonup} \text{Cof} \nabla \bar{\varphi}_i$ in $L^2(\Omega)$ and by uniqueness of the weak limit in $L^2(\Omega)$, $\bar{X}_i = \text{Cof} \nabla \bar{\varphi}_i$ and $\text{Cof} V_{i,j} \underset{j \to +\infty}{\rightharpoonup} \text{Cof} \nabla \bar{\varphi}_i$ in $L^4(\Omega)$.

We then have from the generalized Hölder's inequality :

$$\det V_{i,j} = V_{i,11} (\text{Cof} V_{i,j})_{11} + V_{i,21} (\text{Cof} V_{i,j})_{21} + V_{i,31} (\text{Cof} V_{i,j})_{31} \in L^{\frac{4}{3}}(\Omega),$$
$$= (w_{i,11} + \frac{\partial \varphi_{i,1}}{\partial x_1})((\text{Cof} \nabla \varphi_{i,j})_{11} + d_{i,j}) + \text{ similar components}.$$

Then $\forall \Phi \in L^4(\Omega)$, and for all $i = 1, \cdots, M$,

$$\int_\Omega \det V_{i,j} : \Phi \, dx = \int_\Omega w_{i,11} (\text{Cof} \nabla \varphi_{i,j})_{11} \Phi \, dx + \int_\Omega w_{i,11} d_{i,j} \Phi \, dx + \int_\Omega \frac{\partial \varphi_{i,1}}{\partial x_1} d_{i,j} \Phi \, dx$$
$$+ \int_\Omega \frac{\partial \varphi_{i,1}}{\partial x_1} (\text{Cof} \nabla \varphi_{i,j})_{11} \Phi \, dx + \text{ similar components}.$$

But we have for all $i = 1, \cdots, M$,

$$\left|\int_\Omega w_{i,11}(\mathrm{Cof}\nabla\varphi_{i,j})_{11}\Phi\,dx\right| \leq \|w_{i,11}\|_{L^4(\Omega)}\|\mathrm{Cof}\nabla\varphi_{i,j}\|_{L^2(\Omega)}\|\Phi\|_{L^4(\Omega)},$$

$$\left|\int_\Omega w_{i,11}d_{i,j}\Phi\,dx\right| \leq \|w_{i,11}\|_{L^4(\Omega)}\|d_{i,j}\|_{L^2(\Omega)}\|\Phi\|_{L^4(\Omega)},$$

$$\left|\int_\Omega \frac{\partial\varphi_{i,1}}{\partial x_1}d_{i,j}\Phi\,dx\right| \leq \|\frac{\partial\varphi_{i,1}}{\partial x_1}\|_{L^4(\Omega)}\|d_{i,j}\|_{L^2(\Omega)}\|\Phi\|_{L^4(\Omega)},$$

with $\|w_{i,11}\|_{L^4(\Omega)} \underset{j\to+\infty}{\longrightarrow} 0$, $\|d_{i,j}\|_{L^2(\Omega)} \underset{j\to+\infty}{\longrightarrow} 0$ and $\|\mathrm{Cof}\nabla\varphi_{i,j}\|_{L^2(\Omega)}$ uniformly bounded since $\mathrm{Cof}\nabla\varphi_{i,j} \underset{j\to+\infty}{\rightharpoonup} \mathrm{Cof}\nabla\bar\varphi_i$ in $L^4(\Omega, M_3(\mathbb{R}))$ and thus in $L^2(\Omega, M_3(\mathbb{R}))$, $\|\frac{\partial\varphi_{i,1}}{\partial x_1}\|_{L^4(\Omega)}$ uniformly bounded since $\varphi_{i,j} \underset{j\to+\infty}{\rightharpoonup} \bar\varphi_i$ in $W^{1,4}(\Omega)$ and $\|\Phi\|_{L^4(\Omega)}$ uniformly bounded.

Since $\varphi_{i,j} \underset{j\to+\infty}{\rightharpoonup} \bar\varphi_i$ in $W^{1,4}(\Omega,\mathbb{R}^3)$, then $\det\nabla\varphi_{i,j} \underset{j\to+\infty}{\rightharpoonup} \det\nabla\bar\varphi_i$ in $L^{\frac{4}{3}}(\Omega)$ and $\int_\Omega \frac{\partial\varphi_{i,1}}{\partial x_1}(\mathrm{Cof}\nabla\varphi_{i,j})_{11}\Phi\,dx \underset{j\to+\infty}{\longrightarrow} \int_\Omega \frac{\partial\bar\varphi_{i,1}}{\partial x_1}(\mathrm{Cof}\nabla\bar\varphi_i)_{11}\Phi\,dx$, for all $i = 1,\cdots,M$. Finally, $\int_\Omega \det V_{i,j}\Phi\,dx \underset{j\to+\infty}{\longrightarrow} \int_\Omega \det\nabla\bar\varphi_i\Phi\,dx$, $\forall\Phi \in L^4(\Omega)$ and so $\det V_{i,j} \underset{j\to+\infty}{\rightharpoonup} \det\nabla\bar\varphi_i$ in $L^{\frac{4}{3}}(\Omega)$ and by uniqueness of the weak limit, $\bar\delta_i = \det\nabla\bar\varphi_i$ and $\det V_{i,j} \underset{j\to+\infty}{\rightharpoonup} \det\nabla\bar\varphi_i$ in $L^2(\Omega)$, for all $i = 1,\cdots,M$.

We know that $W_{Op}$ is convex and continuous. If $\psi_n \underset{n\to+\infty}{\longrightarrow} \psi$ in $L^4(\Omega, M_3(\mathbb{R}))$ and we can extract a subsequence still denoted $(\psi_n)$ such that $\nabla\psi_n \underset{n\to+\infty}{\longrightarrow} \nabla\psi$ almost everywhere on $\Omega$. If $\alpha_n \underset{n\to+\infty}{\longrightarrow} \bar\alpha$ in $L^4(\Omega, M_3(\mathbb{R}))$, then we can extract a subsequence still denoted $(\alpha_n)$ such that $\alpha_n \underset{n\to+\infty}{\longrightarrow} \bar\alpha$ almost everywhere on $\Omega$. If $\delta_n \underset{n\to+\infty}{\longrightarrow} \bar\delta$ in $L^2(\Omega)$, then we can extract a subsequence still denoted $(\delta_n)$ such that $\delta_n \underset{n\to+\infty}{\longrightarrow} \bar\delta$ almost everywhere on $\Omega$. Then by continuity of $W_{Op}$, we get that $W_{Op}(\psi_n, \alpha_n, \delta_n) \underset{j\to+\infty}{\longrightarrow} W_{Op}(\psi, \bar\alpha, \bar\delta)$ almost everywhere on $\Omega$. Then by applying Fatou's lemma, we have that $\liminf\limits_{n\to+\infty}\int_\Omega W_{Op}(\psi_n, \alpha_n, \delta_n)\,dx \geq \int_\Omega W_{Op}(\psi, \bar\alpha, \bar\delta)\,dx$. Since $W_{Op}$ is convex, so is $\int_\Omega W_{Op}(\xi, \alpha, \delta)\,dx$ and [2, Corollaire III.8 p.38] applies. Therefore $\int_\Omega W_{Op}(\xi, \alpha, \delta)\,dx$ is also weakly lower semicontinuous in $L^4(\Omega, M_3(\mathbb{R})) \times L^4(\Omega, M_3(\mathbb{R})) \times L^2(\Omega)$ and we deduce that $+\infty > \liminf\limits_{j\to+\infty}\int_\Omega W_{Op}(V_{i,j}, \mathrm{Cof}V_{i,j}, \det V_{i,j})\,dx \geq \int_\Omega W_{Op}(\nabla\bar\varphi_i, \mathrm{Cof}\nabla\bar\varphi_i, \det\nabla\bar\varphi_i)\,dx$.

Using the same arguments as in the previous proof, we get more regularity of $\bar\varphi_i$, that is to say, $\bar\varphi_i$ is a bi-Lipschitz homeomorphism for each $i = 1,\cdots, M$. By the weak-$*$ lower semicontinuity of $\|.\|_{L^\infty(\Omega, M_3(\mathbb{R}))}$, we have that $\|\nabla\bar\varphi_i\|_{L^\infty(\Omega, M_3(\mathbb{R}))} \leq \liminf\limits_{j\to+\infty}\|V_{i,j}\|_{L^\infty(\Omega, M_3(\mathbb{R}))} \leq \alpha$, $\|(\nabla\bar\varphi_i)^{-1}\|_{L^\infty(\Omega, M_3(\mathbb{R}))} \leq \liminf\limits_{j\to+\infty}\|W_{i,j}\|_{L^\infty(\Omega, M_3(\mathbb{R}))} \leq \beta$ and $\mathbb{1}_{\{\|.\|_{L^\infty(\Omega, M_3(\mathbb{R}))}\leq\alpha\}}(\nabla\bar\varphi_i) \leq \liminf\limits_{j\to+\infty}\mathbb{1}_{\{\|.\|_{L^\infty(\Omega, M_3(\mathbb{R}))}\leq\alpha\}}(V_{i,j})$, $\mathbb{1}_{\{\|.\|_{L^\infty(\Omega, M_3(\mathbb{R}))}\leq\beta\}}((\nabla\bar\varphi_i)^{-1}) \leq \liminf\limits_{j\to+\infty}\mathbb{1}_{\{\|.\|_{L^\infty(\Omega, M_3(\mathbb{R}))}\leq\beta\}}(W_{i,j})$.

We also know that $\varphi_{i,j} \xrightarrow[j\to+\infty]{} \bar{\varphi}_i$ in $\mathcal{C}^{0,\alpha}(\Omega)$, $\bar{\varphi}_i^{-1} \in \mathcal{C}^{0,\alpha'}(\Omega)$ and $\varphi_{i,j} \circ \bar{\varphi}_i^{-1} \xrightarrow[j\to+\infty]{}$ Id in $\mathcal{C}^{0,\alpha''}(\Omega)$ with $\alpha'' = \alpha'\alpha$, $\alpha < \frac{1}{4}$ and $\alpha' < 1 - \frac{3}{q}$, using the same arguments as in the previous proof. Let us prove that $\theta_{T_i,l,j} \circ \varphi_{i,j} \xrightarrow[j\to+\infty]{} \bar{\theta}_{T_i,l} \circ \bar{\varphi}_i$ in $L^1(\Omega)$ for all $i = 1, \cdots, M$.

Since $\varphi_{i,j} \circ \bar{\varphi}_i^{-1} \xrightarrow[j\to+\infty]{}$ Id in $\mathcal{C}^{0,\alpha''}(\Omega)$, then it converges uniformly and so

$$\forall \varepsilon > 0, \exists k \in \mathbb{N}, \forall j \in \mathbb{N}, j \geq k \Rightarrow \sup_{x\in\Omega}|\varphi_{i,j} \circ \bar{\varphi}_i^{-1}(x) - x| \leq \varepsilon.$$

Also given $\eta > 0$, by Egorov's theorem, there exists a set $S_\eta$ with Lebesgue measure $\text{meas}(S_\eta) \leq \eta$ such that $\theta_{T_i,l,j} \xrightarrow[j\to+\infty]{} \bar{\theta}_{T_i,l}$ uniformly on $\Omega \setminus S_\eta$.

Now we choose $K \in \mathbb{N}$ large enough such that $\forall j \in \mathbb{N}$, if $j \geq K$, then

$$|\varphi_{i,j} \circ \bar{\varphi}_i^{-1}(x) - x| \leq \eta, \text{ on } \Omega,$$
$$|\theta_{T_i,l,j}(x) - \bar{\theta}_{T_i,l}(x)| \leq \eta, \text{ on } \Omega \setminus S_\eta.$$

Then by denoting $\mathcal{O}_i = \{x \in \Omega : \bar{\theta}_{T_i,l}(x) = 1\}$,

$$\|\theta_{T_i,l,j} \circ \varphi_{i,j} - \bar{\theta}_{T_i,l} \circ \bar{\varphi}_i\|_{L^1(\Omega)} = \int_{\bar{\varphi}_i^{-1}(\mathcal{O}_i)} |\theta_{T_i,l,j} \circ \varphi_{i,j}(x) - \bar{\theta}_{T_i,l} \circ \bar{\varphi}_i(x)|\, dx$$

$$+ \int_{\bar{\varphi}_i^{-1}(\Omega\setminus\mathcal{O}_i)} |\theta_{T_i,l,j} \circ \varphi_{i,j}(x) - \bar{\theta}_{T_i,l} \circ \bar{\varphi}_i(x)|\, dx,$$

$$= \int_{\mathcal{O}_i} \frac{|\theta_{T_i,l,j} \circ \varphi_{i,j} \circ \bar{\varphi}_i^{-1}(x) - 1|}{\det\nabla\bar{\varphi}_i(\bar{\varphi}_i^{-1}(x))}\, dx + \int_{\Omega\setminus\mathcal{O}_i} \frac{|\theta_{T_i,l,j} \circ \varphi_{i,j} \circ \bar{\varphi}_i^{-1}(x)|}{\det\nabla\bar{\varphi}_i(\bar{\varphi}_i^{-1}(x))}\, dx,$$

$$\leq \int_{(\mathcal{O}_i\setminus(S_\eta))_\eta} \frac{|\theta_{T_i,l,j} \circ \varphi_{i,j} \circ \bar{\varphi}_i^{-1}(x) - 1|}{\det\nabla\bar{\varphi}_i(\bar{\varphi}_i^{-1}(x))}\, dx + \int_{((\Omega\setminus\mathcal{O}_i)\setminus(S_\eta))_\eta} \frac{|\theta_{T_i,l,j} \circ \varphi_{i,j} \circ \bar{\varphi}_i^{-1}(x)|}{\det\nabla\bar{\varphi}_i(\bar{\varphi}_i^{-1}(x))}\, dx$$

$$+ \int_{S_\eta} \frac{|\theta_{T_i,l,j} \circ \varphi_{i,j} \circ \bar{\varphi}_i^{-1}(x) - \bar{\theta}_{T_i,l}(x)|}{\det\nabla\bar{\varphi}_i(\bar{\varphi}_i^{-1}(x))}\, dx + \int_{\mathcal{O}_i^\eta \cup S_\eta^\eta \cup \Omega^\eta} \frac{|\theta_{T_i,l,j} \circ \varphi_{i,j} \circ \bar{\varphi}_i^{-1}(x) - \bar{\theta}_{T_i,l}(x)|}{\det\nabla\bar{\varphi}_i(\bar{\varphi}_i^{-1}(x))}\, dx,$$

with $(\mathcal{O}_i\setminus S_\eta)_\eta = \{x \in \mathcal{O}_i\setminus S_\eta \,|\, \text{dist}(x, \partial(\mathcal{O}_i\setminus S_\eta)) > \eta\}$, $S_\eta^\eta = \{x \in \Omega \,|\, \text{dist}(x, \partial S_\eta) \leq \eta\}$. So $\forall x \in (\mathcal{O}_i \setminus (S_\eta))_\eta$, $\varphi_{i,j} \circ \bar{\varphi}_i^{-1}(x) \in \mathcal{O}_i \setminus S_\eta$, and $\forall x \in ((\Omega \setminus \mathcal{O}_i) \setminus (S_\eta))_\eta$, $\varphi_{i,j} \circ \bar{\varphi}_i^{-1}(x) \in (\Omega \setminus \mathcal{O}_i) \setminus S_\eta$. Then

$$\|\theta_{T_i,l,j} \circ \varphi_{i,j} - \bar{\theta}_{T_i,l} \circ \bar{\varphi}_i\|_{L^1(\Omega)} \leq \int_{(\mathcal{O}_i\setminus S_\eta)_\eta} \frac{\eta}{\det\nabla\bar{\varphi}_i(\bar{\varphi}_i^{-1}(x))}\, dx$$

$$+ \int_{((\Omega\setminus\mathcal{O}_i)\setminus S_\eta)_\eta} \frac{\eta}{\det\nabla\bar{\varphi}_i(\bar{\varphi}_i^{-1}(x))}\, dx$$

$$+ \int_{S_\eta} \frac{1}{\det\nabla\bar{\varphi}_i(\bar{\varphi}_i^{-1}(x))}\, dx$$

$$+ \int_{\mathcal{O}_i^\eta \cup S_\eta^\eta \cup \Omega^\eta} \frac{1}{\det\nabla\bar{\varphi}_i(\bar{\varphi}_i^{-1}(x))}\, dx,$$

taking into account that $\theta_{T_i,l,j} \in BV(\Omega, \{0,1\})$ and $\bar{\theta}_{T_i,l} \in BV(\Omega, \{0,1\})$. Using Hölder's inequality, we get

$$\|\theta_{T_i,l,j} \circ \varphi_{i,j} - \bar{\theta}_{T_i,l} \circ \bar{\varphi}_i\|_{L^1(\Omega)} \leq \eta \left(\int_\Omega dx\right)^{\frac{s}{s+1}} \left(\int_\Omega \frac{1}{\det\nabla\bar{\varphi}_i(\bar{\varphi}_i^{-1}(x))^{s+1}} dx\right)^{\frac{1}{s+1}}$$

$$+ \eta \left(\int_\Omega dx\right)^{\frac{s}{s+1}} \left(\int_\Omega \frac{1}{\det\nabla\bar{\varphi}_i(\bar{\varphi}_i^{-1}(x))^{s+1}} dx\right)^{\frac{1}{s+1}}$$

$$+ \left(\int_{S_\eta} dx\right)^{\frac{s}{s+1}} \left(\int_\Omega \frac{1}{\det\nabla\bar{\varphi}_i(\bar{\varphi}_i^{-1}(x))^{s+1}} dx\right)^{\frac{1}{s+1}}$$

$$+ \left(\int_{S_\eta^\eta \cup \mathcal{O}_i^\eta \cup \Omega^\eta} dx\right)^{\frac{s}{s+1}} \left(\int_\Omega \frac{1}{\det\nabla\bar{\varphi}_i(\bar{\varphi}_i^{-1}(x))^{s+1}} dx\right)^{\frac{1}{s+1}},$$

$$\|\theta_{T_i,l,j} \circ \varphi_{i,j} - \bar{\theta}_{T_i,l} \circ \bar{\varphi}_i\|_{L^1(\Omega)} \leq \eta \left(\mathrm{meas}(\Omega)\right)^{\frac{s}{s+1}} \left(\int_\Omega \frac{1}{\det\nabla\bar{\varphi}_i(x)^s} dx\right)^{\frac{1}{s+1}}$$

$$+ \eta \left(\mathrm{meas}(\Omega)\right)^{\frac{s}{s+1}} \left(\int_\Omega \frac{1}{\det\nabla\bar{\varphi}_i(x)^s} dx\right)^{\frac{1}{s+1}}$$

$$+ \left(\mathrm{meas}(S_\eta)\right)^{\frac{s}{s+1}} \left(\int_\Omega \frac{1}{\det\nabla\bar{\varphi}_i(x)^s} dx\right)^{\frac{1}{s+1}}$$

$$+ \left(\mathrm{meas}(S_\eta^\eta) + \mathrm{meas}(\mathcal{O}_i^\eta) + \mathrm{meas}(\Omega^\eta)\right)^{\frac{s}{s+1}} \left(\int_\Omega \frac{1}{\det\nabla\bar{\varphi}_i(x)^s} dx\right)^{\frac{1}{s+1}}.$$

With $s = 10$ and the finiteness of $\int_\Omega W_{Op}(\nabla\bar{\varphi}_i, \mathrm{Cof}\nabla\bar{\varphi}_i, \det\nabla\bar{\varphi}_i)\, dx$, there exists $C \in \mathbb{R}_*^+$ such that

$$\|\theta_{T_i,l,j} \circ \varphi_{i,j} - \bar{\theta}_{T_i,l} \circ \bar{\varphi}_i\|_{L^1(\Omega)} \leq \eta \left(\mathrm{meas}(\Omega)\right)^{\frac{s}{s+1}} C^{\frac{1}{s+1}} + \eta \left(\mathrm{meas}(\Omega)\right)^{\frac{s}{s+1}} C^{\frac{1}{s+1}} + \left(\mathrm{meas}(S_\eta)\right)^{\frac{s}{s+1}}$$

$$C^{\frac{1}{s+1}} + \left(\int_{\partial S_\eta}\int_{B(x,\eta)} dy\, dx + \int_{\partial \mathcal{O}_i}\int_{B(x,\eta)} dy\, dx + \int_{\partial \Omega}\int_{B(x,\eta)} dy\, dx\right)^{\frac{s}{s+1}}$$

$$C^{\frac{1}{s+1}},$$

$$\|\theta_{T_i,l,j} \circ \varphi_{i,j} - \bar{\theta}_{T_i,l} \circ \bar{\varphi}_i\|_{L^1(\Omega)} \leq \eta \left(\mathrm{meas}(\Omega)\right)^{\frac{s}{s+1}} C^{\frac{1}{s+1}}$$

$$+ \eta \left(\mathrm{meas}((\Omega))\right)^{\frac{s}{s+1}} C^{\frac{1}{s+1}}$$

$$+ (\eta)^{\frac{s}{s+1}} C^{\frac{1}{s+1}}$$

$$+ \left(\frac{4\pi}{3}\eta^3(\mathrm{meas}(\partial S_\eta) + \mathrm{meas}(\partial \mathcal{O}_i) + \mathrm{meas}(\partial \Omega))\right)^{\frac{s}{s+1}} C^{\frac{1}{s+1}}.$$

By letting $\eta$ tend to 0 and so $j$ to infinity, we finally get for all $i = 1, \cdots, M$

$$\|\theta_{T_i,l,j} \circ \varphi_{i,j} - \bar{\theta}_{T_i,l} \circ \bar{\varphi}_i\|_{L^1(\Omega)} \underset{j \to +\infty}{\longrightarrow} 0.$$

This is true for all $l \in \{1, \cdots, N\}$.
Finally, for all $i = 1, \cdots, M$ and all $l = 1, \cdots, N$:

$$\|\bar{\theta}_{\tilde{T}_i,l} - \bar{\theta}_{T_i,l} \circ \bar{\varphi}_i + \bar{\theta}_{R,l}\|_{L^1(\Omega)} \leq \|\theta_{\tilde{T}_i,l,j} - \theta_{T_i,l,j} \circ \varphi_{i,j} + \theta_{R,l,j}\|_{L^1(\Omega)}$$
$$+ \|\bar{\theta}_{\tilde{T}_i,l} - \theta_{\tilde{T}_i,l,j}\|_{L^1(\Omega)} + \|\bar{\theta}_{T_i,l} \circ \bar{\varphi}_i - \theta_{T_i,l,j} \circ \varphi_{i,j}\|_{L^1(\Omega)}$$
$$+ \|\bar{\theta}_{R,l} - \theta_{R,l,j}\|_{L^1(\Omega)},$$
$$\underset{j \to +\infty}{\longrightarrow} 0,$$

from what precedes. We deduce that $\bar{\theta}_{\tilde{T}_i,l} = \bar{\theta}_{T_i,l} \circ \bar{\varphi}_i - \bar{\theta}_{R,l}$ in $L^1(\Omega)$ and so almost everywhere in $\Omega$ and $\theta_{\tilde{T}_i,l,j} \underset{j \to +\infty}{\longrightarrow} \bar{\theta}_{T_i,l} \circ \bar{\varphi}_i - \bar{\theta}_{R,l}$ in $L^1(\Omega)$ for all $i = 1, \cdots, M$ and all $l = 1, \cdots, N$. By the semicontinuity of the total variation we get that

$$\frac{\lambda}{2} \sum_{l=1}^{N} TV(\bar{\theta}_{T_i,l} \circ \bar{\varphi}_i - \bar{\theta}_{R,l}) \leq \frac{\lambda}{2} \liminf_{j \to +\infty} \sum_{l=1}^{N} TV(\theta_{\tilde{T}_i,l,j}).$$

Furthermore, we have that $\theta_{T_i,l,j} \underset{j \to +\infty}{\longrightarrow} \bar{\theta}_{T_i,l}$ in $L^1(\Omega)$ and so almost everywhere in $\Omega$ up to a subsequence. Then $\theta_{T_i,l,j} T_i \underset{j \to +\infty}{\longrightarrow} \bar{\theta}_{T_i,l} T_i$ almost everywhere in $\Omega$ with $|\theta_{T_i,l,j} T_i| \leq \|T_i\|_{L^\infty(\Omega)} \in L^1(\Omega)$ since $\Omega$ is bounded, and by the dominated convergence theorem one gets $\int_\Omega \theta_{T_i,l,j} T_i \, dx \underset{j \to +\infty}{\longrightarrow} \int_\Omega \bar{\theta}_{T_i,l} T_i \, dx$. Using the same reasoning, we get $\int_\Omega \theta_{T_i,l,j} \, dx \underset{j \to +\infty}{\longrightarrow} \int_\Omega \bar{\theta}_{T_i,l} \, dx$. As $T_i \in W^{1,\infty}(\Omega)$, we denote $\kappa_i$ the Lipschitz constant related to $T_i$. Then the chain rule applies and we have that $T_i \circ \varphi_i^k \in W^{1,4}(\Omega)$, $T_i \circ \bar{\varphi}_i \in W^{1,4}(\Omega)$ and $\|T_i \circ \varphi_{i,j} - T_i \circ \bar{\varphi}_i\|_{\mathcal{C}^0(\Omega)} \leq \kappa_i \|\varphi_{i,j} - \bar{\varphi}_i\|_{\mathcal{C}^0(\Omega)} \underset{j \to +\infty}{\longrightarrow} 0$ using the Sobolev compact embedding $W^{1,4}(\Omega) \hookrightarrow_c \mathcal{C}^0(\Omega)$. Therefore, up to a subsequence $\theta_{R,l,j} T_i \circ \varphi_{i,j} \underset{j \to +\infty}{\longrightarrow} \bar{\theta}_{R,l} T_i \circ \bar{\varphi}_i$ almost everywhere in $\Omega$, $\theta_{R,l,j} \underset{j \to +\infty}{\longrightarrow} \bar{\theta}_{R,l}$ almost everywhere in $\Omega$ with $|\theta_{R,l,j} T_i \circ \varphi_{i,j}| \leq \|T_i\|_{L^\infty(\Omega)} \in L^1(\Omega)$ and $|\theta_{R,l,j}| \leq 1 \in L^1(\Omega)$ almost everywhere in $\Omega$. Thus the dominated convergence theorem applies and we get

$$\frac{1}{M} \sum_{i=1}^{M} \int_\Omega \theta_{R,l,j} T_i \circ \varphi_{i,j} \, dx \underset{j \to +\infty}{\longrightarrow} \frac{1}{M} \sum_{i=1}^{M} \int_\Omega \bar{\theta}_{R,l} T_i \circ \bar{\varphi}_i \, dx,$$
$$\int_\Omega \theta_{R,l,j} \, dx \underset{k \to +\infty}{\longrightarrow} \int_\Omega \bar{\theta}_{R,l} \, dx,$$
$$c_{R,l,j} \underset{k \to +\infty}{\longrightarrow} \bar{c}_{R,l}.$$

Then $\forall l \in \{1, \cdots, N\}$ and for all $i = 1, \cdots, M$,

$$\theta_{T_i,l,j}(c_{T_i,l,j} - T_i)^2 \underset{j \to +\infty}{\longrightarrow} \bar{\theta}_{T_i,l}(\bar{c}_{T_i,l} - T_i)^2,$$
$$\theta_{R,l,j}(c_{R,l,j} - R)^2 \underset{j \to +\infty}{\longrightarrow} \bar{\theta}_{R,l}(\bar{c}_{R,l} - R)^2,$$

almost everywhere on $\Omega$ with $|\theta_{T_i,l,j}(c_{T_i,l,j} - T_i)^2| \leq 4\|T_i\|^2_{L^\infty(\Omega)} \in L^1(\Omega)$ and $|\theta_{R,l,j}(c_{R,l,j} - R)^2| \leq 4\|T_i\|^2_{L^\infty(\Omega)} \in L^1(\Omega)$. By applying the dominated convergence theorem, we have for all $l \in \{1, \cdots, N\}$:

$$\|\theta_{T_i,l,j}(c_{T_i,l,j} - T_i)^2\|_{L^1(\Omega)} \underset{j\to+\infty}{\longrightarrow} \|\bar\theta_{T_i,l}(\bar c_{T_i,l} - T_i)^2\|_{L^1(\Omega)},$$

$$\|\theta_{R,l,j}(c_{R,l,j} - R)^2\|_{L^1(\Omega)} \underset{j\to+\infty}{\longrightarrow} \|\bar\theta_{R,l}(\bar c_{R,l} - R)^2\|_{L^1(\Omega)}.$$

By combining all the results, we finally get

$$\frac{1}{M}\sum_{i=1}^M \frac{\gamma_T}{2} \sum_{l=1}^N TV(\bar\theta_{T_i,l}) + \int_\Omega \sum_{l=1}^N \bar\theta_{T_i,l}(\bar c_{T_i,l} - T_i)^2 \, dx$$
$$+ \frac{\gamma_R}{2}\sum_{l=1}^N TV(\bar\theta_{R,l}) + \int_\Omega \sum_{l=1}^N \bar\theta_{R,l}(\bar c_{R,l} - R)^2 \, dx$$
$$+ \frac{\lambda}{2}\sum_{l=1}^N TV(\bar\theta_{T_i,l} \circ \bar\varphi_i - \bar\theta_{R,l}) + \int_\Omega W_{Op}(\nabla\bar\varphi_i, \mathrm{Cof}\nabla\bar\varphi_i, \det\nabla\bar\varphi_i) \, dx$$
$$+ \mathbb{1}_{\{\|\cdot\|_{L^\infty(\Omega, M_3(\mathbb{R}))} \leq \alpha\}}(\nabla\bar\varphi_i) + \mathbb{1}_{\{\|\cdot\|_{L^\infty(\Omega, M_3(\mathbb{R}) \leq \beta\}}}((\nabla\bar\varphi_i)^{-1})$$
$$\leq \liminf_{j\to+\infty} \mathcal{F}_{1,\gamma_j}(\{\varphi_{i,j}, \theta_{T_i,j}, V_{i,j}, W_{i,j}\}_{i=1}^M, (\theta_{\tilde T_i,l,j})_{\substack{i=1,\cdots,M \\ l=1,\cdots,N}}, \theta_{R,j}),$$
$$\leq \limsup_{j\to+\infty} \mathcal{F}_{1,\gamma_j}(\{\varphi_{i,j}, \theta_{T_i,j}, V_{i,j}, W_{i,j}\}_{i=1}^M, (\theta_{\tilde T_i,l,j})_{\substack{i=1,\cdots,M \\ l=1,\cdots,N}}, \theta_{R,j}),$$
$$\leq \limsup_{j\to+\infty} \inf \mathcal{F}_1(\theta_R, \{\theta_{T_i}, \varphi_i\}_{i=1}^M) + \frac{2}{\gamma_j} = \inf \mathcal{F}_i(\theta_R, \{\theta_{T_i}, \varphi_i\}_{i=1}^M).$$

By finiteness of $\int_\Omega W_{Op}(\nabla\bar\varphi_i, \mathrm{Cof}\nabla\bar\varphi_i, \det\nabla\bar\varphi_i) \, dx$ and by continuity of the trace operator, we have that $\bar\varphi_i \in \hat{\mathcal{W}}$. By strong convergence in $L^1(\Omega)$ and so almost everywhere, we get that $\bar\theta_{T_i} \in \mathcal{U}$ and $\bar\theta_R \in \mathcal{U}$ with $\bar\theta_{T_i,l} \circ \bar\varphi_i - \bar\theta_{R,l} \in BV(\Omega, \{-1, 0, 1\})$ for all $l \in \{1, \cdots, N\}$ and all $i \in \{1, \cdots, N\}$. Therefore our decoupled problem (DP) converges to the initial problem (P) as $\gamma$ tends to infinity.

## 3 Detailed presentation of the first two methods to perform PCA

### 3.1 First Approach : Linearisation around the Identity

Let us recall the definition of the Cauchy-stress tensors and the Cauchy axiom relating, in a hyperelastic framework, the deformations and the forces applied on the material.

**Definition 1 (Cauchy-stress tensor).** *In a hyperelastic framework, the Cauchy stress tensor in the reference configuration, is defined by*

$$\sigma^{ref}(x) = \frac{\partial W}{\partial F}(x, F),$$

*where $W$ is the stored energy function of the material. In a deformed configuration (characterised by the transformation $\varphi$), the first Piola-Kirchhoff tensor is defined by $\sigma(x) = \sigma^{ref}(\varphi^{-1}(x))\mathrm{Cof}\nabla\varphi^{-1}(x)$.*

**Theorem 1 (Cauchy axiom).** *In an equilibrium position the following relations hold*

$$\forall y \in \Gamma_1 \subset \partial\Omega,\ t(y, \boldsymbol{n}) = \sigma^{ref}(y)\boldsymbol{n},$$
$$\forall y \in \Omega,\ f(y) = -\mathrm{div}\,\sigma^{ref}(y),$$

*where $t(y, \boldsymbol{n})$ is the pressure applied to the material at the boundary point $y$ in the normal direction $\boldsymbol{n}$, $f(y)$ is the inner volumetric force applied at $y \in \Omega$.*

*Remark 1.* In our model, we aim at having $\theta_{T_i} \circ \varphi_i$ close to $\theta_R$, and therefore $\varphi_i$ is the deformation from $\theta_R$ to $\theta_{T_i}$. The deformation from $\theta_{T_i}$ to $\theta_R$ is $\varphi_i^{-1}$ and thanks to Ball's results [1], we have that

$$\int_\Omega W_{Op}(\nabla\varphi_i)\,dx = \int_\Omega \tilde{W}(\nabla\varphi_i, \mathrm{Cof}\nabla\varphi_i, \det\nabla\varphi_i)\,dx,$$
$$= \int_\Omega \tilde{W}\Big(\frac{\mathrm{Cof}\nabla(\varphi_i^{-1})(\varphi_i)^T}{\det\nabla(\varphi_i^{-1})(\varphi_i)}, \frac{\nabla(\varphi_i^{-1})(\varphi_i)^T}{\det\nabla(\varphi_i^{-1})(\varphi_i)}, \frac{1}{\det\nabla(\varphi_i^{-1})(\varphi_i)}\Big)\,dx,$$
$$= \int_\Omega \det\nabla(\varphi_i^{-1})\tilde{W}\Big(\frac{\mathrm{Cof}\nabla(\varphi_i^{-1})^T}{\det\nabla(\varphi_i^{-1})}, \frac{\nabla(\varphi_i^{-1})^T}{\det\nabla(\varphi_i^{-1})}, \frac{1}{\det\nabla(\varphi_i^{-1})}\Big)\,dx,$$
$$= \int_\Omega \tilde{W}_{Op}(\nabla(\varphi_i^{-1}))\,dx,$$

with $\tilde{W}(\psi, \phi, \delta) = \begin{cases} a_1\|\psi\|^4 + a_2\|\phi\|^4 + a_3(\delta-1)^2 + \frac{a_4}{\delta^{10}} - 3a_1 - 3a_2 - a_4 & \text{if } \delta > 0, \\ +\infty & \text{otherwise} \end{cases}$.

Therefore, the Cauchy-stress tensor applied to the reference configuration $\theta_{T_i}$ associated to the deformation $\varphi_i^{-1}$ to transform $\theta_{T_i}$ into $\theta_R$ is $\sigma_i^{ref}(x) = \frac{\partial \tilde{W}_{Op}}{\partial F}(\nabla\varphi_i^{-1}(x))$. The corresponding first Piola-Kirchhoff tensor applied to the deformed configuration $\theta_R$ is given by $\sigma_i = \frac{\partial \tilde{W}_{Op}}{\partial F}(\nabla\varphi_i^{-1}(\varphi_i))\mathrm{Cof}\nabla\varphi_i$. The inner volumetric forces applied to $\theta_R$ that are necessary to deform $\theta_{T_i}$ into $\theta_R$ are given by $-\mathrm{div}\,\sigma_i$.

Equipped with this material and the previous observation, we consider our mean configuration $\theta_R$ as a free material (not pre-stressed) and look for the displacements $v_i$ induced by a small increase of the impact of $\theta_{T_i}$ parameterised by $\delta$ and minimising the following energy for each $i = 1, \cdots, M$:

$$\inf_{v_i \in \mathcal{V}} \Big\{ \mathcal{F}_i^\delta(v_i) = \int_\Omega \tilde{W}_{Op}((\mathrm{I} + \delta\nabla v_i))\,dx - \delta^2\int_\Omega \mathrm{div}\,\sigma_i : v_i\,dx,$$
$$= \int_\Omega \tilde{W}_{Op}((\mathrm{I} + \delta\nabla v_i))\,dx + \delta^2\int_\Omega \sigma_i : \nabla v_i\,dx - \delta^2\int_{\partial\Omega} \sigma_i v_i \boldsymbol{n}_{\partial\Omega}\,ds,$$
$$= \int_\Omega \tilde{W}_{Op}((\mathrm{I} + \delta\nabla v_i))\,dx + \delta^2\int_\Omega \sigma_i : \nabla v_i\,dx \Big\},$$

with $\mathcal{V} := \{\psi \in W_0^{1,4}(\Omega, \mathbb{R}^3) \,|\, \text{Cof}(I + \delta\nabla\psi) \in L^4(\Omega, M_3(\mathbb{R})), \det(I + \delta\nabla\psi) \in L^2(\Omega), \frac{1}{\det(I+\delta\nabla\psi)} \in L^{10}(\Omega), \det(I + \delta\psi) > 0 \text{ a.e. on } \Omega\}$. The first term corresponds to the stored energy of the chosen Ogden material, while the second component represents the work of the applied forces with the following boundary conditions: $v_i = 0$ on $\partial\Omega$. However, the solution of this problem still lives in a nonlinear space and we thus propose to linearise the stored energy function $\tilde{W}_{Op}$ around the identity. Before linearising the functional using a Taylor development, we introduce the deformation Cauchy-Green tensor $b_i = (I + \delta v_i)^T(I + \delta v_i) = I + \delta(\nabla v_i^T + \nabla v_i) + \delta^2 \nabla v_i^T \nabla v_i = I + 2\delta\epsilon(v_i) + \delta^2 \nabla v_i^T \nabla v_i \approx I + 2\delta\epsilon(v_i)$. We first linearise the deformation tensor then the stored energy function and will see that we come back to the linearised elasticity problem.

$$\mathcal{F}_i^\delta(v_i) = \int_\Omega \tilde{W}_{Op}(I + \delta\nabla v_i)\,dx + \delta^2 \int_\Omega \nabla v_i : \sigma_i\,dx,$$

$$= \int_\Omega \bar{W}_{Op}(b_i) + \delta^2 \int_\Omega \nabla v_i : \sigma_i\,dx,$$

$$= \int_\Omega \bar{W}_{Op}(I + 2\delta\epsilon(v_i)) + \delta^2 \int_\Omega \nabla v_i : \sigma_i\,dx,$$

$$= \int_\Omega \bar{W}_{Op}(I)\,dx + 2\delta \int_\Omega \frac{\partial \bar{W}_{Op}}{\partial b}(I) : \epsilon(v_i)\,dx + 2\delta^2 \int_\Omega \frac{\partial^2 \bar{W}_{Op}}{\partial b^2}(I) : \epsilon(v_i) : \epsilon(v_i)\,dx$$

$$+ \delta^2 \int_\Omega \nabla v_i : \sigma_i\,dx,$$

with $\bar{W}_{Op}(b) = a_1 \frac{II_b^2}{\sqrt{III_b}^3} + a_2 \frac{I_b^2}{\sqrt{III_b}^3} + a_4 \sqrt{III_b}^{11} + \frac{a_3}{\sqrt{III_b}} - 2a_3 + a_3\sqrt{III_b} - (9a_1 + 9a_2 + a_4)\sqrt{III_b}$, $I_b = \text{Tr}(b) = \|I + \delta v_i\|_F^2$, $II_b = \frac{\text{Tr}(b)^2 - \text{Tr}(b^2)}{2} = \|\text{Cof}(I + \delta v_i)\|_F^2$ and $III_b = \det b = \det(I + \delta v_i)^2$. The detailed computations of the derivatives are given in 4. For our configuration to be in equilibrium, we need $\frac{\partial \bar{W}_{Op}}{\partial b}|_{(I_b, II_b, III_b)=(3,3,1)} = 0$, implying $a_4 = \frac{12a_2 + 6a_1}{5}$. We also set $\mu = 12a_1 + 12a_2 > 0$ and $\lambda = 4(32a_1 + 74a_2 + \frac{a_3}{2}) > 0$. Since $I \otimes I : \epsilon(v) : \epsilon(v) = Tr(\epsilon)^2$, $\mathbf{I} : \epsilon(v) : \epsilon(v) = \bar{\mathbf{I}} : \epsilon(v) : \epsilon(v) = Tr(\epsilon^2)$ by symmetry of $\epsilon(v)$, we finally obtain a linearised elasticity functional with $v_i \in H_0^1(\Omega, \mathbb{R}^3)$:

$$\mathcal{F}_i^\delta(v_i) = \int_\Omega \tilde{W}_{Op}(I + \delta v_i)\,dx + \delta^2 \int_\Omega \nabla v_i : \sigma_i\,dx,$$

$$= \int_\Omega \bar{W}_{Op}(I)\,dx + 2\delta \int_\Omega \frac{\partial \bar{W}_{Op}}{\partial b}(I) : \epsilon(v_i)\,dx + 2\delta^2 \int_\Omega \frac{\partial^2 \bar{W}_{Op}}{\partial b^2}(I) : \epsilon(v_i) : \epsilon(v_i)\,dx$$

$$+ \delta^2 \int_\Omega \nabla v_i : \sigma_i\,dx,$$

$$= 2\delta^2 \int_\Omega \mu Tr(\epsilon(v_i)^2) + \frac{\lambda}{2} Tr(\epsilon(v_i))^2\,dx + \delta^2 \int_\Omega \sigma_i : \nabla v_i\,dx.$$

We now consider the following linearised elasticity problem LEP

$$\inf_{v_i \in H_0^1(\Omega, \mathbb{R}^3)} \{F_i^\delta(v_i) = 2\delta^2 \int_\Omega \mu Tr(\epsilon(v_i)^2) + \frac{\lambda}{2} Tr(\epsilon(v_i))^2 \, dx + \delta^2 \int_\Omega \sigma_i : \nabla v_i \, dx\}.$$
(LEP)

By using the Lax-Milgram theorem, one can prove the following theorem.

**Theorem 2 (Existence and uniqueness of a solution to problem LEP).**
*There exists a unique solution of LEP in the linear space $H_0^1(\Omega, \mathbb{R}^2)$.*

*Proof.* In order to prove the existence and uniqueness of a minimizer in $H_0^1(\Omega)$, we recall the Lax-Milgram theorem :

**Theorem 3 (Lax-Milgram).** *Let $\mathcal{H}$ be a Hilbert space. Let $a(.,.)$ be a continuous bilinear form on $\mathcal{H} \times \mathcal{H}$ and coercive on $\mathcal{H}$. Let $L(.)$ be a linear continuous form on $\mathcal{H}$. Then there exists a unique element of $\mathcal{H}$, $u$, such that for all $v \in \mathcal{H}$, $a(u, v) = L(v)$.*
*Furthermore, if $a(.,.)$ is symmetric, then $u$ minimizes $J(v) = \frac{1}{2}a(v, v) - L(v)$, that is to say there exists a unique $u \in \mathcal{H}$ such that $J(u) = \min_{v \in \mathcal{H}} J(v)$.*

We now set $\mathcal{H} = H_0^1(\Omega, \mathbb{R}^3)$, $a(u,v) = \int_\Omega 4\delta^2(\mu\epsilon(u) : \epsilon(v) + \frac{\lambda}{2} Tr(\epsilon(u))Tr(\epsilon(v))) \, dx$ for all $(u,v) \in H_0^1(\Omega, \mathbb{R}^3) \times H_0^1(\Omega, \mathbb{R}^3)$, and $L(v) = -\delta^2 \int_\Omega \sigma_i : \nabla v_i \, dx$ for all $v \in H_0^1(\Omega, \mathbb{R}^3)$. $L$ is linear on $H_0^1(\Omega, \mathbb{R}^3)$ by linearity of the gradient operator and the scalar product in $M_3(\mathbb{R})$. It is also continuous on $H_0^1(\Omega, M_3(\mathbb{R}))$ by using Cauchy-Schwartz inequality and Poincaré's inequality, $\forall v \in H_0^1(\Omega, \mathbb{R}^3)$:

$$L(v) = \int_\Omega \sigma_i : \nabla v \, dx \leq \|\sigma_i\|_{L^2(\Omega, M_3(\mathbb{R}))} \|\nabla v\|_{L^2(\Omega, M_3(\mathbb{R}))} \leq \|\sigma_i\|_{L^2(\Omega, M_3(\mathbb{R}))} \|v\|_{H^1(\Omega, \mathbb{R}^3)},$$

with $\|\sigma_i\|_{L^2(\Omega, M_3(\mathbb{R}))} < +\infty$ since we have proved earlier that $\sigma_i \in L^2(\Omega, M_3(\mathbb{R}))$. $a(.,.)$ is bilinear on $H_0^1(\Omega, \mathbb{R}^3) \times H_0^1(\Omega, \mathbb{R}^3)$ by linearity of $\epsilon$ coming from the linearity of the gradient operator and the linearity of the scalar product in $M_3(\mathbb{R})$ and the trace operator. $a(.,.$ is also symmetric by symmetry of $\epsilon(u)$ and $\epsilon(v)$. $a(.,.)$ is also continuous on $H_0^1(\Omega, \mathbb{R}^3) \times H_0^1(\Omega, \mathbb{R}^3)$ since we have :

$$\forall (u,v) \in H_0^1(\Omega, \mathbb{R}^3) \times H_0^1(\Omega, \mathbb{R}^3), \, a(u,v) = \int_\Omega 4\delta^2(\mu\epsilon(u) : \epsilon(v) + \frac{\lambda}{2} Tr(\epsilon(u))Tr(\epsilon(v))) \, dx,$$
$$\leq 4\delta^2\mu\|\epsilon(u)\|_{L^2(\Omega, M_3(\mathbb{R}))}\|\epsilon(v)\|_{L^2(\Omega, M_3(\mathbb{R}))} + 2\delta^2\lambda\|Tr(\epsilon(u))\|_{L^2(\Omega)}\|Tr(\epsilon(v))\|_{L^2(\Omega)},$$
$$\leq 4\delta^2\mu\|\nabla u\|_{L^2(\Omega, M_3(\mathbb{R}))}\|\nabla v\|_{L^2(\Omega, M_3(\mathbb{R}))} + 2\delta^2\lambda\|\nabla u\|_{L^2(\Omega, M_3(\mathbb{R}))}\|\nabla v\|_{L^2(\Omega, M_3(\mathbb{R}))},$$
$$\leq (4\delta^2\mu + 2\delta^2\lambda)\|u\|_{H^1(\Omega)}\|v\|_{H^1(\Omega)}, \text{ Poincaré's inequality.}$$

We now need to prove that $a(.,.)$ is coercive. Let $v \in H_0^1(\Omega, \mathbb{R}^3)$, then we have

$$a(v,v) = \int_\Omega 4\delta^2(\mu\epsilon(v) : \epsilon(v) + \frac{\lambda}{2} Tr(\epsilon(v))^2 \, dx \geq 4\delta^2\mu\|\epsilon(v)\|_{L^2(\Omega, M_3(\mathbb{R}))}^2.$$

Since $4\delta^2\mu > 0$, we just need to prove the coercivity of $\|\epsilon(v)\|_{L^2(\Omega, M_3(\mathbb{R}))}$. To do so, we use a reductio ad absurdum reasoning as in [4] and assume there exists

a sequence $(v_n) \in H_0^1(\Omega, M_3(\mathbb{R}))$ such that $\|\epsilon(v_n)\|_{L^2(\Omega, M_3(\mathbb{R}))} \underset{n \to +\infty}{\longrightarrow} 0$ and $\|v_n\|_{H_0^1(\Omega, \mathbb{R}^3)} = 1$. As $(v_n)$ is uniformly bounded according to $n$ in $H^1(\Omega, \mathbb{R}^3)$, we can extract a subsequence still denoted $(v_n)$ such that $v_n \underset{n \to +\infty}{\rightharpoonup} v^*$ in $H^1(\Omega, \mathbb{R}^3)$ and by continuity of the trace operator, we get that $v^* \in H_0^1(\Omega)$. We also have that $\epsilon(v_n) \underset{n \to +\infty}{\rightharpoonup} \epsilon(v^*)$ in $L^2(\Omega, M_3(\mathbb{R}))$. However, we have assumed that $\epsilon(v_n) \underset{n \to +\infty}{\longrightarrow} 0$ in $L^2(\Omega, M_3(\mathbb{R}))$ and so by uniqueness of the weak limit, we get that $\epsilon(v^*) = 0$. Then by [4, Lemma IV.3.3], it comes that $v^*$ is an infinitesimal rigid displacement. But a rigid displacement can only be null on a straight line if $a \neq 0$, nowhere if $a \neq 0$, $b \neq 0$ or everywhere if $a = b = 0$. Since $v^*$ is null on $\partial \Omega$ with meas$(\partial \Omega) > 0$, then we are necessary in the last case and so $v^* = 0$. Then Rellich theorem allows us to deduce that $v_n \to 0$ in $L^2(\Omega, \mathbb{R}^3)$. Then by using Korn's inequality ([4, Theorem IV.3.2]), we get that $\nabla v_n \underset{n \to +\infty}{\longrightarrow} 0$ in $L^2(\Omega, M_3(\mathbb{R}))$ so that $v_n \underset{n \to +\infty}{\longrightarrow} 0$ in $H^1(\Omega, \mathbb{R}^3)$ which contradicts our second hypothesis. Therefore, $a(.,.)$ is coercive on $H_0^1(\Omega, \mathbb{R}^3)$. Lax Milgram's theorem applies and there exists a unique solution to our linearised elasticity problem LEP in $H_0^1(\Omega)$.

This problem can be solved in the finite element approximation framework which results in the resolution of a sparse high-dimensional linear system. A classical PCA is then performed on the obtained displacement fields using the $L^2(\Omega, \mathbb{R}^3)$ scalar product for the covariance operator. We now address the second method based on the Cauchy-stress tensors.

### 3.2   Second Approach : PCA on the Cauchy-stress Tensors

This approach relies on the following lemma.

**Lemma 1.** *The Cauchy-stress tensors $\sigma_i^{ref}$ and $\sigma_i$ are in $L^\infty(\Omega, M_3(\mathbb{R})) \subset L^2(\Omega, M_3(\mathbb{R}))$.*

*Proof.* Let us compute $\sigma_i^{ref}$ and $\sigma_i$ and make sure they belong to $L^\infty(\Omega, M_3(\mathbb{R}))$. We first recall that $\sigma_i^{ref} = \frac{\partial \tilde{W}_{Op}}{\partial F}(\nabla(\varphi_i^{-1})) = \frac{\partial}{\partial F}(\det F(a_1 \frac{\|\text{Cof} F\|_F^4}{\det F^4} + a_2 \frac{\|F\|_F^4}{\det F^4} + a_4 \det F^{10} - a_3(\frac{1}{\det F} - 1)^2 - 9a_1 - 9a_2 - a_4)(\nabla(\varphi_i^{-1}))$. We compute the following

terms :

$$F = \begin{pmatrix} F_{11} & F_{12} & F_{13} \\ F_{21} & F_{22} & F_{23} \\ F_{31} & F_{32} & F_{33} \end{pmatrix},$$

$$\text{Cof} F = \begin{pmatrix} F_{22}F_{33} - F_{32}F_{23} & F_{31}F_{23} - F_{21}F_{33} & F_{21}F_{32} - F_{31}F_{22} \\ F_{32}F_{13} - F_{12}F_{33} & F_{11}F_{33} - F_{13}F_{31} & F_{31}F_{12} - F_{11}F_{32} \\ F_{12}F_{23} - F_{22}F_{13} & F_{21}F_{13} - F_{11}F_{23} & F_{11}F_{22} - F_{21}F_{12} \end{pmatrix},$$

$$\det F = F_{11}(F_{22}F_{33} - F_{32}F_{23}) + F_{12}(F_{31}F_{23} - F_{21}F_{33}) + F_{13}(F_{21}F_{32} - F_{31}F_{22}),$$
$$= F_{21}(F_{32}F_{13} - F_{12}F_{33}) + F_{22}(F_{11}F_{33} - F_{13}F_{31}) + F_{23}(F_{31}F_{12} - F_{11}F_{32}),$$
$$= F_{31}(F_{12}F_{23} - F_{22}F_{13}) + F_{32}(F_{21}F_{13} - F_{11}F_{23}) + F_{33}(F_{11}F_{22} - F_{21}F_{12}),$$

$$\frac{\partial \|F\|_F^2}{\partial F} = 2F,$$

$$\left(\frac{\partial \|\text{Cof} F\|_F^2}{\partial F}\right)_1^T = 2 \begin{pmatrix} F_{33}\text{Cof} F_{22} - F_{32}\text{Cof} F_{23} - F_{23}\text{Cof} F_{32} + F_{22}\text{Cof} F_{33} \\ -F_{33}\text{Cof} F_{21} + F_{31}\text{Cof} F_{23} + F_{23}\text{Cof} F_{31} - F_{21}\text{Cof} F_{33} \\ F_{32}\text{Cof} F_{21} - F_{31}\text{Cof} F_{22} - F_{22}\text{Cof} F_{31} + F_{21}\text{Cof} F_{32} \end{pmatrix},$$

$$\left(\frac{\partial \|\text{Cof} F\|_F^2}{\partial F}\right)_2^T = 2 \begin{pmatrix} -F_{31}\text{Cof} F_{12} + F_{32}\text{Cof} F_{13} + F_{13}\text{Cof} F_{32} - F_{12}\text{Cof} F_{33} \\ F_{33}\text{Cof} F_{11} - F_{31}\text{Cof} F_{13} - F_{13}\text{Cof} F_{31} + F_{11}\text{Cof} F_{33} \\ -F_{32}\text{Cof} F_{11} + F_{31}\text{Cof} F_{12} + F_{12}\text{Cof} F_{31} - F_{11}\text{Cof} F_{32} \end{pmatrix},$$

$$\left(\frac{\partial \|\text{Cof} F\|_F^2}{\partial F}\right)_3^T = 2 \begin{pmatrix} F_{23}\text{Cof} F_{12} - F_{22}\text{Cof} F_{13} - F_{13}\text{Cof} F_{22} + F_{12}\text{Cof} F_{23} \\ -F_{23}\text{Cof} F_{11} + F_{21}\text{Cof} F_{13} + F_{13}\text{Cof} F_{21} - F_{11}\text{Cof} F_{23} \\ F_{22}\text{Cof} F_{11} - F_{21}\text{Cof} F_{12} - F_{12}\text{Cof} F_{21} + F_{11}\text{Cof} F_{22} \end{pmatrix},$$

$$\frac{\partial \det F}{\partial F} = \text{Cof} F,$$

$$\sigma_i^{ref} = 2a_1 \frac{\partial \|\text{Cof} F\|_F^2}{\partial F} \frac{\|\text{Cof} F\|_F^2}{\det F^3} - 3a_1 \frac{\|\text{Cof} F\|_F^4}{\det F^4} \frac{\partial \det F}{\partial F} - 9a_1 \frac{\partial \det F}{\partial F}$$
$$+ 2a_2 \frac{\partial \|F\|_F^2}{\partial F} \frac{\|F\|_F^2}{\det F^3} - 9a_2 \frac{\partial \det F}{\partial F} - 3a_2 \frac{\|F\|_F^4}{\det F^4} \frac{\partial \det F}{\partial F}$$
$$- \frac{a_3}{\det F^2} \frac{\partial \det F}{\partial F} + a_3 \frac{\partial \det F}{\partial F} + 11a_4 \det F^{10} \frac{\partial \det F}{\partial F} - a_4 \frac{\partial \det F}{\partial F},$$

$$= 2a_1 \frac{\|\text{Cof} F\|_F^2}{\det F^3} \frac{\partial \|\text{Cof} F\|_F^2}{\partial F} - 3a_1 \frac{\|\text{Cof} F\|_F^4}{\det F^4} \text{Cof} F - 9a_1 \text{Cof} F$$
$$+ 4a_2 \frac{\|F\|_F^2}{\det F^3} F - 9a_2 \text{Cof} F - 3a_2 \frac{\|F\|_F^4}{\det F^4} \text{Cof} F$$
$$- \frac{a_3}{\det F^2} \text{Cof} F + a_3 \text{Cof} F + 11a_4 \det F^{10} \text{Cof} F - a_4 \text{Cof} F.$$

Since, for all $i \in \{1, \cdots, M\}$, $\|\nabla \varphi_i\|_{L^\infty(\Omega, M_3(\mathbb{R}))} \leq \alpha$ and $\|(\nabla \varphi_i)^{-1}\|_{L^\infty(\Omega, M_3(\mathbb{R}))} \leq \beta$, we deduce that

$$\|\nabla(\varphi_i^{-1})\|_{L^\infty(\Omega, M_3(\mathbb{R}))} = \|(\nabla \varphi_i)^{-1}(\varphi_i^{-1})\|_{L^\infty(\Omega, M_3(\mathbb{R}))} = \|(\nabla \varphi_i)^{-1}\|_{L^\infty(\Omega, M_3(\mathbb{R}))} \leq \beta$$

and

$$\|\nabla \varphi_i\|_{L^\infty(\Omega, M_3(\mathbb{R}))} = \|(\nabla \varphi_i^{-1})^{-1}(\varphi_i)\|_{L^\infty(\Omega, M_3(\mathbb{R}))} = \|(\nabla \varphi_i^{-1})^{-1}\|_{L^\infty(\Omega, M_3(\mathbb{R}))}$$
$$= \|\frac{\mathrm{Cof} \nabla (\varphi_i)^{-1}}{\det \nabla (\varphi_i)^{-1}}\|_{L^\infty(\Omega, M_3(\mathbb{R}))} \leq \alpha$$

so that $\|\mathrm{Cof}\nabla(\varphi_i)^{-1}\|_{L^\infty(\Omega, M_3(\mathbb{R}))} \leq 18\beta^2$, $\|\det\nabla(\varphi_i)^{-1}\|_{L^\infty(\Omega)} \leq 6\beta^3$ and $\|\mathrm{Cof}\nabla\varphi_i\|_{L^\infty(\Omega, M_3(\mathbb{R}))}$
$\leq 18\alpha^2$, $\|\det\nabla\varphi_i\|_{L^\infty(\Omega)} = \|\frac{1}{\det\nabla(\varphi_i^{-1})(\varphi_i)}\|_{L^\infty(\Omega)} = \|\frac{1}{\det\nabla(\varphi_i)^{-1}}\|_{L^\infty(\Omega)} \leq 6\alpha^3$.
From these observations, we deduce that $\sigma_i^{ref} \in L^\infty(\Omega, M_3(\mathbb{R})) \subset L^2(\Omega, M_3(\mathbb{R}))$ and therefore $\sigma_i \in L^\infty(\Omega, M_3(\mathbb{R})) \subset L^2(\Omega, M_3(\mathbb{R}))$.

We thus propose to perform a Principal Component Analysis on the tensors $\sigma_i$ directly since they belong to $L^2(\Omega, M_3(\mathbb{R}))$ and following [5]. To do so, let $V = L^2(\Omega, M_3(\mathbb{R}))$ and let $\mathcal{V}$ be a random variable with values in $V$. The associated covariance operator is thus defined by

$$\mathcal{C}_\mathcal{V}(\sigma) = \mathbb{E}[\langle \sigma, \mathcal{V} - \mathbb{E}\mathcal{V}\rangle_V (\mathcal{V} - \mathbb{E}\mathcal{V})], \forall \sigma \in V.$$

Then, the corresponding empirical estimation of this operator is given by

$$\hat{\mathcal{C}}(\sigma) = \frac{1}{M}\sum_{i=1}^M \langle \sigma, \sigma_i - \bar{\sigma}\rangle_V (\sigma_i - \bar{v}), \sigma \in V, \ \bar{\sigma} = \frac{1}{M}\sum_{i=1}^M \sigma_i.$$

We are now looking for the eigenvalues $(\lambda_j)$ and eigenvectors $(\psi_j)$ associated to this operator, and then sort them in a decreasing order, i.e. $\lambda_1 \geq \lambda_2 \geq \cdots$, so that $(\psi_j)$ is the $j$-th principal component function. Since the number of observations is small with respect to the size of the space, an appropriate choice for the base of $V$ is given by the collection of observations $\sigma_i$.
We then aim to solve

$$\frac{1}{M}\sum_{i=1}^M (\sigma_i - \bar{V})\sum_{j=1}^M b_{ji}\langle(\sigma_i - \bar{\sigma}),(\sigma_j - \bar{\sigma})\rangle_V = \lambda_j \sum_{i=1}^M b_{ji}(\sigma_i - \bar{\sigma}),$$
$$\iff \frac{1}{M}W b_j = \lambda_j b_j,$$

where $W = (W_{ij})_{i,j} = (\langle \sigma_i - \bar{\sigma}, \sigma_j - \bar{\sigma}\rangle_V)_{i,j} \in M_M(\mathbb{R})$. It results in finding the eigenvalues and eigenvectors of the matrix $W$. Then the $j$-th PC function is given by $\psi_j = \sum_{i=1}^M b_{ji}(\sigma_i - \bar{\sigma})$.
In order to come back to the deformation space and to properly model the loaded configuration, we compose the sought deformation with the previous deformations. We then solve the following minimization problem with the last term ensuring that the sought displacements $v_j$ are Lipschitz continuous so that the

chain rule applies:

$$\inf \Bigg\{ \mathcal{F}(v_j) = \frac{1}{M} \sum_{k=1}^{M} \int_{\Omega} \tilde{W}_{Op}(\nabla((\text{Id} - \delta v_j)^{-1} \circ \varphi_k^{-1}))\, dx - \delta^2 \int_{\Omega} \psi_j : \nabla v_j \, dx$$
$$+ \mathbb{1}_{\{\|\cdot\|_{L^{\infty}(\Omega)} \leq \alpha\}}(\text{I} - \delta \nabla v_j),$$
$$= \frac{1}{M} \sum_{k=1}^{M} \int_{\Omega} \tilde{W}_{Op}(\nabla(\varphi_k^{-1}) - \delta \nabla v_i(\varphi_k^{-1}) \nabla \varphi_k^{-1})\, dx$$
$$- \delta^2 \int_{\Omega} \psi_j : \nabla v_j \, dx + \mathbb{1}_{\{\|\cdot\|_{L^{\infty}(\Omega)} \leq \alpha\}}(\text{I} - \delta \nabla v_j),$$
$$= \frac{1}{M} \sum_{k=1}^{M} \int_{\Omega} \det \nabla \varphi_k \tilde{W}_{Op}(\nabla(\varphi_k^{-1})(\varphi_k) - \delta \nabla v_i \nabla \varphi_k^{-1}(\varphi_k))\, dy$$
$$- \delta^2 \int_{\Omega} \psi_j : \nabla v_j \, dx + \mathbb{1}_{\{\|\cdot\|_{L^{\infty}(\Omega)} \leq \alpha\}}(\text{I} - \delta \nabla v_j),$$
$$= \frac{1}{M} \sum_{k=1}^{M} \int_{\Omega} \det \nabla \varphi_k \tilde{W}_{Op}((\nabla \varphi_k)^{-1} - \delta \nabla v_i (\nabla \varphi_k)^{-1})\, dy$$
$$- \delta^2 \int_{\Omega} \psi_j : \nabla v_j \, dx + \mathbb{1}_{\{\|\cdot\|_{L^{\infty}(\Omega)} \leq \alpha\}}(\text{I} - \delta \nabla v_j) \Bigg\}. \tag{P1}$$

$\tilde{W}_{Op}$ is continuous since $\lim_{\det \xi \to 0^+} \tilde{W}_{Op}(\xi) = +\infty$. Let us set $\hat{\mathcal{W}}_1 = \{\psi \in W_0^{1,\infty}(\Omega, \mathbb{R}^3),\, \det(\text{I} - \delta \nabla \psi) > 0$ a.e. on $\Omega, \|\text{I} - \delta \nabla \psi\|_{L^{\infty}(\Omega)} \leq \alpha,\, \forall k \in \{1, \cdots, M\},\, \frac{\|\text{Cof}((\nabla \varphi_k)^{-1} - \delta \nabla \psi (\nabla \varphi_k)^{-1})\|_F^4}{\det((\nabla \varphi_k)^{-1} - \delta \nabla \psi (\nabla \varphi_k)^{-1})^3} \in L^1(\Omega),\, \frac{\|(\text{I} - \delta \nabla \psi)(\nabla \varphi_k)^{-1}\|_F^4}{\det((\nabla \varphi_k)^{-1} - \delta \nabla \psi (\nabla \varphi_k)^{-1})^3} \in L^1(\Omega),\, \det((\nabla \varphi_k)^{-1} - \delta \nabla \psi (\varphi_k)^{-1}) \in L^{11}(\Omega)\}$.
The minimizer is searched for $v_j \in \hat{\mathcal{W}}$.

**Theorem 4 (Existence of minimizers).** *One can prove that this problem P1 admits at least one solution.*

*Proof.* We follow the classical steps of the direct method of the calculus of variations.

1. **Coercivity inequality** :
   Let us first derive a coercivity inequality :

$$\mathcal{F}(v) \geq \frac{1}{M} \sum_{k=1}^{M} \int_{\Omega} \frac{1}{6\beta^3} \Big( a_1 \frac{\|\text{Cof}((\nabla \varphi_k)^{-1} - \delta \nabla v (\nabla \varphi_k)^{-1})\|_F^4}{\det((\nabla \varphi_k)^{-1} - \delta \nabla v (\nabla \varphi_k)^{-1})^3}$$
$$+ a_2 \frac{\|(\nabla \varphi_k)^{-1} - \delta \nabla v (\nabla \varphi_k)^{-1}\|_F^4}{\det((\nabla \varphi_k)^{-1} - \delta \nabla v (\nabla \varphi_k)^{-1})^3} + a_4 \det(\nabla \varphi_k^{-1} - \delta \nabla v_i (\nabla \varphi_k)^{-1})^{11}$$
$$- (9a_1 + 9a_2 + a_4)36\alpha^6 + \mathbb{1}_{\{\|\cdot\|_{L^{\infty}(\Omega, M_3(\mathbb{R}))} \leq \alpha\}}(\text{I} - \delta \nabla v),$$

since $\Omega$ is bounded, $\det((\nabla \varphi_k)^{-1} - \delta \nabla v (\nabla \varphi_k)^{-1}) = \det((\text{I} - \delta \nabla v)(\nabla \varphi_k)^{-1}) = \det(\text{I} - \delta \nabla v) \det((\nabla \varphi_k)^{-1}) \leq 36\alpha^6$, and $\|\det \nabla \varphi_k\|_{L^{\infty}(\Omega)} \leq 6\alpha^3,\, \|\frac{1}{\det \nabla \varphi_k}\|_{L^{\infty}(\Omega)} \leq$

$6\beta^3$ so that for all $x \in \Omega$, $\frac{1}{6\beta^3} \le \det \nabla \varphi_k(x) \le 6\alpha^3$. Furthermore, by taking $v = 0$, then $\mathcal{F}(v) = \frac{1}{M} \int_\Omega \tilde{W}_{Op}(\nabla \varphi_k^{-1}) \, dx < +\infty$ from what precedes. Thus the functional is proper and coercive so that the infimum exists and is finite.

2. **Convergence of a minimizing sequence** :

    Let $(v^n)$ be a minimizing sequence so that $\lim\limits_{n \to +\infty} \mathcal{F}(v^n) = \inf \mathcal{F}(v) < +\infty$ from what precedes. So there exists $M \in \mathbb{N}$ such that $\forall n \in \mathbb{N}$, $n \ge M \Rightarrow \mathcal{F}(v^n) \le \inf \mathcal{F}(v) + 1 < +\infty$. From now on, we will consider $n \ge M$.

    From the previous coercivity inequality, one has:

    - $((I - \delta \nabla v^n)(\nabla \varphi_k)^{-1})$ is uniformly bounded according to $n$ in $L^4(\Omega, M_3(\mathbb{R}))$ for all $k \in \{1, \cdots, M\}$.
    - $(I - \delta \nabla v^n)$ and so $(\nabla v^n)$ is uniformly bounded according to $n$ in $L^\infty(\Omega, M_3(\mathbb{R}))$.
    - $(\mathrm{Cof}((I - \delta \nabla v^n)(\nabla \varphi_k)^{-1})) = (\mathrm{Cof}(I - \delta \nabla v^n) \mathrm{Cof}((\nabla \varphi_k)^{-1})) = (\mathrm{Cof}(I - \delta \nabla v^n) \, (\mathrm{Cof} \nabla \varphi_k)^{-1})$ is uniformly bounded according to $n$ in $L^4(\Omega, M_3(\mathbb{R}))$ for all $k \in \{1, \cdots, M\}$.
    - $(\det((I - \delta \nabla v^n)(\nabla \varphi_k)^{-1})) = (\frac{\det(I - \delta \nabla v^n)}{\det \nabla \varphi_k})$ is uniformly bounded according to $n$ in $L^{11}(\Omega)$.

    There exists a subsequence still denoted $((I - \delta \nabla v^n)(\nabla \varphi_k)^{-1})$ such that

    $$(I - \delta \nabla v^n)(\nabla \varphi_k)^{-1} \underset{n \to +\infty}{\rightharpoonup} \alpha_k \text{ in } L^4(\Omega, M_3(\mathbb{R})), \forall k \in \{1, \cdots, M\},$$

    $$I - \delta \nabla v^n \underset{n \to +\infty}{\rightharpoonup} \alpha_k \nabla \varphi_k \text{ in } L^4(\Omega, M_3(\mathbb{R})), \forall k \in \{1, \cdots, M\},$$

    $$\nabla v^n \underset{n \to +\infty}{\rightharpoonup} \bar{V} = \frac{1}{\delta}(I - \alpha_k \nabla \varphi_k), \forall k \in \{1, \cdots, M\}, \text{ in } L^4(\Omega, M_3(\mathbb{R})).$$

    We thus deduce that $(\nabla v^n)$ is uniformly bounded according to $n$ in $L^4(\Omega, M_3(\mathbb{R}))$ and by Poincaré's inequality, we deduce that $(v^n)$ is uniformly bounded according to $n$ in $W^{1,4}(\Omega, \mathbb{R}^3)$. We can thus extract a subsequence still denoted $(v^n)$ such that

    $$v^n \underset{n \to +\infty}{\rightharpoonup} \bar{v} \text{ in } W^{1,4}(\Omega, \mathbb{R}^3).$$

    Furthermore, we can extract a subsequence still denoted $(\nabla v^n)$ such that :

    $$\nabla v^n \underset{n \to +\infty}{\overset{*}{\rightharpoonup}} \nabla \bar{v} \text{ in } L^\infty(\Omega, M_3(\mathbb{R})),$$

    by uniqueness of the weak limit in $L^4(\Omega, M_3(\mathbb{R}))$ and the continuous embedding of $L^\infty(\Omega, M_3(\mathbb{R}))$ into $L^4(\Omega, M_3(\mathbb{R}))$. Besides, using Poincaré-Wirtinger's inequality, there exists a constant $C$ depending only on $\Omega$

such that

$$\|v^n\|_{L^\infty(\Omega,\mathbb{R}^3)} - \frac{\int_\Omega v^n\,dx}{\operatorname{meas}(\Omega)} \leq \|v^n - \frac{\int_\Omega v^n\,dx}{\operatorname{meas}(\Omega)}\|_{L^\infty(\Omega,\mathbb{R}^3)} \leq C\|\nabla v^n\|_{L^\infty(\Omega,M_3(\mathbb{R}))} \leq C\alpha,$$

$$\|v^n\|_{L^\infty(\Omega,\mathbb{R}^3)} \leq C\alpha + \frac{\int_\Omega v^n\,dx}{\operatorname{meas}(\Omega)} \leq C\alpha + \frac{\|v^n\|_{L^1(\Omega,\mathbb{R}^3)}}{\operatorname{meas}(\Omega)} \leq C\alpha + \frac{\|v^n\|_{L^4(\Omega,\mathbb{R}^3)}}{\operatorname{meas}(\Omega)},$$

$$\leq C\alpha + \frac{C'}{\operatorname{meas}(\Omega)}\|\nabla v^n\|_{L^4(\Omega,M_3(\mathbb{R}))},$$

with $C'$ a constant depending only on $\Omega$ coming from Poincaré's inequality. Since, $(\nabla v^n)$ is uniformly bounded in $L^4(\Omega, M_3(\mathbb{R}))$, we deduce that $(v^n)$ is uniformly bounded according to $n$ in $W^{1,\infty}(\Omega, \mathbb{R}^3)$ and there exists a subsequence still denoted $(v^n)$ such that

$$v^n \underset{n\to+\infty}{\overset{*}{\rightharpoonup}} \bar{v} \text{ in } W^{1,\infty}(\Omega, \mathbb{R}^3),$$

by uniqueness of the weak limit in $W^{1,4}(\Omega, \mathbb{R}^3)$ and the continuous embedding of $W^{1,\infty}(\Omega, \mathbb{R}^3)$ into $W^{1,4}(\Omega, \mathbb{R}^3)$. By continuity of the trace operator, we get $\bar{v} \in W_0^{1,\infty}(\Omega, \mathbb{R}^3)$.

We can also extract subsequences of $(\operatorname{Cof}(I - \delta\nabla v^n)(\operatorname{Cof}\nabla\varphi_k)^{-1})$ and $(\frac{\det(I-\delta\nabla v^n)}{\det\nabla\varphi_k})$ still denoted $(\operatorname{Cof}(I-\delta\nabla v^n)(\operatorname{Cof}\nabla\varphi_k)^{-1})$ and $(\frac{\det(I-\delta\nabla v^n)}{\det\nabla\varphi_k})$ for all $k \in \{1,\cdots,M\}$, such that

$$\operatorname{Cof}(I - \delta\nabla v^n)(\operatorname{Cof}\nabla\varphi_k)^{-1} \underset{n\to+\infty}{\rightharpoonup} H_k \text{ in } L^4(\Omega, M_3(\mathbb{R})),$$

$$\operatorname{Cof}(I - \delta\nabla v^n) \underset{n\to+\infty}{\rightharpoonup} H_k\operatorname{Cof}\nabla\varphi_k \text{ in } L^4(\Omega, \mathbb{R})),$$

$$\frac{\det(I - \delta\nabla v^n)}{\det\nabla\varphi_k} \underset{n\to+\infty}{\rightharpoonup} \delta_k \text{ in } L^{11}(\Omega),$$

$$\det(I - \delta\nabla v^n) \underset{n\to+\infty}{\rightharpoonup} \delta_k\det\nabla\varphi_k \text{ in } L^{11}(\Omega).$$

By uniqueness of the weak limit, we have that $H = H_k\operatorname{Cof}\nabla\varphi_k$ and $\delta = \delta_k\det\nabla\varphi_k$ for all $k \in \{1,\cdots,M\}$. By [4, Theorem VI.3.3] and $\operatorname{Id} - \delta v^n \underset{n\to+\infty}{\rightharpoonup} \operatorname{Id} - \delta\bar{v}$ in $W^{1,4}(\Omega, \mathbb{R}^3)$, we have that $H = \operatorname{Cof}(I - \delta\nabla\bar{v})$ and $\delta = \det(I - \delta\nabla\bar{v})$.

3. **Lower semi-continuity** : $\tilde{W}_{Op}$ is convex and continuous. Indeed, let $f(x,y) = \frac{x^4}{y^3}$ for all $(x,y) \in (\mathbb{R}^+)^2$. Then the Hessian matrix of $f$ is given by $H(x,y) = \frac{12x^2}{y^3}\begin{pmatrix} 1 & -\frac{x}{y} \\ -\frac{x}{y} & \frac{x^2}{y^2} \end{pmatrix}$ and is positive semi-definite : $\forall (a,b) \in \mathbb{R}^2$, $(a,b)H(x,y)(a,b)^T = \frac{12x^2}{y^5}(bx - ay)^2 \geq 0$, for all $(x,y) \in (\mathbb{R}^+)^2$. Besides, $\|.\|_F$ is also convex and we have $\forall (A,B) \in M_3(\mathbb{R})^2, \forall (c,d) \in (\mathbb{R}^+)^2, \forall \lambda \in (0,1), \frac{(\|\lambda A+(1-\lambda)B\|_F^4)}{(\lambda c+(1-\lambda)d)^3} \leq \frac{(\lambda\|A\|_F+(1-\lambda)\|B\|_F)^4}{(\lambda c+(1-\lambda)d)^3} = f(\lambda\|A\|_F + (1-\lambda)\|B\|_F, \lambda c+(1-\lambda)d) \leq \lambda f(\|A\|_F, c) + (1-\lambda)f(\|B\|_F, d) = \lambda\frac{\|A\|_F^4}{c^3} +$

$(1-\lambda)\frac{\|B\|_F^4}{d^3}$. If $\psi_n \underset{n\to+\infty}{\longrightarrow} \bar{\psi}$ in $W^{1,4}(\Omega,\mathbb{R}^3)$ then $\nabla\psi_n \underset{n\to+\infty}{\longrightarrow} \nabla\bar{\psi}$ in $L^4(\Omega,M_3(\mathbb{R}))$ and we can extract a subsequence still denoted $(\nabla\psi_n)$ such that $\nabla\psi_n \underset{n\to+\infty}{\longrightarrow} \nabla\bar{\psi}$ almost everywhere in $\Omega$. If $\alpha_n \underset{n\to+\infty}{\longrightarrow} \bar{\alpha}$ in $L^4(\Omega,M_3(\mathbb{R}))$ then we can extract a subsequence still denoted $(\alpha_n)$ such that $\alpha_n \underset{n\to+\infty}{\longrightarrow} \bar{\alpha}$ almost everywhere in $\Omega$. If $\delta_n \underset{n\to+\infty}{\longrightarrow} \bar{\delta}$ in $L^{11}(\Omega)$, then there exists a subsequence still denoted $(\delta_n)$ such that $\delta_n \underset{n\to+\infty}{\longrightarrow} \bar{\delta}$ almost everywhere in $\Omega$. Then by continuity of $\tilde{W}_{Op}$, we get

$$\tilde{W}_{Op}(\nabla\psi_n,\alpha_n,\delta_n) \underset{n\to+\infty}{\longrightarrow} \tilde{W}_{Op}(\nabla\bar{\psi},\bar{\alpha},\bar{\delta}) \text{ almost everywhere in } \Omega.$$

Then, by applying Fatou's lemma, we have that

$$\liminf_{n\to+\infty} \int_\Omega \det\nabla\varphi_k \tilde{W}_{Op}(\nabla\psi_n,\alpha_n,\delta_n)\,dx \geq \int_\Omega \det\nabla\varphi_k \tilde{W}_{Op}(\nabla\bar{\psi},\bar{\alpha},\bar{\delta})\,dx.$$

Since $\tilde{W}_{Op}$ is convex, so is $\int_\Omega \det\nabla\varphi_k \tilde{W}_{Op}(\xi,\alpha,\delta)\,dx$, and we can apply [2, Corollaire III.8] to get that $\int_\Omega \det\nabla\varphi_k \tilde{W}_{Op}(\xi,\alpha,\delta)\,dx$ is lower semicontinuous in $L^4(\Omega,M_3(\mathbb{R})) \times L^4(\Omega,M_3(\mathbb{R})) \times L^{11}(\Omega)$. We deduce that

$$+\infty > \liminf_{n\to+\infty} \int_\Omega \det\nabla\varphi_k \tilde{W}_{Op}((\mathrm{I}-\delta\nabla v^n)(\nabla\varphi_k)^{-1}, \mathrm{Cof}(\mathrm{I}-\delta\nabla v^n)(\mathrm{Cof}\nabla\varphi_k)^{-1}),$$
$$\frac{\det(\mathrm{I}-\delta\nabla v^n)}{\det\nabla\varphi_k})\,dx$$
$$\geq \int_\Omega \det\nabla\varphi_k(x)\tilde{W}_{Op}((\mathrm{I}-\delta\nabla\bar{v})(\nabla\varphi_k)^{-1}, \mathrm{Cof}(\mathrm{I}-\delta\nabla\bar{v})(\mathrm{Cof}\nabla\varphi_k)^{-1},$$
$$\frac{\det(\mathrm{I}-\delta\nabla\bar{v})}{\det\nabla\varphi_k})\,dx.$$

By the weak-$*$ lower semi-continuity of the $\|.\|_{L^\infty(\Omega,M_3(\mathbb{R}))}$, we get that

$$\|\mathrm{I}-\delta\nabla\bar{v}\|_{L^\infty(\Omega,M_3(\mathbb{R}))} \leq \liminf_{n\to+\infty} \|\mathrm{I}-\nabla v^n\|_{L^\infty(\Omega,M_3(\mathbb{R}))} \leq \alpha,$$

so that

$$\mathbb{1}_{\{\|.\|_{L^\infty(\Omega,M_3(\mathbb{R}))}\leq\alpha\}}(\mathrm{I}-\nabla\bar{v}) = 0 \leq \liminf_{n\to+\infty} \mathbb{1}_{\{\|.\|_{L^\infty(\Omega,M_3(\mathbb{R}))}\leq\alpha\}}(\mathrm{I}-\delta\nabla v^n).$$

Since $\psi_j \in L^2(\Omega,M_3(\mathbb{R}))$ and $\nabla v^n \underset{n\to+\infty}{\rightharpoonup} \nabla\bar{v}$ in $L^4(\Omega,M_3(\mathbb{R}))$ and so in $L^2(\Omega,M_3(\mathbb{R}))$, we have that $-\delta^2 \int_\Omega \psi_j : \nabla\bar{v}\,dx = \lim_{n\to+\infty} -\delta^2 \int_\Omega \psi_j : \nabla v^n\,dx$.

By combining all the results, we get that

$$\liminf_{n\to+\infty} \mathcal{F}(v^n) \geq \mathcal{F}(\bar{v}).$$

Also $\bar{v} \in W_0^{1,\infty}(\Omega, \mathbb{R}^3)$, $\frac{\|\text{Cof}((I-\delta\nabla\bar{v})(\nabla\varphi_k)^{-1})\|_F^4}{\det((I-\delta\nabla\bar{v})(\nabla\varphi_k)^{-1})^3} \in L^1(\Omega)$, $\frac{\|(I-\delta\nabla\bar{v})(\nabla\varphi_k)^{-1}\|_F^4}{\det((I-\delta\nabla\bar{v})(\nabla\varphi_k)^{-1})^3} \in L^1(\Omega)$ by finiteness of $\mathcal{F}(\bar{v})$ inducing the finiteness of $\int_\Omega \det\nabla\varphi_k \tilde{W}_{Op}((I-\delta\nabla\bar{v})(\nabla\varphi_k)^{-1}, \text{Cof}((I-\delta\nabla\bar{v})(\nabla\varphi_k)^{-1}), \det((I-\delta\nabla\bar{v})(\nabla\varphi_k)^{-1}))\,dx$ for all $k \in \{1, \cdots, M\}$. Since $\tilde{W}_{Op}((I-\delta\nabla\bar{v})(\nabla\varphi_k)^{-1}, \text{Cof}((I-\delta\nabla\bar{v})(\nabla\varphi_k)^{-1}), \det((I-\delta\nabla\bar{v})(\nabla\varphi_k)^{-1})) = +\infty$ when $\det((I-\delta\nabla\bar{v})(\nabla\varphi_k)^{-1}) = \frac{\det(I-\delta\nabla\bar{v})(x)}{\det\nabla\varphi_k(x)} \leq 0 \iff \det(I-\delta\nabla\bar{v}) \leq 0$ since $\det\nabla\varphi_k(x) > 0$ on $\Omega$, the set on which it happens must be of null measure otherwise we would have $\mathcal{F}(\bar{v}) = +\infty$. So $\det(I-\delta\nabla\bar{v}) > 0$ almost everywhere in $\Omega$ and there exists a minimizer to our initial problem $\bar{v} \in \hat{\mathcal{W}}_1$.

This problem is hard to solve in practice and we follow the same strategy as previously by adding additional variables and using an alternative optimisation scheme.

## 4 Computation of the derivatives

$$\frac{\partial \bar{W}_{Op}}{\partial I_b} = \frac{2a_2 I_b}{\sqrt{III_b}^3} = 6a_2|_{(I_b, II_b, III_b)=(3,3,1)},$$

$$\frac{\partial \bar{W}_{Op}}{\partial II_b} = \frac{2a_1 I_b}{\sqrt{III_b}^3} = 6a_1|_{(I_b, II_b, III_b)=(3,3,1)},$$

$$\frac{\partial \bar{W}_{Op}}{\partial III_b} = \frac{-3a_1 II_b^2}{2\sqrt{III_b}^5} - \frac{3a_2 I_b^2}{2\sqrt{III_b}^5} + \frac{11}{2}a_4\sqrt{III_b}^9 - \frac{a_3}{2III_b\sqrt{III_b}} + \frac{a_3}{2\sqrt{III_b}^3}$$
$$- \frac{9a_1 + 9a_2 + a_4}{2\sqrt{III_b}} = -18a_1 - 18a_2 + 5a_4|_{(I_b, II_b, III_b)=(3,3,1)},$$

$$\frac{\partial^2 \bar{W}_{Op}}{\partial I_b^2} = \frac{2a_2}{\sqrt{III_b}^3} = 2a_2|_{(I_b, II_b, III_b)=(3,3,1)},$$

$$\frac{\partial^2 \bar{W}_{Op}}{\partial II_b \partial I_b} = 0,$$

$$\frac{\partial^2 \bar{W}_{Op}}{\partial III_b \partial I_b} = -\frac{-3a_2 I_b}{III_b^2 \sqrt{III_b}} = -9a_2|_{(I_b, II_b, III_b)=(3,3,1)},$$

$$\frac{\partial^2 \bar{W}_{Op}}{\partial I_b \partial II_b} = 0,$$

$$\frac{\partial^2 \bar{W}_{Op}}{\partial II_b^2} = \frac{2a_1}{\sqrt{III_b}^3} = 2a_1|_{(I_b, II_b, III_b)=(3,3,1)},$$

$$\frac{\partial^2 \bar{W}_{Op}}{\partial III_b \partial II_b} = -\frac{-3a_1 II_b}{III_b^2 \sqrt{III_b}} = -9a_1|_{(I_b, II_b, III_b)=(3,3,1)},$$

$$\frac{\partial^2 \bar{W}_{Op}}{\partial I_b \partial III_b} = \frac{-3a_2 I_b}{III_b^2 \sqrt{III_b}} = -9a_2|_{(I_b, II_b, III_b)=(3,3,1)},$$

$$\frac{\partial^2 \bar{W}_{Op}}{\partial II_b \partial III_b} = \frac{-3a_1 II_b}{III_b^2 \sqrt{III_b}} = -9a_1|_{(I_b, II_b, III_b)=(3,3,1)},$$

$$\frac{\partial^2 \bar{W}_{Op}}{\partial III_b^2} = \frac{15a_1 II_b^2}{4\sqrt{III_b}^7} + \frac{15a_2 I_b^2}{4\sqrt{III_b}^7} + \frac{99a_4}{4}\sqrt{III_b}^7$$

$$+ \frac{3a_3}{4\sqrt{III_b}^3} - \frac{a_3}{4III_b\sqrt{III_b}} + \frac{9a_1 + 9a_2 + a_4}{4\sqrt{III_b}^3},$$

$$= 36a_1 + 36a_+ \frac{a_3}{2} + 25a_4|_{(I_b,II_b,III_b)=(3,3,1)},$$

$$\frac{\partial I_b}{\partial b} = I,$$

$$\frac{\partial II_b}{\partial b} = I_b I - b = 2I|_{(I_b,II_b,III_b)=(3,3,1)},$$

$$\frac{\partial III_b}{\partial b} = III_b b^{-1} = I|_{(I_b,II_b,III_b)=(3,3,1)},$$

$$\frac{\partial I}{\partial b} = 0,$$

$$\frac{\partial b}{\partial b} = \mathbf{I}, \text{Identity tensor of order 4 such that } \mathbf{I} : A : A = Tr(A^T A),$$

$$\frac{\partial b^{-1}}{\partial b} = \frac{\mathbf{I} + \bar{\mathbf{I}}}{2}|_{(I_b,II_b,III_b)=(3,3,1)}, \text{ with } \bar{\mathbf{I}} : A : A = Tr(A^2),$$

$$\frac{\partial \bar{W}_{Op}}{\partial b} = \frac{\partial \bar{W}_{Op}}{\partial I_b}\frac{\partial I_b}{\partial b} + \frac{\partial \bar{W}_{Op}}{\partial II_b}\frac{\partial II_b}{\partial b} + \frac{\partial \bar{W}_{Op}}{\partial III_b}\frac{\partial III_b}{\partial b} = (-12a_1 - 6a_1 + 5a_4)I|_{(I_b,II_b,III_b)=(3,3,1)},$$

$$\frac{\partial^2 \bar{W}_{Op}}{\partial b^2} = \frac{\partial \bar{W}_{Op}}{\partial I_b}\frac{\partial I}{\partial b} + I \otimes (\frac{\partial^2 \bar{W}_{Op}}{\partial I_b^2}\frac{\partial I_b}{\partial b} + \frac{\partial^2 \bar{W}_{Op}}{\partial II_b \partial I_b}\frac{\partial II_b}{\partial b} + \frac{\partial^2 \bar{W}_{Op}}{\partial III_b \partial I_b}\frac{\partial III_b}{\partial b})$$

$$+ \frac{\partial \bar{W}_{Op}}{\partial II_b}\frac{\partial (I_b I - b)}{\partial b} + (I_b I - b) \otimes (\frac{\partial^2 \bar{W}_{Op}}{\partial I_b \partial II_b}\frac{\partial I_b}{\partial b} + \frac{\partial^2 \bar{W}_{Op}}{\partial II_b^2}\frac{\partial II_b}{\partial b} + \frac{\partial^2 \bar{W}_{Op}}{\partial III_b \partial II_b}\frac{\partial III_b}{\partial b})$$

$$+ \frac{\partial \bar{W}_{Op}}{\partial III_b}\frac{\partial (III_b b^{-1})}{\partial b} + (III_b b^{-1}) \otimes (\frac{\partial^2 \bar{W}_{Op}}{\partial I_b \partial III_b}\frac{\partial I_b}{\partial b} + \frac{\partial^2 \bar{W}_{Op}}{\partial II_b \partial III_b}\frac{\partial II_b}{\partial b} + \frac{\partial^2 \bar{W}_{Op}}{\partial III_b^2}\frac{\partial III_b}{\partial b}),$$

$$= I \otimes (\frac{\partial^2 \bar{W}_{Op}}{\partial I_b^2}I + \frac{\partial^2 \bar{W}_{Op}}{\partial II_b \partial I_b}(I_b I - b) + \frac{\partial^2 \bar{W}_{Op}}{\partial III_b \partial I_b}III_b b^{-1})$$

$$+ \frac{\partial \bar{W}_{Op}}{\partial II_b}(I \otimes I - \mathbf{I}) + (I_b I - b) \otimes (\frac{\partial^2 \bar{W}_{Op}}{\partial I_b \partial II_b}I + \frac{\partial^2 \bar{W}_{Op}}{\partial II_b^2}(I_b I - b) + \frac{\partial^2 \bar{W}_{Op}}{\partial III_b \partial II_b}III_b b^{-1})$$

$$+ \frac{\partial \bar{W}_{Op}}{\partial III_b}(b^{-1} \otimes III_b b^{-1} + III_b \frac{\partial b^{-1}}{\partial b})$$

$$+ III_b b^{-1} \otimes (\frac{\partial^2 \bar{W}_{Op}}{\partial I_b \partial III_b}I + \frac{\partial^2 \bar{W}_{Op}}{\partial II_b \partial III_b}(I_b I - b) + \frac{\partial^2 \bar{W}_{Op}}{\partial III_b^2}III_b b^{-1}),$$

$$= \frac{\partial \bar{W}_{Op}}{\partial II_b}(I \otimes I - \mathbf{I}) + \frac{\partial \bar{W}_{Op}}{\partial III_b}III_b(b^{-1} \otimes b^1 + \frac{\partial b^{-1}}{\partial b})$$

$$+ I \otimes (\frac{\partial^2 \bar{W}_{Op}}{\partial I_b^2}I + \frac{\partial^2 \bar{W}_{Op}}{\partial II_b \partial I_b}(I_b I - b) + \frac{\partial^2 \bar{W}_{Op}}{\partial III_b \partial I_b}III_b b^{-1}$$

$$+ (I_b I - b) \otimes (\frac{\partial^2 \bar{W}_{Op}}{\partial I_b \partial II_b}I + \frac{\partial^2 \bar{W}_{Op}}{\partial II_b^2}(I_b I - b) + \frac{\partial^2 \bar{W}_{Op}}{\partial III_b \partial II_b}III_b b^{-1})$$

$$+ III_b b^{-1} \otimes (\frac{\partial^2 \bar{W}_{Op}}{\partial I_b \partial III_b}I + \frac{\partial^2 \bar{W}_{Op}}{\partial II_b \partial III_b}(I_b I - b) + \frac{\partial^2 \bar{W}_{Op}}{\partial III_b^2}III_b b^{-1}),$$

$$= -(\frac{\partial \bar{W}_{Op}}{\partial II_b}\mathbf{I} + \frac{1}{2}\frac{\partial \bar{W}_{Op}}{\partial III_b}(\mathbf{I} + \bar{\mathbf{I}})) + I \otimes I(\frac{\partial \bar{W}_{Op}}{\partial II_b} + \frac{\partial \bar{W}_{Op}}{\partial III_b} + \frac{\partial^2 \bar{W}_{Op}}{\partial I_b^2} + 4\frac{\partial^2 \bar{W}_{Op}}{\partial I_b \partial II_b}$$

$$+ 2\frac{\partial^2 \bar{W}_{Op}}{\partial I_b \partial III_b} + 4\frac{\partial^2 \bar{W}_{Op}}{\partial II_b^2} + 4\frac{\partial^2 \bar{W}_{Op}}{\partial III_b \partial II_b} + \frac{\partial^2 \bar{W}_{Op}}{\partial III_b^2})|_{(I_b,II_b,III_b)=(3,3,1)},$$

$$= (15a_1 + 9a_2 - \frac{5}{2}a_4)\mathbf{I} + (9a_1 + 9a_2 - \frac{5a_4}{2})\bar{\mathbf{I}} + I \otimes I(-4a_1 + 20a_2 + 30a_4 + \frac{a_3}{2}).$$

## 5  Proof of Theorem 3

*Proof.* Let us denote by $a$ the mapping defined by :

$$a : \begin{vmatrix} H^3(\Omega,\mathbb{R}^2) \times H^3(\Omega,\mathbb{R}^2) \to \mathbb{R} \\ (u,v) \mapsto \langle \rho_0(u), \rho_0(v) \rangle_{\mathbb{R}^2, N_0} + \frac{\gamma}{2} \langle \rho_1(\nabla u + \nabla u^T), \rho_1(\nabla v + \nabla v^T) \rangle_{M_2(\mathbb{R}), N_1} \\ +\epsilon(u,v)_{3,\Omega,\mathbb{R}^2} \end{vmatrix},$$

with $(.,.)_{3,\Omega,\mathbb{R}^2}$ denoting the semi-norm in $H^3(\Omega,\mathbb{R}^2)$, and by $L$ the mapping defined by :

$$L : \begin{vmatrix} H^3(\Omega,\mathbb{R}^2) \to \mathbb{R} \\ v \mapsto \langle \rho_0(f_k), \rho_0(v) \rangle_{\mathbb{R}^2, N_0} + \frac{\gamma}{2} \langle \rho_1(\nabla f_k + \nabla f_k^T + \nabla f_k^T \nabla f_k), \rho_1(\nabla v + \nabla v^T) \rangle_{M_2(\mathbb{R}), N_1} \end{vmatrix}.$$

Let us notice that $\forall v \in H^3(\Omega,\mathbb{R}^2)$,

$$\begin{aligned}\mathcal{F}_{\epsilon,k}(v) &= \langle \rho_0(v) - \rho_0(f_k) \rangle^2_{\mathbb{R}^2, N_0} + \frac{\gamma}{2} \langle \rho_1(\nabla v + \nabla v^T) - \rho_1(\nabla f_k + \nabla f_k^T + \nabla f_k^T \nabla f_k) \rangle^2_{M_2(\mathbb{R}), N_1} \\ &\quad + \epsilon |v|^2_{3,\Omega,\mathbb{R}^2}, \\ &= \langle \rho_0(v) \rangle^2_{\mathbb{R}^2, N_0} + \langle \rho_0(f_k) \rangle^2_{\mathbb{R}^2, N_0} - 2\langle \rho_0(v), \rho_0(f_k) \rangle_{\mathbb{R}^2, N_0} + \frac{\gamma}{2} \langle \rho_1(\nabla v + \nabla v^T) \rangle^2_{M_2(\mathbb{R}), N_1} \\ &\quad + \frac{\gamma}{2} \langle \rho_1(\nabla f_k + \nabla f_k^T + \nabla f_k^T \nabla f_k) \rangle^2_{M_2(\mathbb{R}), N_1} \\ &\quad - \gamma \langle \rho_1(\nabla v + \nabla v^T), \rho_1(\nabla f_k + \nabla f_k^T + \nabla f_k^T \nabla f_k) \rangle_{M_2(\mathbb{R}), N_1} + \epsilon |v|^2_{3,\Omega,\mathbb{R}^2}, \\ &= a(v,v) - 2L(v) + \langle \rho_0(f_k) \rangle^2_{\mathbb{R}^2, N_0} + \frac{\gamma}{2} \langle \rho_1(\nabla f_k + \nabla f_k^T + \nabla f_k^T \nabla f_k) \rangle^2_{M_2(\mathbb{R}), N_1}.\end{aligned}$$

Also problem (2) is equivalent to :

$$\begin{cases} \text{Search for } u_{\epsilon,k} \in H^3(\Omega,\mathbb{R}^2) \text{ such that:} \\ \forall v \in H^3(\Omega,\mathbb{R}^2), \forall \mu \in \mathbb{R}, \mathcal{F}_{\epsilon,k}(u_{\epsilon,k}) \leq \mathcal{F}_{\epsilon,k}(u_{\epsilon,k} + \mu v) \end{cases}.$$

Using the bilinearity and symmetry of the mapping $a$, and the linearity of the mapping $L$, one has :

$$\begin{aligned}\mathcal{F}_{\epsilon,k}(u_{\epsilon,k} + \mu v) &= a(u_{\epsilon,k}, u_{\epsilon,k}) + 2\mu a(u_{\epsilon,k}, v) + \mu^2 a(v,v) - 2L(u_{\epsilon,k}) - 2\mu L(v) + \langle \rho_0(f_k) \rangle^2_{\mathbb{R}^2, N_0} \\ &\quad + \langle \rho_1(\nabla f_k + \nabla f_k^T + \nabla f_k^T \nabla f_k) \rangle^2_{M_2(\mathbb{R}), N_1}, \\ &= \mathcal{F}_{\epsilon,k}(u_{\epsilon,k}) + 2\mu[a(u_{\epsilon,k}, v) - L(v)] + \mu^2 a(v,v),\end{aligned}$$

which means, using the previous reformulation that :

$$2\mu[a(u_{\epsilon,k}, v) - L(v)] + \mu^2 a(v,v) \geq 0.$$

Assuming now that $\mu > 0$, and dividing the previous inequality by $\mu$, and finally letting $\mu$ tend to 0, we get :

$$a(u_{\epsilon,k}, v) - L(v) \geq 0.$$

Assuming that $\mu < 0$, dividing the inequality by $\mu$ and letting $\mu$ tend to 0, we finally obtain :

$$a(u_{\epsilon,k}, v) - L(v) \leq 0.$$

We then deduce that $a(u_{\epsilon,k}, v) = L(v)$, for all $v \in H^3(\Omega,\mathbb{R}^2)$. The inverse is readily obtained thanks to the previous relations.

## 6   Proof of Lemma 1

*Proof.* Let us take $f \in H^3(\Omega, \mathbb{R}^2)$ such that $\|f\|_{A_0,3,\Omega,\mathbb{R}^2} = 0$. It implies that $|f|_{3,\Omega,\mathbb{R}^2} = 0$ and taking into account the connectedness of $\Omega$, it yields $f \in P^2(\Omega, \mathbb{R}^2)$. We also have that $\rho_1(\nabla f + \nabla f^T) = 0$ meaning that $\frac{\partial f_1}{\partial x}(b_i) = 0$, $\frac{\partial f_2}{\partial y}(b_i) = 0$, $\frac{\partial f_1}{\partial y}(b_i) + \frac{\partial f_2}{\partial x}(b_i) = 0$, for $i = 1, \cdots, N_1$. Since $f \in P^2(\Omega, \mathbb{R}^2)$, it means that $\frac{\partial f_1}{\partial x} \in P^1(\Omega)$, and $\frac{\partial f_2}{\partial y} \in P^1(\Omega)$. Since $A_1$ contains a $P^1(\Omega, \mathbb{R}^2)$-unisolvent subset, it results that $\frac{\partial f_1}{\partial x} \equiv 0$ and $\frac{\partial f_2}{\partial y} \equiv 0$. We therefore have that

$$f_1(x,y) = g_1(y),$$
$$f_2(x,y) = g_1(x).$$

Since, $\frac{\partial f_1}{\partial y} + \frac{\partial f_2}{\partial x} \in P^1(\Omega)$ and $A_1$ contains a $P^1(\Omega, \mathbb{R}^2)$ unisolvent subset, we deduce that

$$\frac{\partial f_1}{\partial y} + \frac{\partial f_2}{\partial x} \equiv 0 \Leftrightarrow \frac{\partial f_1}{\partial y} = g_1'(y) \equiv -\frac{\partial f_2}{\partial x} = -g_2'(x), \forall (x,y) \in \Omega.$$

It thus means that $f_1(x,y) = ay + b$, $f_2(x,y) = -ax + c$ and are in $P^1(\Omega)$ with $f_1(a_i) = f_2(a_i) = 0$ for all $i = 1, \cdots, N_0$. Since $A_0$ contains a $P^1(\Omega, \mathbb{R}^2)$-unisolvent subset, we deduce that $f_1 \equiv f_2 \equiv 0$. It is now clear that $\|.\|_{A_0,3,\Omega,\mathbb{R}^2}$ is a norm on $H^3(\Omega, \mathbb{R}^2)$ associated with a scalar product. We now prove the equivalence of the norm $\|.\|_{A_0,\Omega,3,\mathbb{R}^2}$ with the norm $\|.\|_{3,\Omega,\mathbb{R}^2}$.

First, we have $\forall f \in H^3(\Omega, \mathbb{R}^2)$, $\forall a_i \in A_0$, $i = 1, \cdots, N_0$, $\forall b_j \in A_1$, $j = 1, \cdots, N_1$ :

$\langle f(a_i) \rangle_{\mathbb{R}^2} \leq \|f\|_{\mathcal{C}^0(\Omega, \mathbb{R}^2)} \leq c\|f\|_{3,\Omega,\mathbb{R}^2}$ thanks to Sobolev's embedding,

$\langle \nabla f(b_j) + \nabla f(b_j)^T \rangle_{M_2(\mathbb{R})} \leq 2\|f\|_{\mathcal{C}^1(\Omega, \mathbb{R}^2)} \leq 2c'\|f\|_{3,\Omega,\mathbb{R}^2}$ thanks to Sobolev's embedding,

so, $\|f\|_{A,3,\Omega,\mathbb{R}^2} \leq (1 + c^2 N + 4c'^2 N)^{\frac{1}{2}} \|f\|_{3,\Omega,\mathbb{R}^2}$.

Furthermore, let us take $k = 3 = m$, $p = 2$ in Nečas theorem [6, Chapter 2, section 7.1], and let us take $\rho_0(f)$ and $\rho_1(\nabla f + \nabla f^T)$ as functionals $f_i$ (using the same reasoning as previously to show the property (7.1 bis) $\forall v \in P^2(\Omega, \mathbb{R}^2)$, $\sum_{i=1}^{l} |f_i(v)|^2 = 0 \Leftrightarrow v \equiv 0$). Then there exists a positive constant $c_1 > 0$ such that

$$c_1 \|f\|_{3,\Omega,\mathbb{R}^2} \leq [|f|^2_{3,\Omega,\mathbb{R}^2} + \sum_{i=1}^{N_0} \langle f(a_i) \rangle^2_{\mathbb{R}^2} + \sum_{j=1}^{N_1} \langle \nabla f(b_j) + \nabla f(b_j)^T \rangle^2_{M_2(\mathbb{R})}]^{\frac{1}{2}},$$

which concludes the proof.

# 7  Proof of Theorem 4

*Proof.* The mapping $L$ is a linear and continuous form on $H^3(\Omega, \mathbb{R}^2)$. Indeed,

$$|L(v)| = |\sum_{i=1}^{N_0} \langle f_k(a_i), v(a_i)\rangle_{\mathbb{R}^2} + \sum_{j=1}^{N_1} \frac{\gamma}{2} \langle \nabla f_k(b_j) + \nabla f_k(b_j)^T + \nabla f_k(b_j)^T \nabla f_k(b_j), \nabla v(b_j)$$
$$+ \nabla v(b_j)^T\rangle_{M_2(\mathbb{R})}|,$$
$$\leq \sum_{i=1}^{N_0} |\langle f_k(a_i), v(a_i)\rangle_{\mathbb{R}^2}| + \frac{\gamma}{2}\sum_{j=1}^{N_1} |\langle \nabla f_k(b_j) + \nabla f_k(b_j)^T + \nabla f_k(b_j)^T \nabla f_k(b_j), \nabla v(b_j)$$
$$+ \nabla v(b_j)^T\rangle_{M_2(\mathbb{R})}|,$$
$$\leq \sum_{i=1}^{N_0} |\langle f_k(a_i)\rangle_{\mathbb{R}^2}||\langle v(a_i)\rangle_{\mathbb{R}^2}| + \frac{\gamma}{2}\sum_{j=1}^{N_1} |\langle \nabla f_k(b_j) + \nabla f_k(b_j)^T + \nabla f_k(b_j)^T \nabla f_k(b_j)\rangle_{M_2(\mathbb{R})}|$$
$$|\langle \nabla v(b_j) + \nabla v(b_j)^T\rangle_{M_2(\mathbb{R})}|,$$
$$\leq \sum_{i=1}^{N_0} |\langle f_k(a_i)\rangle_{\mathbb{R}^2}|\|v\|_{\mathcal{C}^0(\Omega,\mathbb{R}^2)} + \gamma\sum_{j=1}^{N_1} |\langle \nabla f_k(b_j) + \nabla f_k(b_j)^T + \nabla f_k(b_j)^T \nabla f_k(b_j)\rangle_{M_2(\mathbb{R})}|$$
$$\|v\|_{\mathcal{C}^1(\Omega,\mathbb{R}^2)},$$
$$\leq (cN_0\|f_k\|_{\mathcal{C}^0(\Omega,\mathbb{R}^2)} + c'N_1\gamma \max_{i=1,\cdots,N_1} \langle \nabla f_k(b_i) + \nabla f_k(b_i)^T + \nabla f_k(b_i)^T \nabla f_k(b_i)\rangle_{M_2(\mathbb{R})})$$
$$\|v\|_{3,\Omega,\mathbb{R}^2}.$$

Moreover, the mapping $a$ is a symmetric, bilinear form, continuous on $H^3(\Omega, \mathbb{R}^2) \times H^3(\Omega, \mathbb{R}^2)$. While symmetry and bilinearity are obvious, the continuity of $a$ can be obtained by using the equivalence of norms established in the previous lemma : $\forall (u,v) \in H^3(\Omega, \mathbb{R}^2) \times H^3(\Omega, \mathbb{R}^2)$,

$$|a(u,v)| = |\sum_{i=1}^{N_0} \langle u(a_i), v(a_i)\rangle_{\mathbb{R}^2} + \frac{\gamma}{2}\sum_{j=1}^{N_1} \langle \nabla u(b_j) + \nabla u(b_j)^T, \nabla v(b_j) + \nabla v(b_j)^T\rangle_{M_2(\mathbb{R})} + \epsilon(u,v)_{3,\Omega,\mathbb{R}^2},$$
$$\leq \max(1, \epsilon, \frac{\gamma}{2})[\sum_{i=1}^{N_0} |\langle u(a_i), v(a_i)\rangle_{\mathbb{R}^2}| + \sum_{j=1}^{N_1} |\langle \nabla u(b_j) + \nabla u(b_j)^T, \nabla v(b_j) + \nabla v(b_j)^T\rangle_{M_2(\mathbb{R})}|$$
$$+ |(u,v)_{3,\Omega,\mathbb{R}^2}|],$$
$$\leq 4\max(1, \epsilon, \frac{\gamma}{2})[\sum_{i=1}^{N_0} \langle u(a_i)\rangle_{\mathbb{R}^2}\langle v(a_i)\rangle_{\mathbb{R}^2} + \sum_{j=1}^{N_1} \langle \nabla u(b_j)\rangle_{M_2(\mathbb{R})}\langle \nabla v(b_j)\rangle_{M_2(\mathbb{R})} + |u|_{3,\Omega,\mathbb{R}^2}|v|_{3,\Omega,\mathbb{R}^2}.$$

This last inequality proves that $a$ is continuous on $H^3(\Omega, \mathbb{R}^2) \times H^3(\Omega, \mathbb{R}^2)$.
To finish with, we prove that $a$ is $H^3(\Omega, \mathbb{R}^3)$-elliptic. Let $v \in H^3(\Omega, \mathbb{R}^2)$,

$$a(v,v) = \langle \rho_0(u)\rangle^2_{\mathbb{R}^2,N_0} + \frac{\gamma}{2}\langle \rho_1(\nabla v + \nabla v^T)\rangle^2_{M_2(\mathbb{R}),N_1} + \epsilon|v|^2_{3,\Omega,\mathbb{R}^2},$$
$$\geq \min(1, \epsilon, \frac{\gamma}{2})\|v\|^2_{A,3,\Omega,\mathbb{R}^2}.$$

Using again the equivalence of norms established above, we deduce that $a$ is $H^3(\Omega, \mathbb{R}^2)$-elliptic. The Lax-Milgram theorem enables us to conclude that the variational problem (3) has a unique solution denoted by $u_{\epsilon,k}$.

## 8 Proof of Lemma 2

*Proof.* The Sobolev's embedding gives that

$$\exists C_1 > 0, \forall d \in D, \forall f \in H^3(\Omega, \mathbb{R}^2), \|f\|_{A_0^d, 3, \Omega, \mathbb{R}^2} \leq C_1 \|f\|_{3, \Omega, \mathbb{R}^2},$$

by using the same arguments as previously, with $C_1$ independent of $d$ and depending only on $\mathcal{N}$. Let us now find a constant $C_2$ independent of $d$ such that the inequality $\|f\|_{3,\Omega,\mathbb{R}^2} \leq C_2 \|f\|_{A_0^d,3,\Omega,\mathbb{R}^2}$ holds. For $f \in H^3(\Omega, \mathbb{R}^2)$, we have

$$\frac{1}{2}\sum_{j=1}^{\mathcal{N}} \langle f(b_{0j}) \rangle_{\mathbb{R}^2}^2 = \frac{1}{2}\sum_{i=1}^{\mathcal{N}} \langle f(b_{0j}) - f(a_{0j}^d) + f(a_{0j}^d) \rangle_{\mathbb{R}^2}^2,$$

$$= \frac{1}{2}\sum_{i=1}^{\mathcal{N}} \langle f(b_{0j}) - f(a_{0j}^d) \rangle_{\mathbb{R}^2}^2 + \langle f(a_{0j}^d) \rangle_{\mathbb{R}^2}^2 + 2\langle f(b_{0j}) - f(a_{0j}^d), f(a_{0j}^d) \rangle_{\mathbb{R}^2},$$

$$\leq \sum_{i=1}^{\mathcal{N}} \langle f(b_{0j}) - f(a_{0j}^d) \rangle_{\mathbb{R}^2}^2 + \langle f(a_{0j}^d) \rangle_{\mathbb{R}^2}^2.$$

The open subset $\Omega$ having a Lipschitz continuous boundary, the space $H^3(\Omega, \mathbb{R}^2)$ verifies the Sobolev's Hölder embedding theorem namely $\exists \lambda \in ]0,1], H^3(\Omega, \mathbb{R}^2) \hookrightarrow \mathcal{C}^{0,\lambda}(\bar{\Omega}, \mathbb{R}^2)$. Thus $f \in \mathcal{C}^{0,\lambda}(\bar{\Omega}, \mathbb{R}^2)$, and $\exists C > 0, \forall j = 1, \cdots, \mathcal{N}, \forall d \in D$,

$$\langle f(b_{0j}) - f(a_{0j}^d) \rangle_{\mathbb{R}^2}^2 \leq \|f\|_{\mathcal{C}^{0,\lambda}(\bar{\Omega},\mathbb{R}^2)}^2 \langle b_{0j} - a_{0j}^d \rangle_{\mathbb{R}^2}^{2\lambda},$$

$$\leq C^2 \|f\|_{H^3(\Omega,\mathbb{R}^2)}^2 \langle b_{0j} - a_{0j}^d \rangle_{\mathbb{R}^2}^{2\lambda}$$

Besides, from the hypotheses, it comes that $\forall j = 1, \cdots, \mathcal{N}$

$$\forall \beta_j > 0, \exists \eta_{\beta_j} > 0, \forall d \in D, (d \leq \eta_{\beta_j} \Rightarrow \langle a_{0j}^d - b_{0j} \rangle_{\mathbb{R}^2} \leq \beta_{\eta_j}).$$

Then $\forall j = 1, \cdots, \mathcal{N}$,

$$\forall \beta_j > 0, \exists \eta_{\beta_j}, \forall d \in D, (d \leq \eta_{\beta_j} \Rightarrow \langle f(b_{0j}) - f(a_{0j}^d) \rangle_{\mathbb{R}^2}^2 \leq C^2 \|f\|_{H^3(\Omega,\mathbb{R}^2)}^2).$$

Let $\beta > 0$ and let us take $\beta_j = \beta$ for all $j = 1, \cdots, \mathcal{N}$, and $\eta = \min\{\eta_{\beta_1}, \cdots, \eta_{\beta_{\mathcal{N}}}\}$, then

$$d \leq \eta \Rightarrow \sum_{j=1}^{\mathcal{N}} \langle f(b_{0j}) - f(a_{0j}^d) \rangle_{\mathbb{R}^2}^2 \leq C^2 \beta^{2\lambda} \mathcal{N} \|f\|_{H^3(\Omega,\mathbb{R}^2)}^2.$$

This implies that $\forall \beta > 0, \exists \eta > 0, \forall d \in D, \forall f \in H^3(\Omega, \mathbb{R}^2)$,

$$d \leq \eta \Rightarrow \sum_{j=1}^{\mathcal{N}} \langle f(b_{0j}) - f(a_{0j}^d) \rangle_{\mathbb{R}^2}^2 \leq C^2 \beta^{2\lambda} \mathcal{N} \|f\|_{H^3(\Omega,\mathbb{R}^2)}^2.$$

This implies that $\forall \beta > 0, \exists \eta > 0, \forall d \in D, \forall f \in H^3(\Omega, \mathbb{R}^2)$,

$$d \leq \eta \Rightarrow \frac{1}{2}\sum_{j=1}^{\mathcal{N}} \langle f(b_{0j})\rangle_{\mathbb{R}^2}^2 + |f|_{3,\Omega,\mathbb{R}^2}^2 - C^2\beta^{2\lambda}\mathcal{N}\|f\|_{H^3(\Omega,\mathbb{R}^2)}^2 \leq \sum_{i=1}^{\mathcal{N}} \langle f(a_{0j}^d)\rangle_{\mathbb{R}^2}^2 + |f|_{3,\Omega,\mathbb{R}^2}^2,$$

$$d \leq \eta \Rightarrow \frac{1}{2}\sum_{j=1}^{\mathcal{N}} \langle f(b_{0j})\rangle_{\mathbb{R}^2}^2 + \sum_{i=1}^{N_1}\langle \nabla f(b_i) + \nabla f(b_i)^T\rangle_{M_2(\mathbb{R})}^2 |f|_{3,\Omega,\mathbb{R}^2}^2 - C^2\beta^{2\lambda}\mathcal{N}\|f\|_{H^3(\Omega,\mathbb{R}^2)}^2$$

$$\leq \sum_{i=1}^{\mathcal{N}} \langle f(a_{0j}^d)\rangle_{\mathbb{R}^2}^2 + \sum_{i=1}^{N_1}\langle \nabla f(b_i) + \nabla f(b_i)^T\rangle_{M_2(\mathbb{R})}^2 + |f|_{3,\Omega,\mathbb{R}^2}^2.$$

As previously shown, the mapping $f \in H^3(\Omega, \mathbb{R}^2) \mapsto (\frac{1}{2}\sum\limits_{j=1}^{\mathcal{N}} \langle f(b_{0j})\rangle_{\mathbb{R}^2}^2 + \sum\limits_{i=1}^{N_1}\langle \nabla f(b_i) +$

$\nabla f(b_i)^T \rangle_{M_2(\mathbb{R})}^2 + |f|_{3,\Omega,\mathbb{R}^2}^2)^{\frac{1}{2}}$ is a norm on $H^3(\Omega, \mathbb{R}^2)$ equivalent to the norm $\|.\|_{3,\Omega,\mathbb{R}^2}$, so $\forall \beta > 0, \exists \eta > 0, \forall d \in D, \forall f \in H^3(\Omega, \mathbb{R}^2)$,

$$d \leq \eta \Rightarrow (C'' - C^2\beta^{2\lambda}\mathcal{N})\|f\|_{H^3(\Omega,\mathbb{R}^2)}^2 \leq \|f\|_{A_0^d, 3, \Omega, \mathbb{R}^2}^2.$$

By choosing $\beta$ adequately, the norm equivalence is obtained.

## 9 Proof of Theorem 5

*Proof.* The proof is divided into 3 steps that we detail hereafter :

- **First step** : We start by proving that the sequence $(u_\epsilon^d)_{d \in D \cap ]0,\eta]}$ is bounded in $H^3(\Omega, \mathbb{R}^2)$ for fixed $N_1$, and $\epsilon$. In the minimisation problem (4), let us take $v = f_k$. Then we have

$$\langle \rho^d(u_\epsilon^d - f_k)\rangle_{\mathbb{R}^2, N(d)}^2 + \frac{\gamma}{2}\langle \rho_1(\nabla u_\epsilon^d + (\nabla u_\epsilon^d)^T - \nabla f_k - \nabla f_k^T - \nabla f_k^T \nabla f_k)\rangle_{M_2(\mathbb{R}),N_1}^2 + \epsilon |u_\epsilon^d|_{3,\Omega,\mathbb{R}^2}^2$$

$$\leq \epsilon |f_k|_{3,\Omega,\mathbb{R}^2}^2 + \frac{\gamma}{2}\sum_{i=1}^{N_1}\langle \nabla f_k(b_i)^T \nabla f_k(b_i)\rangle_{M_2(\mathbb{R})}^2,$$

$$\leq \epsilon |f_k|_{3,\Omega,\mathbb{R}^2}^2 + 2\gamma \sum_{i=1}^{N_1}\langle \nabla f_k(b_i)\rangle_{M_2(\mathbb{R})}^2,$$

$$\leq \epsilon |f_k|_{3,\Omega,\mathbb{R}^2}^2 + 2\gamma N_1 \|f_k\|_{\mathcal{C}^1(\bar{\Omega},\mathbb{R}^2)}^2,$$

from which we deduce that

$$\begin{cases} |u_\epsilon^d|_{3,\Omega,\mathbb{R}^2}^2 \leq |f_k|_{3,\Omega,\mathbb{R}^2}^2 + \frac{2\gamma}{\epsilon}\|f_k\|_{\mathcal{C}^1(\bar{\Omega},\mathbb{R}^2)}^2 N_1, \\ \langle \rho^d(u_\epsilon^d - f_k)\rangle_{\mathbb{R}^2, N}^2 \leq \epsilon|f_k|_{3,\Omega,\mathbb{R}^2}^2 + 2\gamma N_1\|f_k\|_{\mathcal{C}^1(\Omega,\mathbb{R}^2)}^2, \\ \langle \rho_1(\nabla u_\epsilon^d + (\nabla u_\epsilon^d)^T - \nabla f_k - \nabla f_k^T - \nabla f_k^T \nabla f_k)\rangle_{M_2(\mathbb{R}),N_1}^2 \leq \frac{2\epsilon}{\gamma}|f_k|_{3,\Omega,\mathbb{R}^2}^2 + 4N_1\|f_k\|_{\mathcal{C}^1(\bar{\Omega},\mathbb{R}^2)}^2. \end{cases}$$

As $A_0^d \subset A^d$ with $A_0^d$ containing $\mathcal{N}$ elements forming a $P^1$-unisolvent set. So one has

$$\sum_{a \in A_0^d} \langle u_\epsilon^d(a) - f_k(a)\rangle_{\mathbb{R}^2}^2 \leq \sum_{a \in A^d} \langle u_\epsilon^d(a) - f_k(a)\rangle_{\mathbb{R}^2}^2 \leq \epsilon |f_k|_{3,\Omega,\mathbb{R}^2}^2 + 2\gamma N_1\|f_k\|_{\mathcal{C}^1(\Omega,\mathbb{R}^2)}^2.$$

Moreover

$$\sum_{a\in A_0^d} \langle u_\epsilon^d(a)\rangle_{\mathbb{R}^2}^2 = \sum_{a\in A_0^d} \langle u_\epsilon^d(a) - f_k(a) + f_k(a)\rangle_{\mathbb{R}^2}^2,$$

$$\leq 2\sum_{a\in A_0^d} \langle u_\epsilon^d(a) - f_k(a)\rangle_{\mathbb{R}^2}^2 + \langle f_k(a)\rangle_{\mathbb{R}^2}^2,$$

$$\leq 2\epsilon|f_k|_{3,\Omega,\mathbb{R}^2}^2 + 4N_1\gamma\|f_k\|_{\mathcal{C}^1(\bar{\Omega},\mathbb{R}^2)}^2 + 2\mathcal{N}\|f_k\|_{\mathcal{C}^0(\bar{\Omega},\mathbb{R}^2)}^2,$$

$$\sum_{a\in A_1} \langle \nabla u_\epsilon^d(a) + (\nabla u_\epsilon^d(a))^T\rangle_{M_2(\mathbb{R})}^2 = \sum_{a\in A_1} \langle \nabla u_\epsilon^d(a) + \nabla u_\epsilon^d(a)^T - \nabla f_k(a) - \nabla f_k(a)^T$$
$$- \nabla f_k(a)^T \nabla f_k(a) + \nabla f_k(a) + \nabla f_k(a)^T + \nabla f_k(a)^T \nabla f_k(a)\rangle_{M_2(\mathbb{R})}^2,$$

$$\leq 2\sum_{a\in A_1} \langle \nabla u_\epsilon^d(a) + \nabla u_\epsilon^d(a)^T - \nabla f_k(a) - \nabla f_k(a)^T - \nabla f_k(a)^T \nabla f_k(a)\rangle_{M_2(\mathbb{R})}^2$$
$$+ \langle \nabla f_k(a) + \nabla f_k(a)^T + \nabla f_k(a)^T \nabla f_k(a)\rangle_{M_2(\mathbb{R})}^2,$$

$$\leq \frac{2\epsilon}{\gamma}|f_k|_{3,\Omega,\mathbb{R}^2}^2 + 8N_1\|f_k\|_{\mathcal{C}^1(\bar{\Omega},\mathbb{R}^2)}^2 + 32N_1\|f_k\|_{\mathcal{C}^1(\bar{\Omega},\mathbb{R}^2)}^2.$$

Finally, using the equivalence of norm from the previous lemmas, and the previous inequalities, we obtain that

$$\exists \nu > 0, \forall d \in D, d \leq \eta \Rightarrow \|u_\epsilon^d\|_{H^3(\Omega,\mathbb{R}^2)}^2 \leq \nu.$$

The sequence $(u_\epsilon^d)_{d\in D\cap ]0,\eta]}$ is bounded in $H^3(\Omega,\mathbb{R}^2)$ independently of $d$ so one can extract a subsequence $(u_\epsilon^{d_l})_{l\in\mathbb{N}}$ with $\lim_{l\to+\infty} d_l = 0$ (since 0 is an accumulation point of $D$) that weakly converges to an element of $H^3(\Omega,\mathbb{R}^2)$ denoted by $f_k^*$ : $u_\epsilon^{d_l} \underset{l\to+\infty}{\rightharpoonup} f_k^*$ in $H^3(\Omega,\mathbb{R}^2)$.

– **Second step** : In the second step, we prove that $f_k^* = f_k$. Let us assume that $f_k^* \neq f_k$ that is, there exists a non-empty open set $w$ included in $\Omega$ and a positive real $\alpha$ such that

$$\forall x \in w, \langle f_k(x) - f_k^*(x)\rangle_{\mathbb{R}^2} > \alpha.$$

Let us now set $\xi = 1+E\left[\frac{\epsilon|f_k|_{3,\Omega,\mathbb{R}^2}^2 + 2\gamma N_1 \|f_k\|_{\mathcal{C}^1(\Omega,\mathbb{R}^2)}^2}{\alpha^2}\right]$, where $E[.]$ denotes the integer part of the argument. Let $B_0 = \{p_{01}, p_{02}, \cdots, p_{0\xi}\}$ be a subset of $\xi$ distinct points from $w$. One has :

$$\forall i = 1, \cdots, \xi, \exists (p_{0i}^d)_{d\in D}, \forall d \in D, p_{0i}^d \in A^d \text{ and } p_{0i} = \lim_{d\to 0} p_{0i}^d.$$

For any $d \in D$, let $B_0^d$ be the set $\{p_{01}^d, \cdots, p_{0\xi}^d\}$. As previously proved and taking into account that $B_0^d \subset A^d$, we have

$$\sum_{i=1}^\xi u_\epsilon^{d_l}(p_{0i}^{d_l} - f_k(p_{0i}^{d_l}))\rangle_{\mathbb{R}^2}^2 \leq \epsilon|f_k|_{3,\Omega,\mathbb{R}^2}^2 + 2\gamma N_1\|f_k\|_{\mathcal{C}^1(\bar{\Omega},\mathbb{R}^2)}^2.$$

Besides, $\forall i = 1, \cdots, \xi$,

$$\langle u_\epsilon^{d_l}(p_{0i}^{d_l}) - f_k^*(p_{0i})\rangle_{\mathbb{R}^2} = \langle u_\epsilon^{d_l}(p_{0i}^{d_l}) - u_\epsilon^{d_l}(p_{0i}) + u_\epsilon^{d_l}(p_{0i}) - f_k^*(p_{0i})\rangle_{\mathbb{R}^2},$$
$$\leq \langle u_\epsilon^{d_l}(p_{0i}^{d_l}) - u_\epsilon^{d_l}(p_{0i})\rangle_{\mathbb{R}^2} + \langle u_\epsilon^{d_l}(p_{0i}) - f_k^*(p_{0i})\rangle_{\mathbb{R}^2}$$

Since $u_\epsilon^{d_l} \in H^3(\Omega, \mathbb{R}^2) \hookrightarrow \mathcal{C}^{0,\lambda}(\bar{\Omega}, \mathbb{R}^2)$, there exists $C_3 > 0$ such that

$$\langle u_\epsilon^{d_l}(p_{0i}^{d_l}) - u_\epsilon^{d_l}(p_{0i})\rangle_{\mathbb{R}^2} \leq C_3 \langle p_{0i}^{d_l} - p_{0i}\rangle_{\mathbb{R}^2}^\lambda.$$

But $\lim_{l\to+\infty} d_l = 0$, and $p_{0i} = \lim_{l\to+\infty} p_{0i}^{d_l}$, we thus deduce that

$$\lim_{l\to+\infty} \langle u_\epsilon^{d_l}(p_{0i}^{d_l}) - u_\epsilon^{d_l}(p_{0i})\rangle_{\mathbb{R}^2} = 0.$$

The Rellisch-Kondrachov compact embedding theorem gives that $H^3(\Omega, \mathbb{R}^2) \underset{c}{\hookrightarrow} \mathcal{C}^{0,\lambda}(\bar{\Omega}, \mathbb{R}^2)$. Thus, $(u_\epsilon^{d_l})_{l\in\mathbb{N}}$ uniformly converges to $f_k^*$ and then

$$\lim_{l\to+\infty} \langle u_\epsilon^{d_l}(p_{0i}) - f_k^*(p_{0i})\rangle_{\mathbb{R}^2} = 0.$$

Therefore, we can conclude that

$$\lim_{l\to+\infty} \langle u_\epsilon^{d_l}(p_{0i}^{d_l}) - f_k^*(p_{0i})\rangle_{\mathbb{R}^2} = 0, \text{ that is } \lim_{l\to+\infty} u_\epsilon^{d_l}(p_{0i}^{d_l}) = f_k^*(p_{0i}).$$

Letting $l$ tend to infinity in the first inequality, we get

$$\sum_{l=1}^\xi \langle f_k^*(p_{0i}) - f_k(p_{0i})\rangle_{\mathbb{R}^2}^2 \leq \epsilon |f_k|_{3,\Omega,\mathbb{R}^2}^2 + 2\gamma N_1 \|f_k\|_{\mathcal{C}^1(\bar{\Omega},\mathbb{R}^2)}^2,$$

and so

$$\xi \alpha^2 \leq \epsilon |f_k|_{3,\Omega,\mathbb{R}^2}^2 + 2\gamma N_1 \|f_k\|_{\mathcal{C}^1(\bar{\Omega},\mathbb{R}^2)}^2,$$

which is in contradiction with the choice of $\xi$. Then $f_k^* = f_k$.

– **Third step** : In the last step, we come back to the initial displacements in $W^{1,\infty}(\Omega, \mathbb{R}^2)$. We have

$$\lim_{k\to+\infty, l\to+\infty} \|u_\epsilon^{d_l} - u\|_{1,\infty} \leq \lim_{k\to+\infty}(\lim_{l\to+\infty} \|u_\epsilon^{d_l} - f_k\|_{1,\infty} + \|f_k - u\|_{1,\infty}),$$
$$\leq 0,$$

since $\|f_k - u\|_{1,\infty} \underset{k\to+\infty}{\longrightarrow} 0$ by construction and for all $k \in \mathbb{N}$, $\lim_{l\to+\infty} \|u_\epsilon^{d_l} - f_k\|_{1,\infty}$ as $H^3(\Omega, \mathbb{R}^2) \underset{c}{\hookrightarrow} W^{1,\infty}(\Omega, \mathbb{R}^2)$.

## 10  Alternative convergence analysis for the third method

Let us now consider the following problem :

$$\begin{cases} \text{Search for } u_\epsilon^d \in H_0^3(\Omega) \text{ such that } \forall t \in H^3(\Omega, \mathbb{R}^2), \\ \langle \rho_0(u_\epsilon^d - f_k) \rangle_{\mathbb{R}^2, N_0}^2 + \frac{\gamma}{2} \langle \rho^d (\nabla u_\epsilon^d + (\nabla u_\epsilon^d)^T - \nabla f_k - \nabla f_k^T - \nabla f_k^T \nabla f_k) \rangle_{M_2(\mathbb{R}), N}^2 + \epsilon |u_\epsilon^d|_{3,\Omega,\mathbb{R}^2}^2 \\ \leq \langle \rho_0(v - f_k) \rangle_{\mathbb{R}^2, N_0}^2 + \frac{\gamma}{2} \langle \rho^d (\nabla v + \nabla v^T - \nabla f_k - \nabla f_k^T - \nabla f_k^T \nabla f_k) \rangle_{M_2(\mathbb{R}), N}^2 + \epsilon |v|_{3,\Omega,\mathbb{R}^2}^2 \end{cases} \tag{1}$$

In this case, the norm $\|.\|_{A^d, 3, \Omega, \mathbb{R}^2}$ is defined by

$$\|f\|_{A^d, 3, \Omega, \mathbb{R}^2} = [\langle \rho_0(f_k) \rangle_{\mathbb{R}^2, N_0}^2 + \langle \rho^d (\nabla f + \nabla f^T) \rangle_{M_2(\mathbb{R}), N}^2 + |f|_{3,\Omega,\mathbb{R}^2}^2]^{\frac{1}{2}}.$$

As shown in the previous lemma, this norm $\|.\|_{A^d, 3, \Omega, \mathbb{R}^2}$ is equivalent to the norm $\|.\|_{3, \Omega, \mathbb{R}^2}$ on $H^3(\Omega, \mathbb{R}^2)$. We first derive the following lemma.

**Lemma 2.** *Let $B_1 = \{b_{01}, \cdots, b_{0\mathcal{N}}\}$ be a fixed $P^1$-unisolvent subset of $\bar{\Omega}$. By hypothesis, $0 \in \bar{D}$ and $\lim_{d \to 0} \sup_{x \in \Omega} \delta(x, A^d) = 0$ holds so*

$$\forall j = 1, \cdots, \mathcal{N}, \exists (a_{0j}^d)_{d \in D}, \forall d \in D, a_{0j}^d \in A^d \text{ and } b_{0j} = \lim_{d \to 0} a_{0j}^d.$$

*For any $d \in D$, let $A_0^d$ be the set $\{a_{0j}^d, \cdots, a_{0\mathcal{N}}^d\}$ and let $\|.\|_{A_0^d, 3, \Omega, \mathbb{R}^2}$ be the sum defined by $\forall f \in H^3(\Omega, \mathbb{R}^2)$,*

$$\|f\|_{A_0^d, 3, \Omega, \mathbb{R}^2} = \left[ \sum_{j=1}^{N_0} \langle f(a_j) \rangle_{\mathbb{R}^2}^2 + \sum_{j=1}^{\mathcal{N}} \langle \nabla f(a_{0j}^d) + \nabla f(a_{0j}^d)^T \rangle_{M_2(\mathbb{R})}^2 \right]^{\frac{1}{2}}.$$

*Then, there exists $\eta > 0$ such that for any $d \leq \eta$, the norm $\|.\|_{A_0^d, 3, \Omega, \mathbb{R}^2}$ is uniformly equivalent over $D \cap ]0, \eta]$ to the norm $\|.\|_{3, \Omega, \mathbb{R}^2}$ in $H^3(\Omega, \mathbb{R}^2)$.*

*Proof.* The Sobolev's embedding gives that

$$\exists c_1 > 0, \forall d \in D, \forall f \in H^3(\Omega, \mathbb{R}^2), \|f\|_{A_0^d, 3, \Omega, \mathbb{R}^2} \leq C_1 \|f\|_{3, \Omega, \mathbb{R}^2},$$

by using similar arguments as previously. The constant $C_1$ is independent of $d$ and only depends on $N_0$ and $\mathcal{N}$.
Let us now find a constant $C_2$ independent of $d$ such that the inequality

$$\|f\|_{3, \Omega, \mathbb{R}^2} \leq C_2 \|f\|_{A_0^d, 3, \Omega, \mathbb{R}^2},$$

holds. For $f \in H^3(\Omega, \mathbb{R}^2)$,

$$\frac{1}{2} \sum_{j=1}^{\mathcal{N}} \langle \nabla f(b_{0j}) + \nabla f(b_{0j})^T \rangle_{\mathbb{R}^2}^2 = \frac{1}{2} \sum_{j=1}^{\mathcal{N}} \langle \nabla f(b_{0j}) + \nabla f(b_{0j})^T - \nabla f(a_{0j}^d) - \nabla f(a_{0j}^d)^T$$
$$+ \nabla f(a_{0j}^d) + \nabla f(a_{0j}^d)^T \rangle_{M_2(\mathbb{R})}^2,$$
$$\leq \sum_{j=1}^{\mathcal{N}} \langle \nabla f(b_{0j}) + \nabla f(b_{0j})^T - \nabla f(a_{0j}^d) - \nabla f(a_{0j}^d)^T \rangle_{M_2(\mathbb{R})}^2 + \langle \nabla f(a_{0j}^d) + \nabla f(a_{0j}^d)^T \rangle_{M_2(\mathbb{R})}^2.$$

The open subset $\Omega$ having a Lipschitz-continuous boundary and the space $H^3(\Omega, \mathbb{R}^2)$ verifies the Sobolev's Hölder embedding theorem, namely

$$\exists \lambda \in ]0,1], \, H^3(\Omega, \mathbb{R}^2) \underset{c}{\hookrightarrow} \mathcal{C}^{1,\lambda}(\bar{\Omega}, \mathbb{R}^2).$$

Thus $f \in \mathcal{C}^{1,\lambda}(\Omega, \mathbb{R}^2)$ and $\exists C > 0, \, \forall j = 1, \cdots, \mathcal{N}, \, \forall d \in D$,

$$\langle \nabla f(b_{0j}) + \nabla f(b_{0j})^T - \nabla f(a_{0j}^d) - \nabla f(a_{0j}^d)^T \rangle^2_{M_2(\mathbb{R})} \leq 2\|f\|^2_{\mathcal{C}^{1,\lambda}(\Omega,\mathbb{R}^2)} \langle b_{0j} - a_{0j}^d \rangle^{2\lambda}_{\mathbb{R}^2},$$
$$\leq 2C^2 \|f\|^2_{3,\Omega,\mathbb{R}^2} \langle b_{0j} - a_{0j}^d \rangle^{2\lambda}_{\mathbb{R}^2}.$$

Besides, we also have from the assumptions that for all $j = 1, \cdots, \mathcal{N}$,

$$\forall \beta_j > 0, \, \exists \eta_{\beta_j} > 0, \, \forall d \in D, \, (d \leq \eta_{\beta_j} \Rightarrow \langle a_{0j}^d - b_{0j} \rangle_{\mathbb{R}^2} \leq \beta_j).$$

Then $\forall j = 1, \cdots, \mathcal{N}$

$$\forall \beta_j > 0, \, \exists \eta_{\beta_j} > 0, \, \forall d \in D, \, (d \leq \eta_{\beta_j} \Rightarrow \langle \nabla f(b_{0j}) + \nabla f(b_{0j})^T - \nabla f(a_{0j}^d) - \nabla f(a_{0j}^d)^T \rangle^2_{M_2(\mathbb{R})}$$
$$\leq 2C^2 \|f\|^2_{3,\Omega,\mathbb{R}^2} \beta_j^{2\lambda}.$$

Let $\delta > 0$ and let us take $\beta_j = \beta, \, \forall j = 1 \cdots, \mathcal{N}, \, \eta = \min\{\eta_{\beta_1}, \cdots, \eta_{\beta_\mathcal{N}}\}$, then

$$\forall d \in D, \, (d \leq \eta \Rightarrow \sum_{j=1}^{\mathcal{N}} \langle \nabla f(b_{0j}) + \nabla f(b_{0j})^T - \nabla f(a_{0j}^d) - \nabla f(a_{0j}^d)^T \rangle^2_{M_2(\mathbb{R})} \leq \mathcal{N} C^2 \beta^{2\lambda} \|f\|^2_{3,\Omega,\mathbb{R}^2}.$$

Finally,

$\forall \beta > 0, \, \exists \eta > 0, \, \forall d \in D, \, \forall f \in H^3(\Omega, \mathbb{R}^2)$,

$$d \leq \eta \Rightarrow \sum_{j=1}^{\mathcal{N}} \langle \nabla f(b_{0j}) + \nabla f(b_{0j})^T - \nabla f(a_{0j}^d) - \nabla f(a_{0j}^d)^T \rangle^2_{M_2(\mathbb{R})} \leq C^2 \beta^{2\lambda} \mathcal{N} \|f\|^2_{3,\Omega,\mathbb{R}^2}.$$

This implies that

$\forall \beta > 0, \, \exists \eta > 0, \, \forall d \in D, \, \forall f \in H^3(\Omega, \mathbb{R}^2), \, d \leq \eta$

$$\Rightarrow \frac{1}{2} \sum_{j=1}^{\mathcal{N}} \langle \nabla f(b_{0j}) + \nabla f(b_{0j}^T) \rangle^2_{M_2(\mathbb{R})} + \sum_{j=1}^{N_0} \langle f(a_i) \rangle^2_{\mathbb{R}^2} + |f|^2_{3,\Omega,\mathbb{R}^2} - C^2 \beta^{2\lambda} \mathcal{N} \|f\|^2_{3,\Omega,\mathbb{R}^2} \leq \|f\|^2_{A_0^d, 3, \Omega, \mathbb{R}^2}.$$

As previously shown, the mapping $f \in H^3(\Omega, \mathbb{R}^2) \mapsto \left[ \sum_{j=1}^{N_0} \langle f(a_i) \rangle^2_{\mathbb{R}^2} + \frac{1}{2} \sum_{j=1}^{\mathcal{N}} \langle \nabla f(b_{0j}) + \nabla f(b_{0j})^T \rangle^2_{M_2(\mathbb{R})} + |f|^2_{3,\Omega,\mathbb{R}^2} \right]^{\frac{1}{2}}$ is a norm on $H^3(\Omega, \mathbb{R}^2)$ equivalent to the norm $\|.\|_{3,\Omega,\mathbb{R}^2}$ so

$$\forall \beta > 0, \, \exists \eta > 0, \, \forall d \in D, \, \forall f \in H^3(\Omega, \mathbb{R}^2),$$
$$d \leq \eta \Rightarrow (C'^2 - C^2 \beta^{2\lambda} \mathcal{N}) \|f\|^2_{3,\Omega,\mathbb{R}^2} \leq \|f\|^2_{3,\Omega,\mathbb{R}^2}.$$

By choosing $\beta$ adequately, the norm equivalence is obtained.

Before considering a convergence result, let us first introduce the following problem

$$\inf_{u \in H^3(\Omega, \mathbb{R}^2)} |u|_{3,\Omega,\mathbb{R}^2}, \qquad (2)$$

$$\text{s.t. } \nabla u + \nabla u^T = \nabla f_k + \nabla f_k^T + \nabla f_k^T \nabla f_k \text{ everywhere on } \Omega,$$

$$\rho_0(u) = \rho_0(f_k),$$

and prove the existence of minimisers.

**Theorem 5 (Existence of minimisers).** *Let us assume that there exists $u_0 \in H^3(\Omega, \mathbb{R}^2)$ such that*

$$\frac{\partial u_{0,1}}{\partial x} = \frac{\partial f_{k,1}}{\partial x} + (\frac{\partial f_{k,1}}{\partial x})^2 + (\frac{\partial f_{k,2}}{\partial x})^2,$$

$$\frac{\partial u_{0,2}}{\partial y} = \frac{\partial f_{k,2}}{\partial y} + (\frac{\partial f_{k,2}}{\partial y})^2 + (\frac{\partial f_{k,1}}{\partial y})^2,$$

$$\frac{\partial u_{0,1}}{\partial y} + \frac{\partial u_{0,2}}{\partial x} = \frac{\partial f_{k,1}}{\partial y} + \frac{\partial f_{k,2}}{\partial x} + \frac{\partial f_{k,1}}{\partial x}\frac{\partial f_{k,1}}{\partial y} + \frac{\partial f_{k,2}}{\partial x}\frac{\partial f_{k,2}}{\partial y},$$

*everywhere on $\Omega$ and for all $i = 1, \cdots, N_0$, $u_0(a_i) = f_k(a_i)$. (This shouldn't be too restrictive if $N_0$ is small enough). Then there s at least one minimiser of the problem 2.*

*Proof.* Let $(u_j)$ be a minimising sequence of the problem such that $\forall j \in \mathbb{N}$,

$$u_j \in H^3(\Omega, \mathbb{R}^2),$$

$$\nabla u_j + \nabla u_j^T = \nabla f_k + \nabla f_k^T + \nabla f_k^T \nabla f_k, \text{ everywhere on } \Omega,$$

$$\rho(u_j) = \rho(f_k).$$

For $j$ large enough, we have the following coercivity inequality

$|u_j|_{3,\Omega,\mathbb{R}^2} \leq |u_0|_{3,\Omega,\mathbb{R}^2}$, and so

$\langle \rho(\nabla u_j + \nabla u_j^T) \rangle^2_{M_2(\mathbb{R})} + \langle \rho(u_j) \rangle^2_{\mathbb{R}^2} + |u_j|^2_{3,\Omega,\mathbb{R}^2} \leq \langle \rho(\nabla u_0 + \nabla u_0^T) \rangle^2_{M_2(\mathbb{R})} + \langle \rho(u_0) \rangle^2_{\mathbb{R}^2} + |u_0|^2_{3,\Omega,\mathbb{R}^2}.$

Thanks to the previous lemma that gives an equivalence of norms, we get that $(u_j)$ is uniformly bounded according to $j$ in $H^3(\Omega, \mathbb{R}^2)$. We thus can extract a sub-sequence still denoted by $(u_j)$ such that

$$u_j \underset{j \to +\infty}{\rightharpoonup} \bar{u} \text{ in } H^3(\Omega, \mathbb{R}^2),$$

with $\bar{u} \in H^3(\Omega, \mathbb{R}^2)$. Since $H^3(\Omega, \mathbb{R}^2) \underset{c}{\hookrightarrow} \mathcal{C}^1(\bar{\Omega}, \mathbb{R}^2)$, we deduce that $u_k \to \bar{u}$ in $\mathcal{C}^1(\bar{\Omega}, \mathbb{R}^2)$, and therefore

$$u_j \underset{j \to +\infty}{\longrightarrow} \bar{u} \text{ everywhere in } \Omega,$$

$$\nabla u_j \underset{j \to +\infty}{\longrightarrow} \nabla \bar{u} \text{ everywhere in } \Omega,$$

$$\nabla u_j^T \underset{j \to +\infty}{\longrightarrow} \nabla \bar{u}^T \text{ everywhere in } \Omega.$$

Hence we have that

$$\rho(\bar{u}) = \lim_{j \to +\infty} \rho(u_j) = \lim_{j \to +\infty} \rho(f_k) = \rho(f_k),$$

$$\nabla \bar{u}(x) + \nabla \bar{u}^T(x) = \lim_{j \to +\infty} \nabla u_j(x) + \nabla u_j(x)^T = \lim_{j \to +\infty} \nabla f_k(x) + \nabla f_k^T(x) + \nabla f_k^T(x) \nabla f_k(x)$$
$$= \nabla f_k(x) + \nabla f_k^T(x) + \nabla f_k^T(x) \nabla f_k(x), \text{ for every } x \in \Omega.$$

Finally, by the semi-continuity of $|.|_{3,\Omega,\mathbb{R}^2}$, we get that

$$|\bar{u}|_{3,\Omega,\mathbb{R}^2} \leq \liminf_{j \to +\infty} |u_j|_{3,\Omega,\mathbb{R}^2},$$

and $\bar{u}$ is a minimiser of the problem 2.

We now establish the following convergence result

**Theorem 6 (Convergence).** *Let $\hat{f}_k$ be a minimiser of problem 2. For any $d \in D$, and $\epsilon \in ]0, \epsilon_0]$, we denote by $u_{\epsilon,k}^d$ the unique solution of problem 1. Then under the above assumptions, there exists a subsequence $(u_{\epsilon_l,k}^{d_l})$ with $\lim_{l \to +\infty} d_l = \lim_{l \to +\infty} \epsilon_l = 0$, such that*

$$u_{\epsilon_l,k}^{d_l} \xrightarrow[l \to +\infty]{} \hat{f}_k$$

*in $H^3(\Omega, \mathbb{R}^2)$ with $\lim_{k \to +\infty} \lim_{l \to +\infty} \|\nabla u_{\epsilon_l,k}^{d_l} + (\nabla u_{\epsilon_l,k}^{d_l})^T - \nabla u - \nabla u^T - \nabla u^T \nabla u\|_{L^\infty(\Omega, M_2(\mathbb{R}))} = 0$.*

*Proof.* The proof is divided into four steps that we detail hereafter.

– **First step** : We start by proving that the sequence $(u_{\epsilon,k}^d)_{d \in D \cap ]0,\eta], \epsilon \in ]0,\epsilon_0]}$ is bounded in $H^3(\Omega, \mathbb{R}^2)$. In the minimisation problem 1, let us take $v = \hat{f}_k$. Then we have

$$\langle \rho(u_{\epsilon,k}^d - f_k) \rangle_{\mathbb{R}^2, N_0}^2 + \frac{\gamma}{2} \langle \rho^d (\nabla u_{\epsilon,k}^d + (\nabla u_{\epsilon,k}^d)^T - \nabla f_k - \nabla f_k^T - \nabla f_k^T \nabla f_k \rangle_{M_2(\mathbb{R})}^2 + \epsilon |u_{\epsilon,k}^d|_{3,\Omega,\mathbb{R}}^2$$
$$\leq \epsilon |\hat{f}_k|_{3,\Omega,\mathbb{R}}^2,$$

from which we deduce that

$$\begin{cases} |u_{\epsilon,k}^d|_{3,\Omega,\mathbb{R}^2} \leq |\hat{f}_k|_{3,\Omega,\mathbb{R}^2}, \\ \frac{\gamma}{2} \langle \rho^d(\nabla u_{\epsilon,k}^d + (\nabla u_{\epsilon,k}^d)^T - \nabla f_k - \nabla f_k^T - \nabla f_k^T \nabla f_k \rangle_{M_2(\mathbb{R})}^2 \leq \epsilon |\hat{f}_k|_{3,\Omega,\mathbb{R}^2}^2, \\ \langle \rho_0(u_{\epsilon,k}^d - f_k) \rangle_{\mathbb{R}^2}^2 \leq \epsilon |\hat{f}_k|_{3,\Omega,\mathbb{R}}^2 \leq \epsilon_0 |\hat{f}_k|_{3,\Omega,\mathbb{R}}^2. \end{cases}$$

As $A_0^d \subset A^d$, one has

$$\sum_{a \in A_0^d} \langle \nabla u_{\epsilon,k}^d(a) + \nabla u_{k,\epsilon}^d(a)^T - \nabla \hat{f}_k(a) - \nabla \hat{f}_k(a)^T \rangle_{M_2(\mathbb{R})}^2 \leq \epsilon_0 |\hat{f}_k|_{3,\Omega,\mathbb{R}^2}^2.$$

Moreover

$$\sum_{a\in A_0^d} \langle \nabla u_{\epsilon,k}^d(a) + \nabla u_{\epsilon,k}^d(a)^T \rangle_{M_2(\mathbb{R})}^2 \leq 2 \sum_{a\in A_0^d} \langle \nabla u_{\epsilon,k}^d(a) + \nabla u_{\epsilon,k}^d(a)^T - \nabla \hat{f}_k(a) - \nabla \hat{f}_k(a)^T \rangle_{M_2(\mathbb{R})}^2$$
$$+ 2 \sum_{a\in A_0^d} \langle \nabla \hat{f}_k(a) + \nabla \hat{f}_k(a)^T \rangle_{M_2(\mathbb{R})}^2,$$
$$\leq 2\epsilon_0 |\hat{f}_k|_{3,\Omega,\mathbb{R}^2}^2 + 4\mathcal{N} \|\hat{f}_k\|_{\mathcal{C}^1(\Omega,\mathbb{R}^2)}^2.$$

Finally, using the equivalence of norm previously established, and the inequalities above, we obtain that

$$\exists \nu > 0,\ \forall d \in D,\ \forall \epsilon \in ]0,\epsilon_0],\ d \leq \eta \Rightarrow \|u_k^d\|_{3,\Omega,\mathbb{R}^2} \leq \nu.$$

The sequence $(u_{\epsilon,k}^d)_{d\in D\cap ]0,\eta],\epsilon\in ]0,\epsilon_0]}$ is bounded in $H^3(\Omega,\mathbb{R}^2)$ so one can extract a subsequence $(u_{\epsilon_l,k}^{d_l})$ with $\lim_{l\to+\infty} d_l = \lim_{l\to+\infty} \epsilon_l = 0$ that weakly converges to $f_k^*$ in $H^3(\Omega,\mathbb{R}^2)$.

- **Second step**: We now prove that $f_k^* = \hat{f}_k$. Let us assume that $\nabla f_k^* + (\nabla f_k^*)^T \neq \nabla \hat{f}_k + \nabla \hat{f}_k^T$, that is, there exists a non-empty open set $\omega$ included in $\Omega$ and a positive real $\alpha$ such that

$$\forall x \in w,\ \langle \nabla f_k^*(x) + (\nabla f_k^*)^T(x) - \nabla \hat{f}_k(x) - \nabla \hat{f}_k^T(x) \rangle_{M_2(\mathbb{R})}^2 > \alpha.$$

Let us set $\xi = E\left[\frac{\epsilon |\hat{f}_k|_{3,\Omega,\mathbb{R}^2}^2}{\alpha^2}\right]$, where $E[.]$ denotes te integer part of the argument. Let $B_0 = \{p_{01},\cdots,p_{0\xi}\}$ be a subset of $\xi$ distinct points from $\omega$. One has

$$\forall i = 1,\cdots,\xi,\ \exists (p_{0i}^d)_{d\in D},\ (\forall d \in D,\ p_{0i} \in A) \text{ and } p_{0i} = \lim_{d\to 0} p_{0i}^d.$$

For any $d \in D$, let $B_0^d$ be the set $\{p_{01}^d,\cdots,p_{0\xi}^d\}$. As previously proved and taking into account that $B_0^d \subset A^d$, we have

$$\sum_{i=1}^{\xi} \langle \nabla u_{\epsilon_l,k}^{d_l}(p_{0i}^{d_l}) + \nabla u_{\epsilon_l,k}^{d_l}(p_{0i}^{d_l})^T - \nabla \hat{f}_k(p_{0i}^{d_l}) - \nabla \hat{f}_k(p_{0i}^{d_l})^T \rangle_{M_2(\mathbb{R})}^2 \leq \epsilon |\hat{f}_k|_{3,\Omega,\mathbb{R}^2}^2.$$

Besides, for all $i = 1,\cdots,\xi$,

$$\langle \nabla u_{\epsilon_l,k}^{d_l}(p_{0i}^{d_l}) + \nabla u_{\epsilon_l,k}^{d_l}(p_{0i}^{d_l}) - \nabla f_k^*(p_{0i}) - \nabla f_k^*(p_{0i})^T \rangle_{M_2(\mathbb{R})}$$
$$= \langle \nabla u_{\epsilon_l,k}^{d_l}(p_{0i}^{d_l}) + \nabla u_{\epsilon_l,k}^{d_l}(p_{0i}^{d_l}) - \nabla u_{\epsilon_l,k}^{d_l}(p_{0i}) - \nabla u_{\epsilon_l,k}^{d_l}(p_{0i})^T + \nabla u_{\epsilon_l,k}^{d_l}(p_{0i}) + \nabla u_{\epsilon_l,k}^{d_l}(p_{0i})^T$$
$$- \nabla f_k^*(p_{0i}) - \nabla f_k^*(p_{0i})^T \rangle_{M_2(\mathbb{R})},$$
$$\leq \langle \nabla u_{\epsilon_l,k}^{d_l}(p_{0i}^{d_l}) + \nabla u_{\epsilon_l,k}^{d_l}(p_{0i}^{d_l}) - \nabla u_{\epsilon_l,k}^{d_l}(p_{0i}) - \nabla u_{\epsilon_l,k}^{d_l}(p_{0i})^T \rangle_{M_2(\mathbb{R})} + \langle \nabla u_{\epsilon_l,k}^{d_l}(p_{0i})$$
$$+ \nabla u_{\epsilon_l,k}^{d_l}(p_{0i})^T - \nabla f_k^*(p_{0i}) - \nabla f_k^*(p_{0i})^T \rangle_{M_2(\mathbb{R})}.$$

But $u_{\epsilon_l,k}^{d_l} \in H^3(\Omega,\mathbb{R}^2) \hookrightarrow_c \mathcal{C}^{1,\lambda}(\Omega,\mathbb{R}^2)$ for $\lambda \in [0,1]$ and so there exists $C_3 > 0$ such that

$$\langle \nabla u_{\epsilon_l,k}^{d_l}(p_{0i}^{d_l}) + \nabla u_{\epsilon_l,k}^{d_l}(p_{0i}^{d_l})^T - \nabla u_{\epsilon_l,k}^{d_l}(p_{0i}) - \nabla u_{\epsilon_l,k}^{d_l}(p_{0i})^T \rangle_{M_2(\mathbb{R})} \leq 2C_3 \langle p_{0i}^{d_l} - p_{0i} \rangle_{\mathbb{R}^2}^\lambda.$$

But $\lim_{l \to +\infty} d_l = 0$ and $p_{0i} = \lim_{l \to +\infty} p_{0i}^{d_l}$, and we deduce that

$$\lim_{l \to +\infty} \langle \nabla u_{\epsilon_l,k}^{d_l}(p_{0i}^{d_l}) + \nabla u_{\epsilon_l,k}^{d_l}(p_{0i}^{d_l})^T - \nabla u_{\epsilon_l,k}^{d_l}(p_{0i}) - \nabla u_{\epsilon_l,k}^{d_l}(p_{0i})^T \rangle_{M_2(\mathbb{R})} = 0.$$

The Rellisch-Kondrachov compact embedding theorem gives that $H^3(\Omega,\mathbb{R}^2) \hookrightarrow_c \mathcal{C}^1(\Omega,\mathbb{R}^2)$, and thus

$$\lim_{l \to +\infty} \langle \nabla u_{\epsilon_l,k}^{d_l}(p_{0i}) + \nabla u_{\epsilon_l,k}^{d_l}(p_{0i})^T \nabla f_k^*(p_{0i}) - \nabla f_k^*(p_{0i})^T \rangle_{M_2(\mathbb{R})} = 0.$$

Therefore, we can conclude that

$$\lim_{l \to +\infty} \langle \nabla u_{\epsilon_l,k}^{d_l}(p_{0i}^{d_l}) + \nabla u_{\epsilon_l,k}^{d_l}(p_{0i}^{d_l})^T - \nabla f_k^*(p_{0i}) - \nabla f_k^*(p_{0i})^T \rangle_{M_2(\mathbb{R})} = 0,$$

and so $\lim_{l \to +\infty} \nabla u_{\epsilon_l,k}^{d_l}(p_{0i}^{d_l}) + \nabla u_{\epsilon_l,k}^{d_l}(p_{0i}^{d_l}) = \nabla f_k^*(p_{0i}) + \nabla f_k^*(p_{0i})^T$. Letting $l$ tend to infinity, it comes

$$\sum_{i=1}^{\xi} \langle \nabla f_k^*(p_{0i}) + \nabla f_k^*(p_{0i})^T - \nabla \hat{f}_k(p_{0i}) - \nabla \hat{f}_k(p_{0i})^T \rangle_{M_2(\mathbb{R})}^2 \leq \epsilon |\hat{f}_k|_{3,\Omega,\mathbb{R}^2}^2,$$

and so $\alpha^2 \xi \leq \epsilon |\hat{f}_k|_{3,\Omega,\mathbb{R}^2}^2$ which is in contradiction with the definition of $\xi$. Consequently, $\nabla f_k^* + (\nabla f_k^*)^T = \nabla \hat{f}_k + \nabla \hat{f}_k^T = \nabla f_k + \nabla f_k^T + \nabla f_k^T \nabla f_k$ everywhere on $\Omega$. Also, since $H^3(\Omega,\mathbb{R}^2) \hookrightarrow_c \mathcal{C}^0(\bar{\Omega},\mathbb{R}^2)$, we have that $\lim_{l \to +\infty} \langle \rho_0(u_{\epsilon_l,k}^{d_l} - f_k) \rangle_{\mathbb{R}^2,N_0}^2 = \langle \rho_0(f_k^* - f_k) \rangle_{\mathbb{R}^2,N_0}^2$, and by letting $l$ tend to infinity in the last inequality we get that

$$\langle \rho_0(f_k^* - f_k) \rangle_{\mathbb{R}^2,N_0}^2 \leq \lim_{l \to +\infty} \epsilon_l |\hat{f}_k|_{3,\Omega,\mathbb{R}^2}^2 = 0.$$

Finally, by weak lower semicontinuity of the seminorm in $H^3(\Omega,\mathbb{R}^2)$, we get that

$$|f_k^*|_{3,\Omega,\mathbb{R}^2} \leq \liminf_{l \to +\infty} |u_{\epsilon_l,k}^{d_l}|_{3,\Omega,\mathbb{R}^2} \leq |\hat{f}_k|_{3,\Omega,\mathbb{R}^2}.$$

Thus $f_k^*$ is a minimiser of problem 2. Without loss of generality, we say $f_k^* = \hat{f}_k$.

– **Third step**: We then have

$$\|\nabla u_{\epsilon_l,k}^{d_l} + (\nabla u_{\epsilon_l,k}^{d_l})^T - \nabla u - \nabla u^T - \nabla u^T \nabla u\|_{L^\infty(\Omega,M_2(\mathbb{R}))}$$
$$\leq \|\nabla u_{\epsilon_l,k}^{d_l} + (\nabla u_{\epsilon_l,k}^{d_l})^T - \nabla \hat{f}_k - \nabla \hat{f}_k^T\|_{L^\infty(\Omega,M_2(\mathbb{R}))} + \|\nabla \hat{f}_k + \nabla \hat{f}_k^T - \nabla u - \nabla u^T$$
$$- \nabla u^T \nabla u\|_{L^\infty(\Omega,M_2(\mathbb{R}))},$$
$$\leq \|\nabla u_{\epsilon_l,k}^{d_l} + (\nabla u_{\epsilon_l,k}^{d_l})^T - \nabla \hat{f}_k - \nabla \hat{f}_k^T\|_{L^\infty(\Omega,M_2(\mathbb{R}))} + \|\nabla f_k + \nabla f_k^T + \nabla f_k^T \nabla f_k - \nabla u$$
$$- \nabla u^T - \nabla u^T \nabla u\|_{L^\infty(\Omega,M_2(\mathbb{R}))},$$

with $\lim_{k\to+\infty}\lim_{l\to+\infty}\|\nabla u_{\epsilon_l,k}^{d_l}+(\nabla u_{\epsilon_l,k}^{d_l})^T-\nabla \hat{f}_k-\nabla \hat{f}_k^T\|_{L^\infty(\Omega,M_2(\mathbb{R}))}=0$ from what precedes and the Sobolev embedding $H^3(\Omega,\mathbb{R}^2)\hookrightarrow_c W^{1,\infty}(\Omega,\mathbb{R}^2)$, and $\lim_{k\to+\infty}\lim_{l\to+\infty}\|\nabla f_k+\nabla f_k^T+\nabla f_k^T\nabla f_k-\nabla u-\nabla u^T-\nabla u^T\nabla u\|_{L^\infty(\Omega,M_2(\mathbb{R}))}=0$, by construction of the sequence $f_k$. We thus have $\lim_{k\to+\infty}\lim_{l\to+\infty}\|\nabla u_{\epsilon_l,k}^{d_l}+(\nabla u_{\epsilon_l,k}^{d_l})^T-\nabla u-\nabla u^T-\nabla u^T\nabla u\|_{L^\infty(\Omega,M_2(\mathbb{R}))}=0$.

– **Fourth step** : We aim at proving that the sequence $(u_{\epsilon_l,k}^{d_l})$ strongly converges to $\hat{f}_k$ in $H^3(\Omega,\mathbb{R}^2)$. The Rellisch-Kondrachov compact embedding theorem gives that

$$\forall r,r'\in\mathbb{R},\ r>r',\ H^r(\Omega,\mathbb{R}^2)\hookrightarrow_c H^{r'}(\Omega,\mathbb{R}^2).$$

In our case, it means that the sequence $(u_{\epsilon_l,k}^{d_l})$ that weakly converges to $\hat{f}_k$ in $H^3(\Omega,\mathbb{R}^2)$, strongly converges to $\hat{f}_k$ in $H^2(\Omega,\mathbb{R}^2)$. We thus just need to prove that

$$\lim_{l\to+\infty}|u_{\epsilon_l,k}^{d_l}-\hat{f}_k|_{3,\Omega,\mathbb{R}^2}=0.$$

One has

$$|u_{\epsilon_l,k}^{d_l}-\hat{f}_k|_{3,\Omega,\mathbb{R}^2}=|u_{\epsilon_l,k}^{d_l}|_{3,\Omega,\mathbb{R}^2}+|\hat{f}_k|_{3,\Omega,\mathbb{R}^2}-2(u_{\epsilon_l,k}^{d_l},\hat{f}_k)_{3,\Omega,\mathbb{R}^2},$$
$$\leq 2|\hat{f}_k|_{3,\Omega,\mathbb{R}^2}-2(u_{\epsilon_l,k}^{d_l},\hat{f}_k)_{3,\Omega,\mathbb{R}^2},$$

so $\lim_{l\to+\infty}|u_{\epsilon_l,k}^{d_l}-\hat{f}_k|_{3,\Omega,\mathbb{R}^2}=0$, and finally $\lim_{l\to+\infty}\|u_{\epsilon_l,k}^{d_l}-\hat{f}_k\|_{3,\Omega,\mathbb{R}^2}=0$ which concludes the proof.


\end{document}